\documentclass{article}
\usepackage[utf8]{inputenc}
\usepackage{amssymb,amsmath, amsthm, mathabx,mathtools}
\usepackage{amsmath,amsfonts,amssymb,amsthm,epsfig,epstopdf,titling,url,array}
\usepackage{tikz}
\usetikzlibrary{shapes,snakes}
\usepackage{color, enumerate}
\usepackage{comment, authblk}
\usepackage{hyperref}
\usepackage{pgfplots}
\usepackage{bm}
\usepackage{stmaryrd}
\usepackage{graphicx}
\usepackage{caption}
\usepackage{subcaption}
\usepackage{titlesec}
\usepackage{epstopdf}

\usepackage{sectsty}
\usepackage{lipsum}
\usepackage{algorithm}
\usepackage[noend]{algpseudocode}
\usepackage{algpseudocode}
\usepackage{pifont}

\titleformat*{\subsection}{\large\bfseries}
\numberwithin{equation}{section}

\pgfplotsset{compat=newest}
\pgfplotsset{plot coordinates/math parser=false}
\newlength\figureheight
\newlength\figurewidth

\DeclareMathOperator{\G}{G}

\usepackage{enumerate}
\usepackage{mathrsfs}
\usepackage{wrapfig}

\usepackage{color}

\numberwithin{equation}{section}

\newcommand{\mm}{{\underline m}}

\newcommand{\beq}{\begin{equation}}
\newcommand{\bEq}{\end{equation}}

\newcommand{\bx}{{\bf{x}}}

\newcommand{\al}{\alpha}

\newcommand{\be}{\begin{equation}}
\newcommand{\ee}{\end{equation}}

\newcommand{\e}{{\varepsilon}}

\newcommand{\fa}{{\mathfrak a}}
\newcommand{\fb}{{\mathfrak b}}

\usepackage{amsmath} 
\usepackage{amssymb}
\usepackage{amsthm}

\renewcommand{\cal}{\mathcal}

\newcommand{\wh}{\widehat}
\newcommand{\wt}{\widetilde}

\newcommand{\ii}{\mathrm{i}} 
\newcommand{\dd}{\mathrm{d}}

\renewcommand{\epsilon}{\varepsilon}
\renewcommand{\leq}{\leqslant}
\renewcommand{\geq}{\geqslant}

\renewcommand{\le}{\leq}
\renewcommand{\ge}{\geq}

\newcommand{\E}{\mathbb{E}}
\newcommand{\R}{\mathbb{R}}
\newcommand{\C}{\mathbb{C}}
\newcommand{\N}{\mathbb{N}}

\newcommand{\norm}[1]{\lVert #1 \rVert}

\DeclareMathOperator{\diag}{diag}
\DeclareMathOperator{\tr}{Tr}

\DeclareMathOperator{\supp}{supp}

\DeclareMathOperator{\re}{Re}
\DeclareMathOperator{\im}{Im}

\DeclareMathOperator{\OO}{O}
\DeclareMathOperator{\oo}{o}

\DeclareMathOperator{\bA}{\mathbf{A}}
\DeclareMathOperator{\bC}{\mathbf{C}}

\DeclareMathOperator{\bv}{\mathbf{v}}
\DeclareMathOperator{\bu}{\mathbf{u}}

\DeclareMathOperator{\bbE}{\mathbb{E}}
\DeclareMathOperator{\bbN}{\mathbb{N}}
\DeclareMathOperator{\bbP}{\mathbb{P}}

\DeclareMathOperator{\sI}{\mathcal{I}}

\DeclareMathOperator{\sW}{\mathcal{W}}

\theoremstyle{plain} 
\newtheorem{theorem}{Theorem}[section]
\newtheorem*{theorem*}{Theorem}
\newtheorem{lemma}[theorem]{Lemma}
\newtheorem{assumption}[theorem]{Assumption}
\newtheorem*{lemma*}{Lemma}

\newtheorem*{corollary*}{Corollary}
\newtheorem{proposition}[theorem]{Proposition}
\newtheorem*{proposition*}{Proposition}

\newtheorem{definition}[theorem]{Definition}
\newtheorem*{definition*}{Definition}
\theoremstyle{remark}

\newtheorem*{example*}{Example}
\newtheorem{remark}[theorem]{Remark}

\newtheorem*{remark*}{Remark}
\newtheorem*{remarks*}{Remarks}

\renewcommand{\Im}{{\rm{Im}}}
\renewcommand{\Re}{{\rm{Re}}}

\newcommand{\uQ}{{\underline{\mathcal Q}}}
\newcommand{\urho}{{\underline{\rho}}}
\newcommand{\um}{\underline{m}}
\newcommand{\uG}{\underline{\mathcal{G}}}
\newcommand{\uR}{\underline{\mathcal{R}}}

\newcommand{\nc}{\normalcolor}

\makeatletter
\def\env@dmatrix{\hskip -\arraycolsep
  \let\@ifnextchar\new@ifnextchar
  \extrarowheight=2ex
  \array{*\c@MaxMatrixCols{>{\displaystyle}c}}}

\newenvironment{bdmatrix}
  {\left[\env@dmatrix}
  {\endmatrix\right]}

\date{}

\allowdisplaybreaks

\title{Edge statistics of large dimensional deformed rectangular matrices}
\author[1,2]{Xiucai Ding \thanks{E-mail: xiucai.ding@duke.edu.}}
\author[3]{Fan Yang  \thanks{E-mail: fyang75@wharton.upenn.edu}}

\affil[1]{Department of Statistics, University of California, Davis}
\affil[2]{Department of Mathematics, Duke University}

\affil[3]{Department of Statistics, University of Pennsylvania}

\begin{document}

\maketitle
\begin{abstract}
We consider the edge statistics of large dimensional deformed rectangular matrices of the form $Y_t=Y+\sqrt{t}X,$ where $Y$ is a $p \times n$ deterministic signal matrix whose rank is comparable to $n$, $X$ is a $p\times n$ random noise matrix with centered i.i.d. entries with variance $n^{-1}$, and $t>0$ gives the noise level. 
This model is referred to as the interference-plus-noise matrix in the study of massive multiple-input multiple-output (MIMO) system, which belongs to the category of the so-called signal-plus-noise model. For the case $t=1$, the spectral statistics of this model have been studied to a certain extent in the literature \cite{DOZIER20071099,DOZIER2007678,VLM2012}. 
In this paper, we study the singular value and singular vector statistics of $Y_t$ around the right-most edge of the singular value spectrum in the harder regime $n^{-2/3}\ll t \ll 1$. 
This regime is harder than the $t=1$ case, because on one hand, the edge behavior of the empirical spectral distribution (ESD) of $YY^\top$ has a strong effect on the edge statistics of $Y_tY_t^\top$ since $t\ll 1$ is ``small", while on the other hand, the edge statistics of $Y_t$ is also not merely a perturbation of those of $Y$ since $t\gg n^{-2/3}$ is ``large".  
Under certain regularity assumptions on $Y,$ we prove the edge universality, eigenvalues rigidity and eigenvector delocalization for the matrices $Y_tY_t^\top$ and $Y_t^\top Y_t$. These results can be used to estimate and infer the massive MIMO system. To prove the main results, we analyze the edge behavior of the asymptotic ESD for  $Y_tY_t^\top$, and establish some sharp local laws on the resolvent of $Y_tY_t^\top$. These results can be of independent interest, and used as useful inputs for many other problems regarding the spectral statistics of $Y_t$.
\end{abstract}
\section{Introduction}

Large dimensional signal-plus-noise matrices are common objects in many scientific fields, such as signal processing \cite{AN17,212753}, image denoising \cite{NEJATI201628,8305076}, wireless communications \cite{VLM2012,VBL} and biology \cite{feng2009,YTCNX}. In these applications, researchers are interested in the estimation and inference of some deterministic matrix, known as the signal matrix, from its noisy observation. Specifically, we consider matrices of the form
\begin{equation}\label{eq_defnmodel}
Y_t=Y+\sqrt{t} X,
\end{equation} 
where $Y$ is a $p \times n$ deterministic signal matrix, $X$ is a white noise matrix whose entries $x_{ij}$ are i.i.d. random variables with mean zero and variance $n^{-1}$, and $t>0$ represents the noise level. In this paper, we consider the high dimensional setting where $p$ is comparable to $n.$ 

There have been a lot of theoretical studies of this model in the literature
by imposing various structural assumptions on $Y.$ Among them, the most popular one is perhaps the low rank structure assumption \cite{AN17,BDW,BENAYCHGEORGES2012120,MR3837103,ding2020,gavish-donoho-2017,RRN14, YMB}. Towards this direction, it is assumed that $Y$ is a low-rank deterministic or random matrix, and admits a singular value decomposition (SVD) 
\begin{equation}\label{eq_svdfory}
Y=\sum_{i=1}^r \sqrt{d_i} \bm{u}_i \bm{v}_i^\top,
\end{equation} 
where $\{\sqrt{d_i}\}$ are the singular values, and $\{\bm{u}_i\}$ and $\{\bm{v}_i\}$ are the left and right singular vectors, respectively. In the low-rank setting, $r$ is a fixed integer that does not change with $n$. This low-rank assumption is popular in many applications, including signal processing \cite{AN17, 212753}, imaging denoising \cite{NEJATI201628,8305076} and statistical genetics \cite{feng2009,YTCNX}. Based on it, many statistical methods have been proposed to estimate $Y$ from the noisy observation $Y_t$: the shrinkage estimation \cite{gavish-donoho-2017,RRN14}, the iterative thresholding procedure \cite{ding2020, YMB}, and the regularization methods \cite{JW, RSV,Yuetal14}, to name but a few.


Although the low-rank assumption is useful in many applications, it is not always feasible, especially in the applications driven by wireless communications and the massive or multicell multiuser MIMO systems \cite{BHKD,NKDA,VBL}, such as the subspace estimation \cite{VLM2012} and direction of arrival (DOA) estimation \cite{5502034}. In these applications, $Y$ is a large rank  interference matrix, where the rank $r$ in (\ref{eq_svdfory}) can be comparable to $n$, 
and $Y_t$ is called an \emph{interference-plus-noise matrix} \cite{VBL,7345553}; see Section \ref{sec_statapp} below for a more detailed discussion. 
Moreover, in modern statistical learning theory, a large-rank matrix $Y$ can provide deep insights into many optimization techniques. For example, it is necessary to take $r$ to be proportional to $n$ in order to obtain a minimax estimator on $Y$ using nuclear norm penalization and singular value thresholding \cite{donoho2014}. Furthermore, it is empirically observed that the mean square error of the minimax estimator of a large-rank $Y$ is closely related to the phase transition phenomenon in matrix completion \cite{Donoho8405}. Motivated by the above applications, 
it is natural to extend the low-rank assumption and study the signal-plus-noise model (\ref{eq_defnmodel}) for  a large-rank signal matrix $Y$.

%

From the perspective of Random Matrix Theory, the signal-plus-noise model (\ref{eq_defnmodel}) falls into the category of the so-called deformed random matrix models \cite{capfreeprob}, some of which haven been studied  in the literature, including the deformed Wigner matrices \cite{KY, knowles2014, LS2015}, deformed sample covariance matrices \cite{Anisotropic}, and separable covariance matrices \cite{dingyang2,yang20190}. In this paper, we call $Y_t$ a \emph{deformed rectangular matrix}. Under the low-rank assumption that $r$ is fixed, the  empirical spectral distribution (ESD) of $Y_tY_t^\top$ is mostly determined by the noise matrix $X$, while the signal matrix $Y$ will give rise to several outlier singular values (i.e., singular values that are detached from the bulk singular value spectrum) depending on the values of $d_i$'s \cite{BENAYCHGEORGES2012120, ding2020}. 
On the other hand, when $Y$ is a large-rank matrix, the ESD  of $Y_tY_t^\top$ will be governed by both the ESD of $YY^\top$ and the noise matrix $X$. 
One purpose of this paper is to extend some of the known results in the low-rank setting to large-rank deformed rectangular matrices, which in turn will provide useful insights into the applications of the interference-plus-noise matrix model \eqref{eq_defnmodel}.

Motivated by the successful applications of edge statistics in high-dimensional statistics, we shall focus on the eigenvalue and eigenvector statistics of $\cal Q_t:=Y_tY_t^\top$ and $\uQ_t:=Y_t^\top Y_t$ (or equivalently, the singular values and singular vectors of $Y_t$) near the right-most edge of the eigenvalue spectrum. 
Furthermore, we are interested in the regime $t\ll 1$, that is, $t\le n^{-\e}$ for some constant $\e>0$. In fact, the case $t\sim 1$ has been studied to certain extent in \cite{DOZIER20071099,DOZIER2007678,VLM2012}. Combining the results there with the arguments of this paper, one can reproduce all the main results of this paper. We remark that there is an important difference between the $t\sim 1$ case and the $t=\oo(1)$ case considered in this paper. In the case $t\sim1$, due to the large component $\sqrt{t}X$, the asymptotic ESD of $\cal Q_t$ will be regular and has a square root behavior around the right-most edge regardless of the edge behavior of $YY^\top$ to some extent. On the other hand, in the $t=\oo(1)$ case, the edge behavior of $YY^\top$ will have a strong effect on the asymptotic ESD of $\cal Q_t$, since $Y_t$ can be regarded as a perturbation of $Y$. Consequently, we need more assumptions on $Y$. In this paper, we shall follow the notion in \cite{edgedbm} for deformed Wigner matrices, and impose an $\eta_*$-regular condition on $Y$ (c.f. Definition \ref{assumption_edgebehavior}) for some scale parameter $0<\eta_* \ll 1$. This regularity assumption ensures a regular square root behavior of the ESD of $YY^\top$ around its right-most edge on the scale $\eta_*$. Moreover, we shall take $t$ such that $\eta_* \ll t^2 \ll 1$. Intuitively, $t=\sqrt{\eta_*}$ is the threshold where the noise component $\sqrt{t}X$ starts to dominate the behavior of the ESD of $\mathcal{Q}_t$. That is to say, when $t \ll \sqrt{\eta_*}$, the ESD of $Y$ around the right-most edge can be well approximated by that of $Y_t$ on scales that are  larger than $\eta_*$.

 We mention that in free probability theory \cite{capfreeprob}, the asymptotic ESD of $\mathcal{Q}_t$ is called \emph{the rectangular free convolution} with the Marchenko-Pastur (MP) law. In \cite{DOZIER20071099,DOZIER2007678,VLM2012}, it has been shown that the rectangular free convolution has a regular square root behavior near the right-most edge of the spectrum in the $t \sim 1$ case. However, the estimates  there diverge as $t\to 0$, and hence are not strong enough for our setting with $t=\oo(1)$. In this paper, we first establish some deterministic estimates of the rectangular free convolution based on a sophisticated analysis of the subordination function in (\ref{eq_defnzeta}). 
 In particular, we will show that for  $t \gg \sqrt{\eta_*}$, the rectangular free convolution still has a regular square root behavior near the right-most edge. 
Based on these estimates, we are able to prove sharp local laws on the resolvents of $\mathcal{Q}_t$ and $\uQ_t$. In the proof, we first establish the local laws for the so-called \emph{rectangular matrix Dyson Brownian motion}, which is a special case of $Y_t$ with a Gaussian random matrix $X$. Together with a self-consistent comparison argument, we can extend the local laws to deformed rectangular matrices with generally distributed $X$, assuming only certain moment conditions. Once we have obtained the local laws, we can prove some important results regarding the eigenvalue and eigenvector statistics of $\cal Q_t$ and $\uQ_t$, including the edge universality, eigenvalues rigidity and eigenvector delocalization.
Although we will not give the detailed proof in this paper, we believe that all the above results hold true even when $t \sim 1$, and the proof in fact will be much easier.   

Finally, we remark that the results of this paper can be a key input for many other problems regarding the spectral statistics of large-rank deformed rectangular matrices. For instance, in \cite{DY20201} we prove the Tracy-Widom fluctuation for the edge eigenvalues of a general class of Gram type random matrices, using the estimates of the rectangular free convolution and the local laws proved in this paper as key technical inputs. Furthermore, our results can be used to study the outlier eigenvalues and eigenvectors when $Y_t$ is perturbed by another low rank matrix, say $Y_0$. This kind of matrix model 
$Y_0+Y+\sqrt{t}X$
will be useful in the estimation and inference of the massive MIMO system; see Section \ref{sec_statapp} below for a more detailed discussion. 

The rest of the paper is organized as follows. In Section \ref{sec_mainresult}, we state our main results regarding the basic properties of the rectangular free convolution, the local laws, and the edge statistics of $\mathcal{Q}_t$ and $\uQ_t$. 
In Section \ref{sec anafree}, we analyze the rectangular free convolution and prove some useful deterministic estimates. Based on these estimates, in Sections \ref{sec largelocal} and \ref{sec localregular}, we prove the local laws for the resolvents of $\mathcal{Q}_t$ and $\uQ_t$ in the case with a Gaussian noise matrix $X$. Finally in Section \ref{sec_pfgen}, we extend these local laws to the case with generally distributed $X$, and prove the results on the edge statistics of $\mathcal{Q}_t$ and $\uQ_t$. 

\vspace{5pt} 

\noindent {\bf Convention.}  The fundamental large parameter is $n$ and we always assume that $p$ is comparable to and depends on $n$. We use $C$ to denote a generic large positive constant, whose value may change from one line to the next. Similarly, we use $\epsilon$, $\tau$, $\delta$, etc. to denote generic small positive constants. If a constant depends on a quantity $a$, we use $C(a)$ or $C_a$ to indicate this dependence. For two quantities $a_n$ and $b_n$ depending on $n$, the notation $a_n = \OO(b_n)$ means that $|a_n| \le C|b_n|$ for some constant $C>0$, and $a_n=\oo(b_n)$ means that $|a_n| \le c_n |b_n|$ for some positive sequence $c_n\downarrow 0$ as $n\to \infty$. We also use the notations $a_n \lesssim b_n$ if $a_n = \OO(b_n)$, and $a_n \sim b_n$ if $a_n = \OO(b_n)$ and $b_n = \OO(a_n)$. For a matrix $A$, we use $\|A\|:=\|A\|_{l^2 \to l^2}$ to denote the operator norm; 
for a vector $\mathbf v=(v_i)_{i=1}^n$, $\|\mathbf v\|\equiv \|\mathbf v\|_2$ stands for the Euclidean norm. 
For a matrix $A$ and a positive number $a$, we write $A=\OO(a)$ if $\|A\|=\OO(a)$. In this paper, we often write an identity matrix of any dimension as $I$ or $1$ without causing any confusions.

\section{Main results}\label{sec_mainresult}

%
%

\subsection{The model and rectangular free convolution}


We consider a class of deformed rectangular matrices of the form (\ref{eq_defnmodel}), where $Y$ is a $p\times n$ deterministic signal matrix of large rank, $X$ is a $p\times n$ random noise matrix whose entries $x_{ij}$ are real independent random variables satisfying
\begin{equation}\label{assm1}
\mathbb{E} x_{ij} =0, \ \quad \ \mathbb{E} \vert x_{ij} \vert^2  =n^{-1}, \quad 1 \leq i \leq p, \ 1 \leq j \leq n,
\ee 
and $t>0$ gives the noise level. Unlike \cite{DOZIER20071099,DOZIER2007678,VLM2012}, we have used $t$ instead of $\sigma^2$ to denote the variance of the noise, because in this paper we will consider a full range of scales for the noise level. 
Hence it is more instructive to take the noise variance to be a varying parameter. Following the random matrix literature (see e.g. \cite{erdos2017dynamical}), we choose it 
to be the time parameter $t$. In particular, one can also consider the dynamics of $Y_t$ as $t$ changes; see Section \ref{sec DBM} below for a more detailed discussion regarding this point of view. In this paper, we consider the high dimensional setting, where the aspect ratio $c_n:=p/n$ converges to a finite positive constant. Without loss of generality, by switching the roles of $Y_t$ and $Y_t^\top$ if necessary, we can assume that  
\begin{equation}\label{eq_defnc}
 \tau\le c_n \le 1, 
\end{equation}
for some small constant $0<\tau<1$. Let 
\be\label{SVD}Y=O_1WO_2^\top,  \ee
be a singular value decomposition of $Y$, where $W$ is a $p\times n$ rectangular diagonal matrix,
$$W= \begin{pmatrix} D , 0\end{pmatrix}, \quad D^2=\diag(d_1, \cdots, d_p),$$
with $\sqrt{d_1} \geq \sqrt{d_2} \geq \cdots \geq \sqrt{d_p}\ge 0$ being the singular values of $Y$. We assume that the ESD of $YY^\top$ has a regular square root behavior near the right edge. 
Following \cite{edgedbm}, we state the regularity condition in terms of the Stieltjes transform of $V:=WW^\top$,
\be\label{eq_defng} m_V(z):= \frac1p\tr \left(V-z \right)^{-1}= \frac1p\sum_{i=1}^p \frac1{d_i -z},\quad z\in \C_+:=\{z\in \C: \im z>0\}.\ee

\begin{definition}[$\eta_*$-regular]\label{assumption_edgebehavior}
Let $\eta_*$ be a parameter satisfying $\eta_*:=n^{-\phi_*}$ for some constant $0<\phi_* \le 2/3$. We say $V$ (or equivalently, $Y$, $W$ or $m_V$) is $\eta_*$-regular around the largest eigenvalue $d_1$ if there exist constants $c_V>0$ and $C_V>1$ such that the following properties hold for $\lambda_+:=d_1$ (here $\lambda_+$ is a standard notation for the right spectral edge in random matrix literature). 
\begin{itemize}
\item  For $z=E+\ii \eta$ with $\lambda_+ - c_V \le E \le \lambda_+ $ and $\eta_* + \sqrt{\eta_* |\lambda_+ - E|}\le \eta \le 10$, we have
\be\label{regular1}
\frac{1}{C_V} \sqrt{|\lambda_+ - E| + \eta} \le \im m_V(E+\ii \eta) \le C_V\sqrt{|\lambda_+ - E| + \eta} .
\ee
For $z=E+\ii \eta$ with $\lambda_+ \le E \le \lambda_+ + c_V$ and $\eta_* \le \eta \le 10$, we have
\be\label{regular2}
\frac{1}{C_V} \frac{\eta}{|\lambda_+ - E| + \eta} \le \im m_V(E+\ii \eta) \le C_V\frac{\eta}{|\lambda_+ - E| + \eta}.
\ee

\item We have $2c_V \le \lambda_+ \le C_V/2$.

\item We have $\norm{V} \le n^{C_V}$.
\end{itemize}
\end{definition}

\begin{remark}
The motivation for conditions \eqref{regular1} and \eqref{regular2} is as follows: if $m(z)$ is the Stieltjes transform of a density $\rho$ with square root behavior around $\lambda_+$, i.e. $\rho(x)\sim \sqrt{(\lambda_+-x)_+}$, then \eqref{regular1} and \eqref{regular2} hold for $\im m(z)$ with $\eta_*=0$. For a general $\eta_*>0$, \eqref{regular1} and \eqref{regular2} essentially mean that the empirical spectral density of $V$ behaves like a square root function near $\lambda_+$ on any scale larger than $\eta_*$. The condition $\eta\le 10$ in the assumption is purely for definiteness of presentation---we can replace 10 with any constant of order 1. 

Consider a large rank matrix $Y$ whose spectral density of singular values follows a continuous function, say $\rho$, on a scale $\eta_* \ll 1$. Then the square root behavior of $\rho$ appears naturally near the spectral edge, which is the point where the density becomes zero. The conditions \eqref{regular1} and \eqref{regular2} hold for many Gram type random matrix ensembles for some $ n^{-2/3} \ll \eta_* \ll 1$, such as sample covariance matrices \cite{pillai2014}, separable sample covariance matrices \cite{dingyang2,yang20190}, random Gram matrices \cite{Alt_Gram,AEK_Gram} and sparse sample covariance matrices \cite{hwang2019}.
\end{remark}


Recall $\mathcal{Q}_t=Y_t Y_t^\top.$ Let $\rho_{w,t}$ be the asymptotic spectral density of $\mathcal{Q}_t$ as $n\to \infty$ and $m_{w,t}$ be the corresponding Stieltjes transform, i.e.,  
\be\label{def mwt}m_{w,t}(z):=\int \frac{\rho_{w,t}(x)\dd x}{x-z}.\ee
Here the $w$ in the subscript refers to the matrix $W$, and we use it to remind ourselves that $\rho_{w,t}$ and $m_{w,t}$ only depend on the singular values of $Y$. It is known that  \cite{DOZIER20071099,DOZIER2007678,VLM2012} for any $t>0$, $m_{w,t}$ is the unique solution to
\begin{equation}\label{originalequation}
m_{w,t}=\frac{1}{p} \sum_{i=1}^p \frac{1}{d_i(1+c_n tm_{w,t})^{-1}-(1+c_ntm_{w,t})z+t(1-c_n)},
\end{equation} 
such that  $\im m_{w,t}>0$ for any $z \in \mathbb{C}_+$. 
Adopting notations from free probability theory \cite{capfreeprob}, we shall call $\rho_{w,t}$ the rectangular free convolution of $\rho_{w,0}$ with the Marchenko-Pastur (MP) law at time $t$. Let $\lambda_{+,t}$ be the right-most edge of $\rho_{w,t}$. 
In \cite{DOZIER20071099,DOZIER2007678,VLM2012}, it has been shown that $\rho_{w,t}$ has a regular square root behavior near $\lambda_{+,t}$ when $t\sim 1$. In the following lemma, we extend this result to the  $t=\oo(1)$ case.

\begin{lemma}\label{lem regularDBM}
Suppose (\ref{assm1}) and \eqref{eq_defnc} hold, and $V$ is $\eta_*$-regular in the sense of Definition \ref{assumption_edgebehavior}. Moreover, we assume that  $t$ satisfies $n^{\epsilon} \eta_* \leq t^2 \leq n^{-\epsilon}$ for a small constant $\epsilon>0$. Then we have 
\begin{equation}\label{sqrtdensity0}
\rho_{w,t}(E) \sim \sqrt{(\lambda_{+,t}-E)_+} \quad \text{for}\quad \lambda_{+,t} - 3c_V/4 \le E\le \lambda_{+,t}+3 c_V/4 ,
\end{equation}
and for $z=E+\ii \eta \in \C_+$,
\begin{equation}\label{Immc}
\im m_{w,t}(z) \sim
\begin{dcases}
\sqrt{|E-\lambda_{+,t}|+\eta}, & \lambda_{+,t} - 3c_V/4 \le  E\leq \lambda_{+,t} \\
\frac{\eta}{\sqrt{|E-\lambda_{+,t}|+\eta}}, &\lambda_{+,t}\le  E  \le \lambda_{+,t} + 3c_V/4
\end{dcases}.
\end{equation} 
\end{lemma}
\begin{proof}
This lemma is an immediate consequence of Lemma \ref{lem_asymdensitysquare} and Lemma \ref{lem_asymdensitysquare2} below.
\end{proof}
\begin{remark}
We have required a lower bound $t^2\gg \eta_*$ in the assumption due to the following reason. For very small $t$, the edge behavior of $\rho_{w,t}$ is only a perturbation of the edge behavior of $YY^\top$ near $\lambda_+$, and $t=\sqrt{\eta_*}$ is the threshold when the random matrix statistics of $\sqrt{t}X$ begins to dominate over the effect of the edge eigenvalues of $YY^\top$. Theoretically, if the entries of $X$ are i.i.d. Gaussian, it has been shown in \cite{DY20201} that the edge statistics of $Y_tY_t^\top$ already converges to the local equilibrium when $t^2\gg \eta_*$.
\end{remark}


\subsection{Rectangular Dyson Brownian Motion}\label{sec DBM}

In this subsection, we state the main results for the case where the entries of $X$ are i.i.d. Gaussian random variables satisfying \eqref{assm1}. From both the theoretical and realistic points of view, this is perhaps the most important example of large deformed rectangular matrices. In this case, we shall call the the evolution of $\cal Q_t:=( Y+\sqrt{t}X)(Y+\sqrt{t}X)^\top$ with respect to $t$ a rectangular {matrix} Dyson Brownian motion (MDBM), and call the evolution of the {\it eigenvalues} of $\cal Q_t$ with respect to $t$ the rectangular Dyson Brownian motion (DBM). They are extensions of the symmetric MDBM and DBM for Wigner type random matrix ensembles, which have been used crucially in proving the universality conjecture for Wigner matrices \cite{EPRSY2010,ESY2011,ESYY2012,Bulk_univ}; for a more extensive review, we refer the reader to \cite{erdos2017dynamical} and references therein. In this paper, we mainly focus on the eigenvalue and eigenvector statistics of $\cal Q_t$ and $\uQ_t$ near the right-most edge for each fixed $t=\oo(1)$. On the other hand, the dynamics of the rectangular DBM with respect to $t$ will be studied in \cite{DY20201}, and the proof there also depends heavily on the results proved in this section.

Most of our main results can be formulated in a simple and unified fashion using the following $(p+n)\times (p+n)$ symmetric block matrix
$$H_t \equiv H_t (Y_t):= \begin{pmatrix} 0 & Y_t \\ Y_t^\top & 0\end{pmatrix},$$
which we shall refer to as the linearization of the matrices $\cal Q_t$ and $\uQ_t$. 

\begin{definition}[Resolvents]\label{resol_not}
We define the resolvent of $H_t$ as
 \begin{equation}\label{eqn_defG}
G(z)\equiv G(Y_t , z) \equiv G(t, X, z):=(z^{1/2}H_t - z)^{-1}, \quad z\in \mathbb C_+ . 
\end{equation}
For $\cal Q_t:= Y_tY_t^\top$ and $\uQ_t:=Y_t^\top Y_t$, we define the resolvents 
\begin{equation}\label{def_green}
  \mathcal G\equiv \cal G(Y_t,z) :=\left({\mathcal Q}_t -z\right)^{-1} , \ \ \ \uG \equiv \uG(Y_t,z):=\left({\uQ}_t-z\right)^{-1} .
\end{equation}
 We denote the empirical spectral density $\rho$ of $ {\mathcal Q}_t$ and its Stieltjes transform as
\be\label{defn_m}
\rho\equiv \rho(Y_t,z):= \frac{1}{p} \sum_{i=1}^p \delta_{\lambda_i( {\mathcal Q}_t)},\quad m(z)\equiv m(Y_t,z):=\int \frac{1}{x-z}\rho(\dd x)=\frac{1}{p} \mathrm{Tr} \, \mathcal G(z),
\ee
where $\lambda_i( {\mathcal Q}_t)$, $1\le i \le p$, denote the eigenvalues of $\cal Q_t$ in the decreasing order.
Similarly, we denote the empirical spectral density $\urho$ of $\uQ_t$ and its Stieltjes transform as
\be\label{defn_m2}
\urho\equiv \urho(Y_t,z):= \frac{1}{n} \sum_{i=1}^n \delta_{\lambda_i( {\uQ}_t)},\quad \um(z)\equiv \um (Y_t,z):=\int \frac{1}{x-z}\urho(\dd x)=\frac{1}{n} \mathrm{Tr} \, \uG(z).
\ee

\end{definition}

Using Schur complement formula, one can check that 
\begin{equation} \label{green2}
G = \left( {\begin{array}{*{20}c}
   { \mathcal G} & z^{-1/2} \mathcal G Y_t \\
   {z^{-1/2} Y_t^\top \mathcal G} & { \uG }  \\
\end{array}} \right)= \left( {\begin{array}{*{20}c}
   { \mathcal G} & z^{-1/2} Y_t \uG   \\
   z^{-1/2} {\uG}Y_t^\top & { \uG }  \\
\end{array}} \right).
\end{equation}
Thus a control of $G(z)$ yields directly a control of the resolvents $\mathcal G$ and $\uG$.

Since $\uQ$ share the same nonzero eigenvalues with $\cal Q$ and has $n-p$ more zero eigenvalues due to (\ref{eq_defnc}), we have 
\be\label{m21} \um =\frac{p}{n} m - \frac{n-p}{nz} = c_n m - \frac{1-c_n}{z}.\ee
We will see that $(m,\um)$ is asymptotically equal to $(m_{w,t},\um_{w,t})$ as $n\to \infty$, where
\be \label{undeline m} \mm_{w,t}: = c_n m_{w,t} - \frac{1-c_n}{z}.\ee
Furthermore, we define the asymptotic matrix limit of $G(z)$ as
\be\label{defnPi}
\Pi(z):=\begin{bdmatrix} \frac{-(1+c_n t m_{w,t})}{z(1+c_n t m_{w,t})(1+t \underline m_{w,t} )-YY^\top} &  \frac{-z^{-1/2} }{z(1+c_n t m_{w,t})(1+t \underline m_{w,t} )-YY^\top}Y \\ 
Y^\top \frac{-z^{-1/2} }{z(1+c_n t m_{w,t})(1+t \underline m_{w,t} )-YY^\top} & \frac{-(1+t \underline m_{w,t} )}{z(1+c_n t m_{w,t})(1+t \underline m_{w,t} )-Y^\top Y}  
\end{bdmatrix}.
\ee
It is easy to check that the inverse of $\Pi$ is 
\be\label{pi_i}
\Pi^{-1}(z):=\begin{pmatrix}  - z(1+t \underline m_{w,t} )I_{p} &  z^{1/2}Y \\ z^{1/2}Y^\top   & -z(1+c_n t m_{w,t})I_n \end{pmatrix}^{-1} .
\ee
For $z=E+\ii\eta\in \C_+$, we introduce the notation
\be\label{defn kappa}\kappa\equiv \kappa_E := |E-\lambda_{+,t}|.\ee
Then for any constant $\vartheta>0$, we define the spectral domains
\begin{equation}\label{eq_domantheta}
\begin{split}
\mathcal{D}_{\vartheta}:=&\left\{z=E+\ii \eta: \lambda_{+,t}- \frac34c_V \leq E \leq \lambda_{+,t} +\vartheta^{-1} t^2, n\eta\left( t+\sqrt{\kappa+\eta}\right) \geq n^{\vartheta} ,\eta\le 10 \right\} \cup \mathcal{D}^{out}_{\vartheta},
\end{split}
\end{equation}
where 
\begin{equation}\label{eq_domanthetaout}
\begin{split}
\mathcal{D}^{out}_{\vartheta}:=\left\{z=E+\ii \eta:  \lambda_{+,t} \le E \le  \lambda_{+,t}+ \frac34c_V,  n\eta \sqrt{\kappa+\eta} \geq n^{\vartheta}, \eta\le 10 \right\}.
\end{split}
\end{equation}

Before stating the first main result of this paper, we introduce the following notion of stochastic domination, which was first introduced in \cite{Average_fluc} and subsequently used in many works on random matrix theory, such as \cite{isotropic,principal,local_circular,Delocal,Semicircle,Anisotropic}. It simplifies the presentation of the results and their proofs by systematizing statements of the form ``$\xi$ is bounded by $\zeta$ with high probability up to a small power of $n$".

\begin{definition}[Stochastic domination]\label{stoch_domination}
(i) Let
\[\xi=\left(\xi^{(n)}(u):n\in\bbN, u\in U^{(n)}\right),\hskip 10pt \zeta=\left(\zeta^{(n)}(u):n\in\bbN, u\in U^{(n)}\right)\]
be two families of nonnegative random variables, where $U^{(n)}$ is a possibly $n$-dependent parameter set. We say $\xi$ is stochastically dominated by $\zeta$, uniformly in $u$, if for any fixed (small) $\epsilon>0$ and (large) $D>0$, 
\[\sup_{u\in U^{(n)}}\bbP\left(\xi^{(n)}(u)>n^\epsilon\zeta^{(n)}(u)\right)\le n^{-D}\]
for large enough $n \ge n_0(\epsilon, D)$, and we shall use the notation $\xi\prec\zeta$. Throughout this paper, the stochastic domination will always be uniform in all parameters that are not explicitly fixed (such as matrix indices, and $z$ that takes values in some compact set). 
If for some complex family of random variables $\xi$ we have $|\xi|\prec\zeta$, then we will also write $\xi \prec \zeta$ or $\xi=\OO_\prec(\zeta)$.

\vspace{5pt}
\noindent (ii) We extend the definition of $\OO_\prec(\cdot)$ to matrices in the operator norm sense as follows. Let $A$ be a family of random matrices and $\zeta$ be a family of nonnegative random variables. Then $A=\OO_\prec(\zeta)$ means that $\|A\| \prec\zeta$.


\vspace{5pt}
\noindent (iii) We say an event $\Xi$ holds with high probability if for any constant $D>0$, $\mathbb P(\Xi)\ge 1- n^{-D}$ for large enough $n$.
\end{definition}

Now we are ready to state the local law on the resolvent $G(z)$. For simplicity of notations, we introduce the following deterministic control parameter 
\begin{equation}\label{eq_defpsi}
\Psi (z)  : =  \sqrt {\frac{{\Im \, m_{w,t}(z)   }}{{n\eta }}} + \frac{1}{n\eta}  .
\end{equation}

\begin{theorem}\label{thm_local}
Suppose \eqref{eq_defnc} holds, $V$ is $\eta_*$-regular in the sense of Definition \ref{assumption_edgebehavior}, and the entries of $X$ are i.i.d. Gaussian random variables satisfying \eqref{assm1}. Let $t$ satisfy $n^{\epsilon} \eta_* \leq t^2 \leq n^{-\epsilon}$ for a small constant $\epsilon>0$. Then for any constant $\vartheta>0$, the following estimates hold uniformly in $z\in \cal D_\vartheta$:
\begin{itemize}
\item[(1)] ({\bf anisotropic local law}) for any deterministic unit vectors $\mathbf u, \mathbf v \in \mathbb R^{p+n}$,
\begin{equation}\label{aniso_law}
\left|  \mathbf u^\top \Pi^{-1}(z)\left[G(z)-\Pi(z)\right] \Pi^{-1}(z)\mathbf v   \right|\prec  t\Psi(z) + \frac{t^{1/2}}{n^{1/2}};
\end{equation}

\item[(2)] ({\bf averaged local law}) for $z\in \cal D_\vartheta$, we have
\be\label{averin}
|m(z)-m_{w,t}(z)|\prec \frac{1}{n\eta} ;
\ee
for $z\in \cal D_\vartheta^{out}$, we have
\be\label{averout}
|m(z)-m_{w,t}(z)|\prec \frac{1}{n(\kappa+\eta)}+\frac{1}{(n\eta)^2\sqrt{\kappa+\eta}}.
\ee
\end{itemize}
\end{theorem}

\begin{remark}
The anisotropic local law \eqref{aniso_law} can be reformulated as follows: for any deterministic unit vectors $\mathbf u, \mathbf v \in \mathbb R^{p+n}$,
\begin{equation}\label{aniso_law_rem}
\left|  \mathbf u^\top \left[G(z)-\Pi(z)\right] \mathbf v   \right|\prec  \left(t\Psi(z) + \frac{t^{1/2}}{n^{1/2}}\right)\|\Pi(z)\mathbf u\|\|\Pi(z)\mathbf v\|.
\end{equation}
In Lemma \ref{lem_domian1control} and Lemma \ref{lem_domian2control} below, we will see that
\be\label{rough Pi} \|\Pi(z)\mathbf u\| \lesssim \varpi^{-1}(z),\quad \varpi(z):=\begin{cases} t^2 + \eta+ t\sqrt{\kappa+\eta}, &\text{ if } E\le \lambda_{+,t}\\ t^2 + \kappa + \eta, &\text{ if } E\ge \lambda_{+,t}\end{cases},\ee
for any unit vector $\mathbf u\in \mathbb R^{p+n}$.
\end{remark}

Theorem \ref{thm_local} has some important implications, including the rigidity of eigenvalues and Tracy-Widom distribution for the edge eigenvalues of $\cal Q_t$. The former one will be presented in Theorem \ref{thm_implication} below, while the latter one will be presented in another paper \cite{DY20201} due to the length constraint. For the purpose of \cite{DY20201}, we also need another type of local law when $Y$ itself is a random matrix. In fact, when $Y$ is random, we can still apply Theorem \ref{thm_local} by conditioning on $Y$, but the asymptotic density $\rho_{w,t}$ and its Stieltjes transform $m_{w,t}$ will be random, which makes the result less convenient to use. On the other hand, if we know that $Y$ already satisfies a local law in the sense that $m_V$ is close to some deterministic function, then we can show that $m_{w,t}$ also converges to a deterministic limit. 

\begin{assumption}\label{assm regularW}
Let $m_c(z)$ be the Stieltjes transform of a deterministic law $\rho_c(x)$ that is compactly supported on $[0,\lambda_+]$, and such that $\|\rho_c\|_\infty =\OO(1)$ and 
\be\label{sqrtrhoc}\rho_c(x)\sim \sqrt{x},\quad \text{for} \ \ \lambda_+-c_V\le x\le \lambda_+.\ee
We assume that the Stieltjes transform of $V$ satisfies that for any constant $\vartheta>0$,
$$|m_{V}(z)-m_c(z)| \prec \frac{1}{n\eta},\quad \lambda_+ - c_V \le E \le \lambda_+,\ 0\le \eta\le 10, \ \sqrt{|E-\lambda_+|+\eta}\ge \frac{n^\vartheta}{n\eta},$$
and
$$|m_{V}(z)-m_c(z)| \prec \frac{1}{n(|E-\lambda_+|+\eta)}+ \frac{1}{(n\eta)^2\sqrt{|E-\lambda_+|+\eta}},\quad \lambda_+ \le E \le \lambda_++c_V,\ n^{-2/3+\vartheta}\le \eta\le 10.$$
\end{assumption}

We denote the rectangular free convolution of $\rho_c$ with the MP law at time $t$ as $\rho_{c,t}$, and denote its Stieltjes transform as $m_{c,t}$. More precisely, similar to \eqref{originalequation}, $m_{c,t}(z)$ is the unique solution of
\begin{equation}\nonumber
m_{c,t}= \int \frac{(1+c_n tm_{c,t})\rho_c(x)\dd x}{x- (1+c_ntm_{c,t})^2z+t(1-c_n)(1+c_ntm_{c,t})},
\end{equation} 
such that  $\im m_{c,t}(z)>0$ for any $z \in \mathbb{C}_+$. With \eqref{sqrtrhoc}, it is easy to check that $m_c$ is $\eta_*$-regular in the sense of Definition \ref{assumption_edgebehavior} with $\eta_*=0$. In particular, Lemma \ref{lem regularDBM} holds for any $0< t\le n^{-\e}$, which shows that $\rho_{c,t}$ has a square root behavior around its right edge, denoted by $\lambda_{c,t}$. Similar as before, we define $\kappa_c:=|E-\lambda_{c,t}|$ and the spectral domain
\begin{equation}\label{eq_domanthetas}
\begin{split}
\mathcal{D}^c_{\vartheta}:=&\left\{z=E+\ii \eta: \lambda_{c,t}- \frac34c_V \leq E \leq \lambda_{c,t}, \sqrt{\kappa+\eta} \geq \frac{n^{\vartheta}}{n \eta} ,\eta\le 10 \right\} \\
\bigcup &\left\{z=E+\ii \eta:  \lambda_{c,t} \le E \le  \lambda_{c,t}+ \frac34c_V,  n^{-2/3+\vartheta}\le \eta\le 10 \right\},
\end{split}
\end{equation}
and the control parameter 
\begin{equation}\nonumber 
\Psi_c (z)  : = \sqrt {\frac{{\Im \, m_{c,t}(z)   }}{{n\eta }}} + \frac{1}{n\eta}.
\end{equation}
Now we have the following result.

\begin{theorem}\label{thm_local2}
Suppose \eqref{eq_defnc} holds, $m_V$ satisfies Assumption \ref{assm regularW}, and the entries of $X$ are i.i.d. Gaussian random variables satisfying \eqref{assm1}. Let $t$ satisfy $0 \leq t^2 \leq n^{-\epsilon}$ for a small constant $\epsilon>0$. Then for any constant $\vartheta>0$, the following averaged local laws hold uniformly in $z\in \cal D^c_\vartheta$: 
\be\label{averin2}
|m(z)-m_{c,t}(z)|\prec \frac{1}{n\eta} \quad \text{for}\quad E\le \lambda_{c,t},
\ee
and 
\be\label{averout2}
|m(z)-m_{c,t}(z)|\prec \frac{1}{n(\kappa+\eta)}+\frac{1}{(n\eta)^2\sqrt{\kappa+\eta}}\quad \text{for}\quad E\ge \lambda_{c,t}.
\ee
\end{theorem}

\subsection{General deformed rectangular matrices}\label{sec_mainresultgeneral}

In this section, we extend the local laws in Theorem \ref{thm_local} to the general case where the entries of $X$ are generally distributed, assuming only certain moment assumptions. 
We remark that the proof of the general case actually uses Theorem \ref{thm_local} in an essential way; see the discussions in Section \ref{sec_pfgen}.

We define the following control parameter 
\be\label{defnPhi} \Phi(z):=\frac{1}{\varpi(z)}\left( t\Psi(z) + \frac{t^{1/2}}{n^{1/2}}\right).\ee
Using the definition of $\cal D_\vartheta$ and the definition of $\varpi$ in \eqref{rough Pi}, it is easy to check that 
\be\label{boundpsi} \Phi(z) \lesssim n^{-\vartheta/2}, \quad z\in \cal D_\vartheta.\ee
For any vector $\bu\in \R^{p+n}$, we introduce the notation 
\be\label{upi}   \|\mathbf u\|_\Pi (z):=  \|\Pi(z) \mathbf u_1\|+  \|\Pi(z)\mathbf u_2\|,\ee
where $\bu_1\in \R^p$ and $\bu_2\in \R^n$ are vectors such that $\bu=\begin{pmatrix} \bu_1 \\ \bu_2\end{pmatrix}$.

\begin{theorem}\label{thm_local_gen}
Suppose \eqref{eq_defnc} holds, $V$ is $\eta_*$-regular in the sense of Definition \ref{assumption_edgebehavior}, and $X$ is a $p\times n $ random matrix whose entries $x_{ij}$ are real independent random variables satisfying \eqref{assm1} and 
\begin{equation}\label{condition_3rd}
\mathbb{E} x_{ij} ^3 =0,\quad 1\le i \le p,  \ 1\le j \le n.
\end{equation}
Moreover, assume that the entries of $X$ have finite moments up to any order, that is, for any fixed $k\in \N$,
\begin{equation}\label{conditionA2}
\mathbb{E}\vert  \sqrt{n}x_{ij} \vert^k \le C_k,
\end{equation}
for some constant $C_k>0$. Let $t$ satisfy $n^{\epsilon} \eta_* \leq t^2 \leq n^{-\epsilon}$ for a small constant $\epsilon>0$. Then for any constant $\vartheta>0$, the following estimates hold uniformly in $z\in {\cal D}_\vartheta$:
\begin{itemize}
\item[(1)] ({\bf anisotropic local law}) for any deterministic unit vectors $\mathbf u, \mathbf v \in \mathbb R^{p+n}$,
\begin{equation}\label{aniso_law_gen}
\left|  \mathbf u^\top \left[G(z)-\Pi(z)\right] \mathbf v   \right|\prec \Phi(z)  \|\mathbf u\|_\Pi^{1/2}  \|\mathbf v\|_\Pi^{1/2};
\end{equation}

\item[(2)] ({\bf averaged local law}) for $z\in \cal D_\vartheta$, we have
\be\label{averin_gen}
|m(z)-m_{w,t}(z)|\prec \frac{1}{n\eta} ;
\ee
for $z\in \cal D_\vartheta^{out}$, we have
\be\label{averout_gen}
|m(z)-m_{w,t}(z)|\prec \frac{1}{n(\kappa+\eta)}+\frac{1}{(n\eta)^2\sqrt{\kappa+\eta}} + \frac{\Phi}{n\eta}.
\ee
\end{itemize}
\end{theorem}

\begin{remark} 
Theorem \ref{thm_local_gen} is proved through a comparison with the Gaussian case in Theorem \ref{thm_local}, where we need the first three moments of $x_{ij}$ to match those of the Gaussian random variables. Without the condition (\ref{condition_3rd}), the comparison argument cannot give the local laws up to a small enough $\eta$ (for our purpose, we need $\eta$ to be as small as $n^{-2/3}$). 
We will try to remove the assumption (\ref{condition_3rd}) in future works.
\end{remark}

Consider the singular value decomposition of $Y_t$,
$$Y_t  = \sum_{k = 1}^{p} {\sqrt {\lambda_k} \bm{\xi}_k } \bm{\zeta} _{k}^\top ,$$
 where $\sqrt{\lambda_1}\ge \sqrt{\lambda_2} \ge \ldots \ge \sqrt{\lambda_{p}}$ are the singular values of $Y_t$, 
$\{\bm{\xi}_{k}\}_{k=1}^{p}$ are the left singular vectors, and $\{\bm{\zeta}_{k}\}_{k=1}^{n}$ are the right-singular vectors.
As a consequence of Theorem \ref{thm_local_gen}, we can obtain some important estimates on $\{\lambda_k\}$, $\{\bm{\xi}_{k}\}$ and $\{\bm{\zeta}_{k}\}$. Before stating them, we first introduce some notations. For any fixed $E$, let $\eta_l(E)$ (``$l$" stands for ``lower bound") be the unique solution of
  $$n\eta_l(E)\left( t + \sqrt{\kappa_E + \eta_l(E)}\right)=1.$$
For $t$ satisfying $t \gg n^{-1/3}$, it is easy to check that   
\be\label{etalE}
 \eta_l(E)\sim \frac{1}{n(t+\sqrt{\kappa_E})}. 
\ee
We define the classical location $\gamma_j$ for the $j$-th eigenvalue of $\mathcal Q_t$ as
\begin{equation}\label{gammaj}
\gamma_j:=\sup_{x}\left\{\int_{x}^{+\infty} \rho_{w,t}(x)dx > \frac{j-1}{p}\right\} .
\end{equation}
In particular, by the square root behavior of $\rho_{w,t}$ in \eqref{sqrtdensity0}, we have $\gamma_1 = \lambda_{+,t}$, and $\kappa_{\gamma_j}=|\gamma_j -\lambda_{+,t}|\sim j^{2/3}n^{-2/3}$ for $j\ge 2$ and $\gamma_j \ge \lambda_{+,t}-c_V/2$. 

\begin{theorem}\label{thm_implication}
Suppose the assumptions of Theorem \ref{thm_local_gen} hold. 
\begin{itemize}
\item[(1)] ({\bf Eigenvalues rigidity}) For any $k$ such that $\lambda_{+,t} -  c_V/2 < \gamma_k \le \lambda_{+,t}$, we have
\begin{equation}\label{rigidity}
\vert \lambda_k - \gamma_k \vert \prec  n^{-2/3}  k^{-1/3}+  \eta_{l}(\gamma_k)  .
\end{equation}

\item[(2)] ({\bf Edge universality}) There exist constants $\epsilon,\delta >0$ such that for all $x\in \mathbb R$, 
\be\label{edgeuniv} 
\begin{split}
\mathbb{P}^{G} \left(n^{{2}/{3}}(\lambda_1-\lambda_{+,t}) \leq x-n^{-\epsilon}\right)-n^{-\delta}& \leq \mathbb{P}\left(n^{{2}/{3}}(\lambda_1-\lambda_{+,t})\leq x\right) \\
& \leq \mathbb{P}^{G}\left(n^{{2}/{3}}(\lambda_1-\lambda_{+,t}) \leq x+n^{-\epsilon}\right)+n^{-\delta} , 
 \end{split}
\ee
 where $\mathbb P^G$ denotes the law for $X=(x_{ij})$ with i.i.d. Gaussian entries satisfying \eqref{assm1}.
 
\item[(3)] ({\bf Eigenvector delocalization}) For any deterministic unit vectors $\mathbf u \in \mathbb R^{p}$, $\mathbf v  \in \mathbb R^{n}$ and any $k$ such that $\lambda_{+,t} -  c_V/2 < \gamma_k \le \lambda_{+,t}$, we have that
\begin{equation}\label{delocal}
\begin{split}
 & \left| \mathbf u^\top \bm{\xi}_k  \right|^2 \prec \eta_l(\gamma_k)\cdot \left[\im (\mathbf u^\top \Pi (z_k) \mathbf u) + \Phi(z_k)\cdot \|\mathbf u\|_{\Pi}(z_k)\right], \\
  & \left| \mathbf v^\top \bm{\zeta}_k  \right|^2 \prec  \eta_l(\gamma_k)\cdot \left[\im (\mathbf v^\top \Pi (z_k) \mathbf v) + \Phi(z_k)\cdot \|\mathbf v\|_{\Pi}(z_k)\right],
  \end{split}
\end{equation}
where we denote $z_k:=\gamma_k+ \ii \eta_l(\gamma_k)$.
\end{itemize}
\end{theorem}
 \begin{remark} \label{rigid_multi}
As in \cite{EKYY,EYY,LY}, \eqref{edgeuniv} can be generalized to the finite correlation function of the $k$ largest eigenvalues for any fixed $k\in \N$:
\begin{align}
 \mathbb{P}^{G} \left( \left(n^{{2}/{3}}(\lambda_i-\lambda_{+,t}) \leq x_i -n^{-\epsilon}\right)_{1\le i \le k}\right)-n^{-\delta} \le \mathbb{P}  \left( \left(n^{{2}/{3}}(\lambda_i-\lambda_{+,t})  \leq x_i \right)_{1\le i \le k}\right)  \nonumber\\
  \le \mathbb{P}^{G} \left( \left(n^{{2}/{3}}(\lambda_i-\lambda_{+,t}) \leq x_i + n^{-\epsilon}\right)_{1\le i \le k}\right)+ n^{-\delta},  \label{edgeuniv_ext}
\end{align}
for all $x_i\in \R$, $1\le i \le k$. Moreover, in \cite{DY20201} we will show that $\mathbb{P}^{G}  \left( (n^{{2}/{3}}(\lambda_i-\lambda_{+,t}) \leq x_i)_{1\le i \le k} \right)$ converges to the Tracy-Widom law \cite{TW1,TW}. Thus \eqref{edgeuniv} and \eqref{edgeuniv_ext} actually show that for a general $X$, the largest few eigenvalues obey the Tracy-Widom fluctuation.
\end{remark}


\subsection{Statistical applications}\label{sec_statapp}
In this section, we discuss some potential applications of our results in high dimensional statistics. Specifically, we shall consider the model used in multicell multiuser MIMO system \cite{VBL} as an example, which belongs to the massive MIMO system  \cite{BHKD, VBS, NKDA}. The massive MIMO system is a promising technique to deal with large wireless communication systems, such as the design of 5G \cite{NKDA}. 
In contrast to the standard MIMO system that assumes a low-rank structure of $Y$ \cite{RMT_Wireless}, the massive MIMO system requires a large-rank signal matrix. 

First we introduce the model for the multicell multiuser MIMO system \cite{VBL}. Suppose there is a single target cell and $r$ nearby interfering cells. Each cell contains a single base station equipped with $p$ antennas and $K$ single-antenna users. Consider the uplink (reverse link) transmission where the target base station receives
signals from all users in all cells. Then we observe $n$ i.i.d. samples $\bm{y}_i$, $i=1,2,\cdots,n$, each of which can be modeled as 
\begin{equation}\label{eq_massmimomodel}
\bm{y}_i=\mathbf{H} \bm{z}_i+\sum_{k=1}^r \mathbf{W}_k \mathbf{z}^k_i+\bm{w}_i.  
\end{equation}
Here the transmitted data $\bm{z}_i \in \mathbb{R}^K$ is a centered random vector with covariance matrix $\Gamma$; $\mathbf{H} \in \mathbb{R}^{p \times K}$ is the channel matrix between the base station and the $K$ users; $\mathbf{z}_i^k \in \mathbb{R}^{K}$ is the interfering data in the $k$-th interfering cell with i.i.d. centered entries of unit variance; $\mathbf{W}_k \in \mathbb{R}^{p \times K}$ is the channel matrix between the base station and the users in cell $k$; $\bm{w}_i \in \mathbb{R}^p$ is the additive noise with i.i.d. centered entries of variance $\sqrt{t}$. We assume that all the random vectors $\bm{z}_i$, $\mathbf{z}^k_i$ and $\bm{w}_i$ are independent of each other. 
Suppose that the number of users in each cell is fixed, i.e., $K$ is a fixed integer,  and that the number of the neighboring interfering cells is large, i.e., $r$ is large. Denote 
\begin{equation*}
\widetilde{\mathbf{z}}_i=((\mathbf{z}_i^1)^\top, \cdots, (\mathbf{z}_i^r)^\top)^\top,\quad W:=(\mathbf{W}_1,\cdots, \mathbf{W}_r).
\end{equation*}       
Then we obtain the following matrix model by concatenating the $n$ observed samples:
\begin{equation}\label{eq_matrixmodelmimo}
\widetilde{Y}_t:=(\bm{y}_1, \cdots, \bm{y}_n)=\mathbf{H} \Gamma^{1/2} Z+W Z_I+\sqrt{t} X, 
\end{equation}
where $Z \in \mathbb{R}^{K \times n}$, $Z_I:=(\widetilde{\mathbf{z}}_1,\cdots, \widetilde{\mathbf{z}}_n) \in \mathbb{R}^{(rK) \times n}$ and $X:=t^{-1/2}(\bm{w}_1,\cdots, \bm{w}_n)\in \R^{p\times n}$ are independent random matrices with i.i.d. centered entries of unit variance. 
Denoting $Y_0:=\mathbf{H} \Gamma^{1/2} Z$ and $Y:=W Z_I$, 
 we can rewrite (\ref{eq_matrixmodelmimo}) as 
\begin{equation}\label{eq_defnmodemmimogeneralgenral}
\widetilde{Y}_t=Y_0+Y+\sqrt{t}X,
\end{equation} 
where $Y_0$ is a low-rank matrix representing the transmitted signals of the home cell, $Y$ is a large-rank matrix representing the signals of the interfering cells, and $\sqrt{t}X$ is the additive noise. Note that $Y_t:=Y+\sqrt{t}X$ is the deformed rectangular matrix model,
which, as mentioned in \cite[Section 3]{VBL}, is called the \emph{interference-plus-noise} matrix in the above application. 
We are interested in the estimation and inference of the model (\ref{eq_defnmodemmimogeneralgenral}). We will propose some useful statistics based on our results in Section \ref{sec_mainresultgeneral}. 
For definiteness, we assume that $Y_t$ satisfies the assumptions of Theorem \ref{thm_local_gen} in the following discussion,

In the first application, we are interested in detecting the signals, that is, the existence of $Y_0$. Let $\mathrm{r}_*$ the number of non-zero singular values of $Y_0$. 
Formally, we consider the following hypothesis testing problem 
\begin{equation}\label{eq_testhypothesis}
\mathbf{H}_0: \ \mathrm{r}_*=0 \quad \text{vs} \quad  \mathbf{H}_a: \ \mathrm{r}_*>0. 
\end{equation}   
Under the null hypothesis of $\mathbf{H}_0,$
by \eqref{edgeuniv_ext} we know that the joint distribution of the largest few eigenvalues of $\widetilde{Y}_t \widetilde{Y}_t^\top $ is universal regardless of the distributions of the entries of $X$. Since we have no a priori information on the interference matrix $Y$ and the noise level $t,$ we shall use the following pivotal statistic \cite{OAea} 
\begin{equation}\label{eq_testonestat}
\mathbb{T}_1=\frac{\mu_1-\mu_2}{\mu_2-\mu_3},
\end{equation}  
where $\mu_1 \geq \mu_2 \geq \cdots \geq \mu_{p \wedge n} $ are the eigenvalues of $\widetilde{Y}_t \widetilde{Y}_t^\top.$ The statistic $\mathbb{T}_1$ will be powerful if some singular values of $Y_0$ are 
above the threshold for BBP transition such that they give rise to some outliers of $\wt Y_t$, that is, singular values those are detached from the bulk singular value spectrum. This kind of assumption appears commonly in the literature of signal detection, see e.g. \cite{AN17,RRN14,5447639,nadler2008}. Furthermore, by the remark below \eqref{edgeuniv_ext}, under the null hypothesis $\mathbf{H}_0,$ $\mathbb{T}_1$ actually satisfies an explicit distribution that can be derived from the Tracy-Widom law. 

For the second application, we consider the estimation of the number of signals once we reject the null hypothesis of (\ref{eq_testhypothesis}). For simplicity, for now we assume that the eigenvalues $d_1>d_2>\cdots>d_*$ of $Y_0Y_0^\top$ are reasonably large such that they gives rise to $\mathrm{r}_*$ outliers 
$\mu_1>\mu_2>\cdots>\mu_{\mathrm{r}_*}$.  Following the discussions in Sections 3 and 4 of \cite{capfreeprob}, one can show that with probability $1-\oo(1)$,
\begin{equation*}
 \mu_i=\zeta_t^{-1}(d_i)+\oo(1), \quad 1 \leq i \leq \mathrm{r}_*,
\end{equation*} 
where $\zeta_t^{-1}(\cdot)$ is the inverse function of the subordination function defined in (\ref{eq_defnzeta}). With the estimates proved in Section \ref{sec_contour} below, we can show that $\mu_{i}>\lambda_{+,t}$ if $d_{i}>\zeta_t(\lambda_{+,t})$, where $\zeta_t(\lambda_{+,t})$ gives the threshold for BBP transition. On the other hand, by \eqref{rigidity} and Cauchy interlacing theorem, we have that 
\begin{equation*}
\mu_{j+\mathrm{r}_*}=\lambda_{+,t}+\OO_\prec(n^{-2/3+\e}), \quad \text{for any fixed} \ \ j \geq 1. 
\end{equation*}
 In light of the above observations, we propose the following statistic,
\begin{equation*}
\widehat{\mathrm{r}}_* :=\underset{1 \leq i \leq \ell}{\mathrm{arg min}} \left \{ \frac{\mu_{i+1}}{\mu_{i+2}}-1 \leq \omega \right\}.
\end{equation*}   
Here $\ell$ is a pre-given large constant and $\omega$ is a small number that can be chosen using a calibration procedure. We refer the readers to \cite[Section 4.1]{dingyang2} for more details. Using our result, Theorem \ref{thm_implication}, it is not hard to show that $\widehat{\mathrm{r}}_*$ is a consistent estimator of $\mathrm{r}_*.$ We also remark that our local law, Theorem \ref{thm_local_gen}, combined with the strategy in \cite{BDW, principal, ding2020, dingyang2} can give optimal convergent rates and exact asymptotic distributions for the outlier eigenvalues $\mu_i$, $1 \leq i \leq \mathrm{r}_*$. However, this requires a lot more dedicated efforts and is beyond the scope of the current paper. We will pursue this direction somewhere else. We also remark that in general, it may happen that only a subset of the eigenvalues of $Y_0Y_0^\top$ are above the BBP transition threshold, say $\mu_1> \cdots> d_{r_+} > \zeta_t(\lambda_{+,t})> d_{r_++1}> \cdots >d_{\mathrm{r}_*}$ for some $0<r_+<\mathrm{r}_*$. In this case, the estimator $\widehat{\mathrm{r}}_* $ will consistently give the value $r_+$, and it is known that eigenvalues $ d_{r_++1}, \cdots ,d_{\mathrm{r}_*}$ cannot be detected reliably using the singular values of $\widetilde{Y}_t$ only.


Finally, we mention that our results can be used to study many other problems involving large-rank deformed rectangular matrices. For instance, in Section \ref{sec anafree}, we will conduct a thorough analysis of $m_{w,t}$ and the equation (\ref{originalequation}). Based on the results there, we can propose a convex optimization based methodology to estimate the large-rank matrix $Y$ by utilizing (\ref{originalequation}) and the strategy in \cite{karoui2008spectrum}. Moreover, the model (\ref{eq_defnmodemmimogeneralgenral}) also appears in many other statistical problems. For example, in the factor model \cite{FWZZ,LETTAU20201,OAea}, $Y_0$ represents the excess return matrix, $Y$ is the cross-section (i.e. the common factors) part, and $\sqrt{t}X$ is the idiosyncratic component. In existing literature, $Y$ is commonly assumed to be either sparse or low-rank. Based on our results, we can study the factor model beyond the low-rank assumption. We will consider these problems in future works.







\section{Analysis of rectangular free convolution} \label{sec anafree}

The proof of the main results depends crucially on a good understanding of the rectangular free convolution $\rho_{w,t}$ and its Stieltjes transform $m_{w,t}$. In this section, we prove some deterministic estimates on them given that $m_{w,0}\equiv m_V$ is $\eta_*$-regular as in Definition \ref{assumption_edgebehavior}. In particular, we will show that $\rho_{w,t}$ has a regular square root behavior around the right edge as given in Lemma \ref{lem regularDBM}. When $t\sim 1$, some of the estimates have been proved in \cite{DOZIER20071099,DOZIER2007678,VLM2012}. Here we extend them to the case $n^{\epsilon} \eta_* \leq t^2 \leq n^{-\epsilon}$, which requires much more careful estimates regarding the equation \eqref{originalequation}. We expect the results of this section to be of independent interest in the statistical estimation of large rank deformed rectangular matrices.


\subsection{Basic estimates}


In this section, we collect some known estimates from the previous works \cite{DOZIER20071099,DOZIER2007678,VLM2012}. 
Following \cite{DOZIER20071099}, we denote $b_t(z):=1+c_n t m_{w,t}(z).$ It is easy to see from (\ref{originalequation}) that $b_t$ satisfies the following equation 
\begin{equation}\label{originaleqautionbbb}
b_t=1+\frac{tc_n}{p} \sum_{i=1}^p \frac{1}{b_t^{-1} d_i-b_t z+t(1-c_n)}.
\end{equation} 
Next, we introduce the notation
\begin{equation}\label{eq_defnzeta}
\zeta_t(z):=b_t^2 z-b_t t(1-c_n),
\end{equation}
which $\zeta$ is actually the subordination function of the rectangular free convolution \cite{capfreeprob}. Then the equation \eqref{originaleqautionbbb} can be rewritten as
\begin{equation}\label{eq_keyequation11}
\frac{1}{c_n t}\left(1-\frac{1}{b_t}\right)=m_{w,0}(\zeta_t).
\end{equation}
We remark that $m_{w,0}(\zeta_t)$ is well-defined because $\im \zeta_t>0$ whenever $\im z>0$; see Lemma \ref{existuniq} below.

%
%
%
%
%
%
As shown later, our main analysis will boil down to the study of the two analytic functions $\zeta_t$ and $b_t$ on $\C_+:=\{ z\in \C: \im z>0\}$. 
We first summarize some basic properties of these quantities, which have been proved in previous works \cite{DOZIER20071099,DOZIER2007678,VLM2012}. 

\begin{lemma}[Existence and uniqueness of asymptotic density] \label{existuniq}
For any $t>0$, the following properties hold.
\begin{itemize}
\item[(i)] There exists a unique solution $m_{w,t}$ to equation \eqref{originalequation} satisfying that $\im m_{w,t}(z)> 0$ and $\im z m_{w,t}(z)> 0$ if $z\in \C_+ $.

\item[(ii)] For all $x \in \mathbb{R}\setminus\{0\},$ $\lim_{\eta  \downarrow 0} m_{w,t}(x+\ii \eta)$ exits, and we denote it as $m_{w,t}(x).$ The function $m_{w,t}$ is continuous on $\mathbb{R}\setminus \{0\}$, 
and $\rho_{w,t}(x):=\pi^{-1} \im m_{w,t}(x)$ is a continuous probability density function on $\mathbb{R}_+:=\{x\in \R:x>0\}$. Moreover, $m_{w,t}$ is the Stieltjes transform of $\rho_{w,t}$. Finally, $m_{w,t}(x)$ is a solution to (\ref{originalequation}) for $z=x$. 

\item[(iii)] For all $x \in \mathbb{R}\setminus\{0\},$ $\lim_{\eta  \downarrow 0} \zeta_t(x+\ii \eta)$ exits, and we denote it as $\zeta_t(x).$ Moreover, we have $\im \zeta_t(z)>0$ if $z\in \C_+ $. 

\item[(iv)] We have $\re b_t(z)>0$ for all $z\in \C_+$ and 
\be\label{rough mt}|m_{w,t}(z)|\le (ct|z|)^{-1/2}.\ee
\end{itemize}
\end{lemma}
\begin{proof}
(i) follows from \cite[Theorem 4.1]{DOZIER2007678}, (ii) and (iii) follow from \cite[Theorem 2.1]{DOZIER20071099} and \cite[Proposition 1]{VLM2012}, and (iv) follows from \cite[Lemma 2.1]{DOZIER20071099}. 
\end{proof}

Denote the support of $\rho_{w,t}$ as $S_{w,t}.$ It is has been shown in \cite{DOZIER20071099,VLM2012} that the support and edges of $S_{w,t}$ can be completely characterized by $m_{w,t}.$  
\begin{lemma}\label{lem_supportproperty} The interior $\mathtt{Int}(S_{w,t})$ of $S_{w,t}$ is given by 
\begin{equation*}
\mathtt{Int}(S_{w,t})=\{x>0: \im m_{w,t}(x)>0  \}=\{x>0: \im \zeta_t(x)>0 \} , 
\end{equation*}
which is a subset of $\R_+$. Moreover, $\zeta_t(x) \notin \{ d_1, \cdots, d_p\}$ when $x \notin \partial S_{w,t}.$ 
\end{lemma}
\begin{proof}
The result was contained in \cite[Propositions 1 and 2]{VLM2012}.
\end{proof}

The following lemma characterizes the right-most edge of $S_{w,t}.$ 
From equation \eqref{eq_keyequation11}, we can solve that
\be\nonumber
m_{w,t}=\frac{m_{w,0}(\zeta_{t}) }{1-c_n tm_{w,0}(\zeta_{t})}.
\ee
Plugging it into \eqref{eq_defnzeta}, we get
\be\label{simplePhizeta}\Phi_t(\zeta_t(z))=z,
\ee
where $\Phi_t$ is an analytic function on $\C_+$ defined as
\begin{equation}\label{eq_subcompansion}
\Phi_t(\zeta)=\zeta(1-c_nt m_{w,0}(\zeta))^2+(1-c_n)t(1-c_n t m_{w,0}(\zeta)),\quad \zeta\in \C_+.
\end{equation} 
In \cite{VLM2012}, the authors characterize the support of $\rho_{\omega,t}$ and its edges using the local extrema of $\Phi_t$ on $ \R$. 

\begin{lemma}\label{lem_eigenvalueslarger} 
Fix any $t>0$. The function $\Phi_t(x)$ on $\R\setminus \{0\}$ admits $2q$ positive local extrema counting multiplicities for some integer $q \geq 1$. The preminages of these extrema are denoted by $\zeta_{1,-}(t)<0<\zeta_{1,+}(t) \leq \zeta_{2,-}(t) \leq \zeta_{2,+}(t) \leq \cdots \leq \zeta_{q,-}(t) \leq \zeta_{q,+}(t),$ and they belong to  the set $\{\zeta \in \mathbb{R}: 1-c_ntm_{w,0}(\zeta_t)>0 \}.$ Moreover, the rightmost edge of $S_{w,t}$ is given by $\lambda_{+,t}=\Phi_t(\zeta_{q,+}(t))$, and $\Phi_t$ is increasing on the intervals $(-\infty, \zeta_{1,-}(t)],\ [\zeta_{1,+}(t), \zeta_{2,-}(t)],\ \cdots, \ [\zeta_{q-1,+}(t), \zeta_{q,-}(t)],$ $[\zeta_{q,+}(t), \infty).$ Finally, for $k=1,\cdots, q,$ each interval $(\zeta_{k,-}(t), \zeta_{k,+}(t))$ contains at least one of the elements of $\{d_1,\cdots, d_p, 0\}$, and $\zeta_{q,-}(t) < d_1 < \zeta_{q,+}(t)$. 
  
\end{lemma}
\begin{proof}
See \cite[Proposition 3]{VLM2012} and the discussion below \cite[Theorem 2]{VLM2012} or \cite[Lemma 1]{loubaton2011}.
\end{proof}

For our purpose, in some cases the equation \eqref{eq_keyequation11} is more convenient to use than \eqref{simplePhizeta}. Now we rewrite \eqref{eq_keyequation11} into a equation of $\zeta_t$ and $z$. 
We focus on $z\in \C_+$ with $\re z>0$. Then we can solve from \eqref{eq_defnzeta} that
\begin{equation}\label{eq_banotherform}
b_t =\frac{t(1-c_n)+\sqrt{t^2(1-c_n)^2+4 \zeta_t z}}{2z},
\end{equation}
where we have chosen the branch of the solution such that Lemma \ref{existuniq} (iv) holds.
Together with (\ref{eq_keyequation11}), we find that the pair $(z, b_t)$ is a solution to (\ref{eq_keyequation11}) if and only if $(z,\zeta_t)$ is a solution to
\begin{equation}\label{final_derivation}
F_t(z, \zeta_t)=0, \quad F_t(z, \zeta_t):=1+\frac{t(1-c_n)-\sqrt{t^2(1-c_n)^2+4 \zeta_t z}}{2 \zeta}-c_n t m_{w,0}(\zeta_t). 
\end{equation}
Since $\Phi_t(\zeta_t(x))=x$ and $F_t(x,\zeta_t)=0$ are the same equation, from Lemma \ref{lem_eigenvalueslarger} we can derive the following characterization of the edges of $S_{w,t}$.

\begin{lemma}\label{lem_partialedge} Denote $a_{k,\pm}(t):= \Phi_t(\zeta_{k,\pm}(t))$, $1\le k \le q$. Then $(a_{k,\pm}(t), \zeta_{k,\pm}(t))$ are real solutions of 
\begin{equation}\label{eq_Ftderivative}
F_t(z,\zeta)=0, \quad \text{and} \quad \frac{\partial F_t}{\partial \zeta}(z,\zeta)=0. 
\end{equation} 
\end{lemma} 
\begin{proof}
By chain rule, if we regard $z$ as a function of $\zeta,$ then we have 
\begin{equation}\label{partialFt}
0=\frac{\dd F_t}{\dd \zeta}=\frac{\partial F_t}{\partial \zeta}+\frac{\partial F_t}{\partial z} z'(\zeta). 
\end{equation}
By Lemma \ref{lem_eigenvalueslarger}, we have $\Phi_t'(\zeta_{k,\pm})=0$ since $a_{k,\pm}$ are local extrema of $\Phi_t$. Then from equation \eqref{simplePhizeta}, we can derive  
$$z'(\zeta_{k,\pm})=\Phi_t'(\zeta_{k,\pm})=0,$$ 
Plugging it into \eqref{partialFt} and using $z(\zeta_{k,\pm}(t))=a_{k,\pm}(t)$ by definition, we get
$$\frac{\partial F_t}{\partial \zeta}(a_{k,\pm}(t), \zeta_{k,\pm}(t))=0,$$
which concludes the proof.
\end{proof}

\begin{remark}
As an application of Lemma \ref{lem_partialedge}, we can use it to derive an expression for the derivative $\partial_t \lambda_{+,t}$ of the right edge, which will be used in \cite{DY20201}. 
Taking derivative of \eqref{final_derivation} with respect to $t$ and using \eqref{eq_Ftderivative}, we get that at $z=\lambda_{+,t}$ and $\zeta_+(t):=\zeta_t(\lambda_{+,t}) $,
\begin{align*}
\frac{\partial F(t,\lambda_{+,t},\zeta_+(t))}{\partial t}  + \frac{\partial F(t,\lambda_{+,t},\zeta_+(t))}{\partial z}\frac{\dd\lambda_{+,t}}{\dd t} =0,
\end{align*}
where we denote $F(t,z,\zeta)\equiv F_t(z,\zeta)$. Thus we can solve that
\begin{align} 
\frac{\dd \lambda_{+,t}}{\dd t}&=\left[\frac{1-c_n}{2\zeta_+(t)}-c_n m_{w,0}(\zeta_+(t)) \right]\sqrt{t^2(1-c_n)^2+4 \zeta_+(t)\cdot \lambda_{+,t}} - \frac{(1-c_n)^2t}{2\zeta_+(t)} \nonumber\\
&=\left[\frac{1-c_n}{2\zeta_+(t)}-\frac{c_n m_{w,t}(\lambda_{+,t})}{b(\lambda_{+,t})} \right]\sqrt{t^2(1-c_n)^2+4 \zeta_+(t)\cdot \lambda_{+,t}} - \frac{(1-c_n)^2t}{2\zeta_+(t)},\label{eq_defnpsiderviative0}
\end{align}
where we used \eqref{eq_keyequation11} in the second step. 
\end{remark}

Our proof will use extensively the following estimates in Lemma \ref{lem_immanalysis}, which are consequences of the regularity assumption in Definition \ref{assumption_edgebehavior}. Define the spectral domains 
\begin{align}
\mathcal{D}&:=\{z=E+\ii \eta: \lambda_+ \leq E \leq \lambda_+ +3c_V/4, \ 2\eta_* \leq \eta \leq 10 \} \nonumber\\
& \cup   \{z=E+\ii \eta: \lambda_+ - 3c_V/4 \leq E \leq \lambda_+, \ \eta_*+\sqrt{\eta_*(\lambda_+-E)} \leq \eta \leq 10 \}, \nonumber\\
&\cup  \{ z=E+\ii \eta:  \lambda_++2\eta_* \leq E \leq \lambda_++3c_V/4, \ 0 \leq \eta \leq 10\}. \label{eq_domain}
\end{align}
\begin{lemma}[Lemma C.1 of \cite{edgedbm}]\label{lem_immanalysis}
Suppose $V$ is $\eta_*$-regular in the sense of Definition \ref{assumption_edgebehavior}. Let $\mu_V$ be the measure associated with $m_{V}$. For any fixed $a \geq 2,$ the following estimates hold for $z=E+\ii \eta \in \mathcal{D}$: if $E\le \lambda_+$, we have 
\begin{equation}\label{eq_boundnear}
\int \frac{\dd \mu_V(x)}{|x-E-\ii \eta|^a} \sim \frac{\sqrt{|E-\lambda_+|+\eta}}{\eta^{a-1}};
\end{equation}
if $E> \lambda_+$, we have 
\begin{equation}\label{eq_boundfarawayregion}
\int \frac{\dd \mu_V(x)}{|x-E-\ii \eta|^a} \sim \frac{1}{(|E-\lambda_+|+\eta)^{a-3/2}}.
\end{equation} 
\end{lemma}
\subsection{Behavior of the contour $\zeta_t(E)$}\label{sec_contour}

In this subsection, we study the behaviors of $\zeta_t(E)$ for $E\in \R$ around the right edge $\lambda_+$ of $V$. Throughout the rest of this section, we assume that $V$ is $\eta_*$-regular, and  
\be\label{restr_t}
t:=n^{-1/3+\omega},  \quad \text{with } \ \  1/3-\phi_*/2-\e/2 \le \omega \le 1/3-\e/2,
\ee
such that $n^\e \eta_*\ll t^2 \le n^{-\e}$. We will not repeat them in the assumptions of our results. 

For simplicity of notations, we shall abbreviate $b_t$ and $\zeta_t$ as $b$ and $\zeta$, respectively. Moreover, we centralize $\zeta$ at the right-most edge $\lambda_+$ of $V$ as 
\begin{equation}\label{eq_xit}
\xi(z)\equiv \xi_t(z) :=\zeta_t(z) -\lambda_{+},  \quad \text{and}\quad \xi_+\equiv \xi_+(t) :=\zeta_+(t) -\lambda_{+},
\end{equation}
where $\zeta_+(t):=\zeta_t(\lambda_{+,t}).$ The following lemma gives a basic estimate on $\xi_+$. 
\begin{lemma}\label{lem rightedge}
For $\xi_+(t)$ defined in (\ref{eq_xit}), we have $\xi_+(t) \geq 0$ and 
\begin{equation}\label{eq_edgebound}
\xi_+(t) \sim t^2. 
\end{equation}
\end{lemma}
\begin{proof}
The statement $\xi_+(t) \geq 0$ follows directly from Lemma \ref{lem_eigenvalueslarger} since $\zeta_+(t) \equiv \zeta_{q,+}(t) \ge d_1 =\lambda_{+}$. For the estimate (\ref{eq_edgebound}), by Lemma \ref{lem_eigenvalueslarger}, we know that $\Phi_t(\zeta_+(t))$ is the only local extrema of $\Phi_t(\zeta)$ on the interval $(d_1,+\infty)$. Hence we have $\Phi'_t(\zeta_+(t))=0$, which gives the equation
\begin{equation} \label{eq_deterministicequation}
(1-c_nt m_{w,0}(\zeta_+))^2 - 2 c_nt m_{w,0}'(\zeta_+) \cdot \zeta_+ \left( 1-c_nt m_{w,0}(\zeta_+)\right) - c_n(1-c_n)t^2m_{w,0}'(\zeta_+)=0.
\ee
From this equation, we can get that
\begin{equation}\label{eq_firstderivative}
 c_nt m_{w,0}'(\zeta_+) = \frac{(1-c_nt m_{w,0}(\zeta_+))^2}{2\zeta_+ \left( 1-c_nt m_{w,0}(\zeta_+)\right)+(1-c_n)t}.
 \ee
By \eqref{rough mt}, we have the bound 
\begin{equation} \label{eq_bnorm}
b=1+\OO(t^{1/2}),
\end{equation} 
which gives $c_nt m_{w,0}(\zeta_+)=\OO(t^{1/2})$ by \eqref{eq_keyequation11}. Plugging it into \eqref{eq_firstderivative}, we obtain that 
\begin{equation} \label{eq_m'norm} m'_{\omega,0}(\zeta_+(t)) \sim t^{-1}.\ee
Together with \eqref{eq_boundfarawayregion}, it implies that $\sqrt{\xi_+(t)} \sim t$. 
\end{proof}


%


For $E \leq \lambda_{+,t},$ $\xi(E)$ has a nonzero imaginary part by Lemma \ref{lem_supportproperty},  and we denote 
\begin{equation}\label{eq_defnzetalambda}
\xi(E):=\alpha(E)+\ii \beta(E).
\end{equation}
We now establish an equation satisfied by $\alpha$ and $\beta$. 
We remark that this equation corresponds to equation (7.12) of \cite{edgedbm}, which takes a much simpler than our equation \eqref{eq_deterministicalphabeta} due to the simple form of additive free convolution.    

\begin{lemma}\label{lemma_dterministicrealtion}  
For any $E\in \R$, $\al\equiv \alpha(E)$ and $\beta \equiv \beta(E)$ satisfy the following equation 
\begin{align}
&1-2c_nt \int  \frac{x \, \dd \mu_{w,0}(x)}{(x-\alpha-\lambda_+)^2+\beta^2}  +c_n^2 t^2  \left[  \left(  \int\frac{x\, \dd \mu_{w,0}(x)}{(x-\alpha-\lambda_+)^2+\beta^2} \right)^2\right. \nonumber \\
& \left.+\frac{1-c_n}{c_n}\int \frac{\dd\mu_{w,0}(x) }{(x-\alpha-\lambda_+)^2+\beta^2}  -\left((\alpha+\lambda_+)^2 + \beta^2\right)\left( \int \frac{\dd \mu_{w,0}(x) }{(x-\alpha-\lambda_+)^2+\beta^2} \right)^2\right]=0. \label{eq_deterministicalphabeta}
\end{align}
where $\mu_{w,0}:=p^{-1}\sum_{i=1}^p\delta_{d_i}$ is the ESD associated with $V$.
\end{lemma}
\begin{proof}
By (\ref{simplePhizeta}), we have that $\Phi_t(\zeta(E))=E.$ By taking the imaginary parts of both sides of this equation and using (\ref{eq_defnzetalambda}), we obtain that  
\begin{align}\label{eq_originalimginary}
&\beta  \left(1-2c_nt \re m_{w,0}(\al+\ii \beta)+c_n^2 t^2 \left[ \left( \re m_{w,0}(\al+\ii \beta)\right)^2-\left(\im m_{w,0}(\al+\ii \beta) \right)^2 \right] \right) \nonumber \\
& + (\alpha+\lambda_+) \left(-2c_nt \im m_{w,0}(\al+\ii \beta) +2c_n^2 t^2 \im m_{w,0}(\al+\ii \beta) \cdot \re m_{w,0}(\al+\ii \beta) \right) \nonumber \\
& -c_n(1-c_n)t^2 \im m_{w,0}(\al+\ii \beta)=0.  
\end{align}
Then after a straightforward calculation using \eqref{eq_defng}, we can conclude  \eqref{eq_deterministicalphabeta}.  
\end{proof}
%
%

Now with Lemma \ref{lemma_dterministicrealtion}, we study the behaviors of $\al$ and $\beta$ for $E$ around $\lambda_+$. The next lemma corresponds to \cite[Lemma 7.1]{edgedbm}, but our proof is slightly different from the proof there. 
\begin{lemma}\label{lem_edgebehavior} 
For $ - 3c_V/4  \le \alpha  \leq \xi_+(t)$ and $c_V/8\le E\le \lambda_{+,t}$, we have 
\begin{equation}\label{betasimal}
\beta \sim t|\alpha-\xi_+(t)|^{1/2}.
\end{equation} 
\end{lemma}
\begin{proof}
Under the given condition, we have $ c_V< \re \zeta=\al + \lambda_+ \le \lambda_+ + \OO(t^2)$. Moreover, by \eqref{eq_keyequation11} and \eqref{rough mt} we have that 
\be\label{c_nt m}
|c_nt m_{w,0}(\zeta_t)|=\left|\frac{c_ntm_{w,t}(\zeta)}{1+c_ntm_{w,t}(\zeta)}\right| \lesssim t^{1/2}.
\ee

In the following proof, for simplicity of notations, we treat function $\Phi_t(\zeta)$ as a function of $\xi=\zeta-\lambda_+$. We consider the following two cases for $\xi$.

\vspace{5pt}
\noindent{\bf Case 1}: $|\xi-\xi_+(t)| \leq \tau t^2$ for some small constant $\tau>0$. By (\ref{eq_edgebound}), for small enough $\tau>0$, we have that $\al \gtrsim t^2\gg \eta_*$. Then using Lemma \ref{lem_immanalysis}, we can get that for any fixed $k\ge 1$,
\begin{equation} \label{eq_derivativecomplex}
|c_nt m_{w,0}^{(k)}(\zeta)| \lesssim t^{-(2k-2)},\quad \text{for} \ \ |\xi-\xi_+| \leq \tau t^2.
\end{equation} 
Moreover, by (\ref{eq_edgebound}) and Lemma \ref{lem_immanalysis}, we have that
\begin{equation}\label{eq_boundederivative}
|c_nt m_{w,0}^{(k)}(\zeta_+)| \sim t^{-(2k-2)}.  
\end{equation}
Note that $m_{w,0}^{(k)}(\xi_+)$ are real numbers for all $k$.

Now for the equation \eqref{simplePhizeta}, we expand $\Phi_t(\xi)$ around $\xi_+$ and get that 
\begin{equation}\label{eq_taylorbasic}
E - \lambda_{+,t}= \Phi_t(\xi)-\Phi_t(\xi_+)=\frac{\Phi_t''(\xi_+)}{2}(\xi - \xi_+)^2+\frac{\Phi_t^{(3)}(\xi_+)}{6}(\xi - \xi_+)^3+\OO\left(t^{-6}|\xi - \xi_+|^4 \right).
\end{equation}
Using \eqref{eq_derivativecomplex}, it is easy to check that
\begin{equation} \label{eq_derivativePhi}
|\Phi_t^{(k)}(\xi)| \lesssim t^{-(2k-2)} .
\end{equation} 
Moreover, we can calculate directly that
\begin{align}
\Phi_t''(\xi)&=- 2c_nt m''_{w,0}(\zeta)\cdot \zeta (1-c_nt m_{w,0}(\zeta_t))- 4c_nt m'_{w,0}(\zeta)\cdot (1-c_nt m_{w,0}(\zeta_t)) \nonumber \\
&+ 2\zeta[c_nt m'_{w,0}(\zeta)]^2 -c_n(1-c_n)t^2m''_{w,0}(\zeta) =- 2 c_nt m''_{w,0}(\zeta)\cdot \zeta  +\OO(t^{-3/2}) ,\label{Phit''0}
\end{align}
where we used \eqref{c_nt m} and \eqref{eq_boundederivative} in the second step. Since 
$$ m''_{w,0}(\zeta_+) = \int\frac{\dd\mu_{w,0}(x)}{(x-\zeta_+)^3}<0,$$
we get that $\Phi_t''(\xi_+)>0$ and
\begin{equation} \label{eq_derivativePhi2}
\Phi_t''(\xi_+) \sim t^{-2} .
\end{equation} 
Now inverting equation \eqref{eq_taylorbasic} and using \eqref{eq_derivativePhi}-\eqref{eq_derivativePhi2}, we obtain that
$$\xi-\xi_+=\sqrt{\frac{2(E-\lambda_{+,t})}{\Phi_t''(\xi_+)}}\left( 1 - \frac{\Phi_t^{(3)}(\xi_+)}{3\Phi_t''(\xi_+)}(\xi-\xi_+)+\OO\left(t^{-4}|\xi - \xi_+|^2 \right)\right).$$
Back-substituting this equation once more we obtain that 
\be\label{Tayloredge}
\xi-\xi_+=\sqrt{\frac{2(E-\lambda_{+,t})}{\Phi_t''(\xi_+)}}\left( 1 - \frac{\Phi_t^{(3)}(\xi_+)}{3\Phi_t''(\xi_+)}\sqrt{\frac{2(E-\lambda_{+,t})}{\Phi_t''(\xi_+)}}+\OO\left(t^{-4}|\xi - \xi_+|^2 \right)\right).
\ee
Taking the real and imaginary part of the above equation and using \eqref{eq_derivativePhi} and \eqref{eq_derivativePhi2}, we obtain that 
\begin{equation}\label{eq_relation1}
|\alpha-\xi_+| \sim |E-\lambda_{+,t}|, \quad \beta \sim t|E-\lambda_{+,t}|^{1/2} \sim t |\alpha-\xi_+|^{1/2}, 
\end{equation}
for $|\xi-\xi_+| \leq \tau t^2.$ 

\vspace{5pt}
\noindent{\bf Case 2}: $-3c_V/4 < \alpha \leq \xi_+ - \tau_1 t^2$ for some small constant $\tau_1>0$. In this case, we have $t |\alpha-\xi_+|^{1/2}\gtrsim t^2$ by (\ref{eq_relation1}) as long as $\tau_1$ is small enough. First suppose that $\beta\gg t |\alpha-\xi_+|^{1/2}$. Then using Lemma \ref{lem_immanalysis} and $|\al|\lesssim |\al-\xi_+|$, we get that
$$ \int \frac{x\dd\mu_{w,0}(x) }{(x-\alpha-\lambda_+)^2+\beta^2} \lesssim \int \frac{ \dd\mu_{w,0}(x)}{(x-\alpha-\lambda_+)^2+\beta^2} =\oo(t^{-1}).$$
This contradicts (\ref{eq_deterministicalphabeta}), so we have $ \beta\lesssim t |\alpha-\xi_+|^{1/2}$. 

On the other hand, suppose $\beta\ll t |\alpha-\xi_+|^{1/2}$. For any small constant $\delta>0$, we take $\beta_0:= \delta t |\alpha-\xi_+|^{1/2}$ and $\zeta_0:=(\al + \lambda_{+}) +\ii \beta_0 $. Then we can check that
$$\frac{\im m_{w,0}(\zeta)}{\beta} = \int \frac{ \dd\mu_{w,0}(x)}{(x-\alpha-\lambda_+)^2+\beta^2} \ge  \int \frac{ \dd\mu_{w,0}(x)}{(x-\alpha-\lambda_+)^2+\beta_0^2}=\frac{\im m_{w,0}(\zeta_0)}{\beta_0}.$$
Then using Lemma \ref{lem_immanalysis}, we can bound 
$$ \frac{\im m_{w,0}(\zeta_0)}{\beta_0} \ge \frac{c_1}{\sqrt{\beta_0}} \ge \frac{c_2}{t\sqrt{\delta}}, \quad \text{if }\ -t^2 \le \al \le  \xi_+ - \tau_1 t^2,$$
and
$$ \frac{\im m_{w,0}(\zeta_0)}{\beta_0} \ge \frac{c_1\sqrt{|\al| + \beta_0}}{\beta_0} \ge \frac{c_2}{t\delta},\quad \text{if }\  \al < - t^2,$$
for some constants $c_1,c_2>0$ that do not depend on $\delta$. 
Now taking the imaginary part of equation  \eqref{final_derivation} gives that
$$\frac{1}{\beta}\im \frac{t(1-c_n)-\sqrt{t^2(1-c_n)^2+4 \zeta E}}{2 \zeta} = \frac{c_n t  \im m_{w,0}(\zeta_t)}{\beta} \ge \frac{c_nc_2}{\sqrt{\delta}}.$$
Using $|\zeta|\ge c_V$ and $t=\oo(1)$, we can bound the left hand side by some constant $C>0$ that does not depend on $\delta$. This gives a contradiction if $\delta$ is taken sufficiently small. Hence we must also have $\beta\gtrsim t |\alpha-\xi_+|^{1/2}$.
\end{proof}

Based on Lemma \ref{lem_edgebehavior}, we are able to prove the following result, which corresponds to \cite[Lemma 7.2]{edgedbm}.
\begin{lemma}\label{eq_edgeclose}
For $ - 3c_V/4  \le \alpha  \leq \xi_+(t)$, we have 
\begin{equation}\label{alEsqrt}
|\alpha(E)-\xi_+(t)| \sim |E-\lambda_{+,t}|. 
\end{equation}
\end{lemma}
\begin{proof}
Note that by (\ref{eq_relation1}), \eqref{alEsqrt} holds true when $|\alpha-\xi_+(t)| \leq \tau t^2.$ To conclude the proof, it suffices to show that $\dd \al/ \dd E \ge 0$ and 
\begin{equation}\label{daldE}
\frac{\dd \alpha}{\dd E} \sim 1, \quad \text{for}\quad  |\alpha-\xi_+| \geq \tau t^2. 
\end{equation} 
From equation \eqref{simplePhizeta}, we obtain that 
\be\label{Cauchyderiv} \frac{\dd \al}{\dd E}= \re \frac1{\Phi_t'(\zeta)} = \frac{\re \Phi_t'(\zeta)}{|\Phi_t'(\zeta)|^2}.\ee
After a tedious but straightforward calculation, we can calculate that
\begin{align}
 \re \Phi_t'(\zeta)&=1+2c_nt\left[ \int \frac{-x}{(x-\alpha-\lambda_+)^2+\beta^2} \dd \mu_{w,0}(x)+ \int \frac{2 \beta^2 x}{\left[ (x-\alpha-\lambda_+)^2+\beta^2 \right]^2} \dd \mu_{w,0}(x) \right] \label{eq_firstpart}\\
& + c_n^2 t^2 \left\{    \left( \int \frac{x-\alpha-\lambda_+ }{(x-\alpha-\lambda_+)^2+\beta^2} \dd \mu_{w,0}(x)\right)^2-\left( \int \frac{\beta }{(x-\alpha-\lambda_+)^2+\beta^2}\dd \mu_{w,0}(x) \right)^2    \right. \nonumber \\
& \left.  + 2 (\alpha+\lambda_+) \int \frac{x-\alpha-\lambda_+}{(x-\alpha-\lambda_+)^2+\beta^2} \dd \mu_{w,0}(x)\cdot \int \frac{(x-\alpha-\lambda_+)^2-\beta^2 }{[(x-\alpha-\lambda_+)^2+\beta^2]^2}\dd \mu_{w,0}(x)\right. \nonumber\\
& \left.  -4 (\alpha+\lambda_+)\int \frac{\beta^2}{(x-\alpha-\lambda_+)^2+\beta^2} \dd \mu_{w,0}(x)\cdot \int \frac{x-\alpha-\lambda_+}{[(x-\alpha-\lambda_+)^2+\beta^2]^2} \dd \mu_{w,0}(x)  \right. \nonumber\\
& \left. -4  \int \frac{ \beta^2 (x-\alpha-\lambda_+) }{[(x-\alpha-\lambda_+)^2+\beta^2]^2}\dd \mu_{w,0}(x)\cdot \int \frac{x-\alpha-\lambda_+}{(x-\alpha-\lambda_+)^2+\beta^2}  \dd \mu_{w,0}(x)\right. \nonumber \\
& \left. -2   \int \frac{\beta^2}{(x-\alpha-\lambda_+)^2+\beta^2} \dd \mu_{w,0}(x)\cdot \int \frac{(x-\alpha-\lambda_+)^2-\beta^2}{[(x-\alpha-\lambda_+)^2+\beta^2]^2}  \dd \mu_{w,0}(x)  \right. \nonumber \\
& \left. -\frac{1-c_n}{c_n} \int \frac{(x-\alpha-\lambda_+)^2-\beta^2}{[(x-\alpha-\lambda_+)^2+\beta^2]^2} \dd \mu_{w,0}(x)
 \right\}. \nonumber
\end{align}
Then using the equation \eqref{eq_deterministicalphabeta}, we can rewrite the above equation as
\begin{align}\label{eq_reducedform}
 \re \Phi_t'(\zeta)=\int \frac{4c_n t \beta^2 x}{\left[ (x-\alpha-\lambda_+)^2+\beta^2 \right]^2} \dd \mu_{w,0}(x)-\mathtt{R},
\end{align}
where $\mathtt{R}$ is defined as 
\be\label{defnR00}
\begin{split}
 \mathtt{R}&:=  4c_n^2 t^2  (\alpha+\lambda_+)\int \frac{\beta^2 }{(x-\alpha-\lambda_+)^2+\beta^2}\dd \mu_{w,0}(x)\cdot \int \frac{x-\alpha-\lambda_+}{[(x-\alpha-\lambda_+)^2+\beta^2]^2}  \dd \mu_{w,0}(x)   \\
&   +4c_n^2 t^2   \int \frac{ \beta^2 x }{[(x-\alpha-\lambda_+)^2+\beta^2]^2}\dd \mu_{w,0}(x)\cdot \int \frac{x-\alpha-\lambda_+}{(x-\alpha-\lambda_+)^2+\beta^2}  \dd \mu_{w,0}(x)  \\
&   +2c_n^2 t^2    \int \frac{\beta^2}{(x-\alpha-\lambda_+)^2+\beta^2} \dd \mu_{w,0}(x)\cdot \int \frac{(x-\alpha-\lambda_+)^2-\beta^2}{[(x-\alpha-\lambda_+)^2+\beta^2]^2}  \dd \mu_{w,0}(x)   \\
&  +2c_n(1-c_n) t^2 \int \frac{(x-\alpha-\lambda_+)^2 }{[(x-\alpha-\lambda_+)^2+\beta^2]^2} \dd \mu_{w,0}(x)=:\mathtt{R}_1+\mathtt{R}_2+\mathtt{R}_3+\mathtt{R}_4 . 
\end{split}
\ee

Now we estimate $ \re \Phi_t'(\zeta)$ using Lemma \ref{lem_immanalysis}. We first consider the case $\al \le \tau_1 t^2$ for some small enough constant $\tau_1>0$. In this case, by \eqref{eq_edgebound} and \eqref{betasimal}, we have that $\zeta = (\al+\lambda_+) + \ii \beta \in \mathcal{D}$ defined in (\ref{eq_domain}). Then using Lemma \ref{lem_immanalysis} and \eqref{betasimal}, we obtain that
\begin{align} 
0< \int \frac{ \dd \mu_{w,0}(x)}{(x-\alpha-\lambda_+)^2+\beta^2} \sim \frac{\sqrt{|\alpha|+\beta}}{\beta} \sim \frac{\sqrt{|\alpha|+t|\alpha-\xi_+(t)|^{1/2}}}{t{|\alpha-\xi_+(t)|}^{1/2}} \sim t^{-1},\label{eq_computationform1}
\end{align}
where we also used $|\al-\lambda_+|+\beta\sim \beta$ for $\al\ge 0$ in the second step, \eqref{betasimal} in the third step, and $\al \le \xi_+(t) - \tau t^2$ by \eqref{eq_edgebound} in the last step. Similarly, we have 
\begin{equation}\label{eq_computationform2}
0<\int \frac{\dd \mu_{w,0}(x)}{\left[ (x-\alpha-\lambda_+)^2+\beta^2 \right]^2}  \sim \frac{\sqrt{|\al|+t|\al-\xi_+(t)|^{1/2}}}{\beta^3}\sim t^{-1}\beta^{-2},
\end{equation}
and
\begin{equation}\label{eq_computationform3}
0<\int \frac{(x-\alpha-\lambda_+)^2}{\left[ (x-\alpha-\lambda_+)^2+\beta^2 \right]^2} \dd \mu_{w,0}(x) \le  \int \frac{ \dd \mu_{w,0}(x)}{(x-\alpha-\lambda_+)^2+\beta^2 } \lesssim  t^{-1}.
\end{equation}
Using \eqref{eq_computationform1}-\eqref{eq_computationform3}, we can bound each term in $\mathtt R$ as:
\begin{align*}
&|\mathtt R_1|\lesssim t^2 \frac{\beta^2}{t}\left[\int \frac{(x-\alpha-\lambda_+)^2}{[(x-\alpha-\lambda_+)^2+\beta^2]^2}  \dd \mu_{w,0}(x)  \right]^{1/2}\left[\int \frac{ \dd \mu_{w,0}(x)  }{[(x-\alpha-\lambda_+)^2+\beta^2]^2} \right]^{1/2}\lesssim \beta \lesssim t,\\
&|\mathtt R_2|\lesssim t^2 \frac{\beta^2}{t \beta^2} \left[\int \frac{(x-\alpha-\lambda_+)^2}{[(x-\alpha-\lambda_+)^2+\beta^2]^2}  \dd \mu_{w,0}(x)  \right]^{1/2}\lesssim t^{1/2},\\
&|\mathtt R_3|\lesssim t^2\cdot \frac{\beta^2}{t}\cdot t^{-1}   \lesssim t^2,\quad |\mathtt R_4|\lesssim t^2\cdot  t^{-1} \le t.
\end{align*}
In sum we get that
$$
|\mathtt R| \le |\mathtt R_1|+ |\mathtt R_2|+|\mathtt R_3|+|\mathtt R_4|\lesssim t^{1/2} .
$$
On the other hand, by \eqref{eq_computationform2} we have
$$0< \int \frac{4c_n t \beta^2 x}{\left[ (x-\alpha-\lambda_+)^2+\beta^2 \right]^2} \dd \mu_{w,0}(x) \sim  1.$$
Hence \eqref{eq_reducedform} gives that $ \re \Phi_t'(\zeta) \sim 1$, which also gives a lower bound $|\Phi_t'(\zeta)| \ge \re \Phi_t'(\zeta) \gtrsim 1$. For an upper bound of $|\Phi_t'(\zeta)|$, we have
$$ |\Phi_t'(\zeta)|=\left|(1-c_nt m_{w,0}(\zeta_t))^2 - 2 c_nt m_{w,0}'(\zeta)\cdot \zeta \left( 1-c_nt m_{w,0}(\zeta_t)\right) - c_n(1-c_n)t^2m_{w,0}'(\zeta)\right|\lesssim 1, $$ 
using \eqref{rough mt} and  
$$ c_n t |m_{w,0}'(\zeta)| \le c_nt\int\frac{\dd\mu_{w,0}(x) }{(x-\al-\lambda_+)^2 + \beta^2}\sim 1$$
by \eqref{eq_computationform1}. This concludes \eqref{daldE} for $\al \le \tau_1 t^2$.

For the case $\tau_1 t^2 \le \al \le \xi_+(t) - \tau t^2$, the proof is similar except that we shall use \eqref{eq_boundfarawayregion} to estimate each term.
\end{proof}

\subsection{Behavior of $\zeta_t(z)$ on general domain}

In this subsection, we first extend the result of Lemma \ref{lem_edgebehavior} to $\xi(z)= \al(z) + \ii \beta(z)$ for complex $z= E+ \ii \eta$ around the right edge $\lambda_+$. 
In the proof, we will regard $\al$ and $\beta$ as functions of $E$ and $\eta$. First, we claim the following simple estimate.
\begin{lemma}
Suppoe $-3c_V/4 \le \al \le \xi_+(t) - \tau t^2$ for some constant $\tau>0$, $|E-\lambda_+| \leq c_V/2$ and $0 \leq \eta \leq 10$. Then we have 
\be\label{claim betasim}
\beta (E,\eta)\ge c_\tau t |E - \lambda_+|^{1/2}
\ee
for some constant $c_\tau>0$.
\end{lemma}
\begin{proof}
First, taking the imaginary parts of both sides of \eqref{eq_defnzeta}, we get
\begin{align}\label{betaeta} 
\beta=\eta \re \left( 1+c_nt m_{w,t}(z)\right)^2 + E\left( 2c_nt \im m_{w,t}(z) + c^2t^2 \im m^2_{w,t}(z)\right) - c_n(1-c_n)t^2 \im m_{w,t}(z) .
\end{align}
On the other hand, taking the imaginary part of equation  \eqref{final_derivation}, we can get that
$$ t \im m_{w,t}(z) \lesssim \beta + \eta.$$
Plugging it into \eqref{betaeta} and using \eqref{rough mt}, we get that for some constant $C>0$,
\begin{align}\label{betaeta2} 
\beta\ge \eta  (1+\OO(t^{1/2})) - Ct\left(\beta+\eta\right)\Rightarrow \beta\ge \frac12\eta. 
\end{align}

The rest of the proof is similar to Case 2 in the proof of Lemma \ref{lem_edgebehavior}. Suppose $\beta\ll t |\alpha-\xi_+|^{1/2}$. For any small constant $\delta>0$, we take $\beta_0:= \delta t |\alpha-\xi_+|^{1/2}$ and $\zeta_0:=(\al + \lambda_{+}) +\ii \beta_0 $. Then we have
$$\frac{\im m_{w,0}(\zeta_t)}{\beta} \ge \frac{\im m_{w,0}(\zeta_0)}{\beta_0} \ge \frac{c_2}{t\sqrt{\delta}}$$
for some constant $c_2>0$ that does not depend on $\delta$. Then taking the imaginary part of equation  \eqref{final_derivation}, we get that for some constant $C>0$ independent of $\delta$,
$$\frac{c_2}{\sqrt{\delta}} \le \frac{t  \im m_{w,0}(\zeta_t)}{\beta}  =  \frac{1}{c_n\beta}\im \frac{t(1-c_n)-\sqrt{t^2(1-c_n)^2+4 \zeta z}}{2 \zeta} \le \frac{C}{c_n}\frac{\beta+\eta}{\beta}\le \frac{3C}{c_n},$$
where we used \eqref{betaeta2} in the last step. This gives a contradiction if $\delta$ is taken sufficiently small. Hence we must have $\beta\gtrsim t |\alpha-\xi_+|^{1/2}$.
\end{proof}

Then we prove the following estimate. 
\begin{lemma}\label{lem_derivativecontrol} 
For $|E-\lambda_+| \leq c_V/2$ and $0 \leq \eta \leq 10,$ we have 
\be\label{Phit'0}|\Phi_t'(\zeta)| \sim \min \left\{1, \frac{|\al-\xi_+(t)|+ \beta}{t^2} \right\}.\ee
\end{lemma} 
\begin{proof}
We first assume that $|\alpha-\xi_+(t)|+\beta  \leq c_{1} t^2$ for some small constant $c_1>0$. In this case, applying the mean value theorem to $\Phi'_t(\xi)$ we get 
\begin{equation}\label{eq_needtobound}
\Phi_t'(\xi)=\Phi_t^{''}(\xi_+(t))(\xi-\xi_+(t))\cdot\left[1+\OO(t^{-2}|\xi-\xi_+(t)|)\right].
\end{equation}
where we used \eqref{eq_derivativePhi}, \eqref{eq_derivativePhi2} and $\Phi_t'(\xi_+)=0$. 
Hence for a small enough $c_1$, we get 
$$ |\Phi_t'(\xi)| \sim |\Phi_t^{''}(\xi_+(t))||\xi-\xi_+(t)| \sim t^{-2}|\xi - \xi_+(t)| .$$

It remains to prove that $|\Phi_t'(\xi) |\sim 1$ when $|\alpha-\xi_+(t)|+\beta  > c_{1} t^2$. The proof is based on a careful analysis of $\re \Phi_t'(\xi)$. First we observe that (\ref{eq_firstpart}) still holds for general $z= E+\ii \eta$. On the other hand, the right-hand side of equation (\ref{eq_originalimginary}) is now $\eta$, and hence (\ref{eq_reducedform}) becomes
\begin{align}\label{eq_reducedform22}
 \re \Phi_t'(\zeta)=\frac{\eta}{\beta}+\int \frac{4c_n t \beta^2 x}{\left[ (x-\alpha-\lambda_+)^2+\beta^2 \right]^2} \dd \mu_{w,0}(x)-\mathtt{R},
\end{align}
We will estimate $ \re \Phi_t'(\zeta)$ using \eqref{eq_firstpart} and \eqref{eq_reducedform22}. 



\vspace{5pt}

\noindent{\bf Case 1:} We first consider the case where $\al \le \xi_+(t) -\tau t^2$ for some constant $\tau>0$. Together with \eqref{claim betasim}, we see that $\zeta=(\al+\lambda_+)+\ii \beta\in \cal D$. Then with the same arguments as in the proof of Lemma \ref{eq_edgeclose}, we can derive from \eqref{eq_firstpart} and \eqref{eq_reducedform22} that
\begin{align}\label{eq_firstpart2222}
 \re \Phi_t'(\zeta)&=1+2c_nt\left[ \int \frac{-x}{(x-\alpha-\lambda_+)^2+\beta^2} \dd \mu_{w,0}(x)+ \int \frac{2 \beta^2 x}{\left[ (x-\alpha-\lambda_+)^2+\beta^2 \right]^2} \dd \mu_{w,0}(x) \right] +\oo(1),
\end{align}
and 
\begin{align}\label{eq_reducedform2222}
 \re \Phi_t'(\zeta)=\frac{\eta}{\beta}+\int \frac{4c_n t \beta^2 x}{\left[ (x-\alpha-\lambda_+)^2+\beta^2 \right]^2} \dd \mu_{w,0}(x) +\oo(1).
\end{align}
Then using \eqref{eq_boundnear}, we get that
$$ 2c_nt \int \frac{x}{(x-\alpha-\lambda_+)^2+\beta^2} \dd \mu_{w,0}(x)\sim \frac{t\sqrt{|\al| + \beta}}{\beta}=:Q,$$
and
$$4c_n t\int \frac{ \beta^2 x}{\left[ (x-\alpha-\lambda_+)^2+\beta^2 \right]^2} \dd \mu_{w,0}(x) \sim t \beta^2\frac{\sqrt{|\al| + \beta}}{\beta^3}=Q.$$
Inserting these two estimates into \eqref{eq_firstpart2222} and \eqref{eq_reducedform2222}, we obtain that for some constants $c_1,C_1>0$,
$$ \re \Phi_t'(\zeta)\ge \max\{1- C_1Q, c_1Q\} \gtrsim 1,$$
which gives a lower bound for $|\Phi_t'(\zeta)| $. Finally, using \eqref{eq_boundnear} and \eqref{claim betasim}, it is easy to check that $|\Phi_t'(\zeta)|\lesssim 1$. Hence we obtain the estimates $\re\Phi'_t(\zeta)\sim |\Phi_t'(\zeta)| \sim 1 $. 

\vspace{5pt}

\noindent{\bf Case 2:} Second, we assume that $|\alpha-\xi_+(t)|+\beta \geq C_1 t^2$ for some large constant $C_1>0$ and $\al \ge \xi_+(t) -\tau t^2$. Then $\zeta=(\al+\lambda_+)+\ii \beta\in \mathcal{D}$ with $\al \gtrsim t^2$ by (\ref{eq_edgebound}). Hence using \eqref{eq_boundfarawayregion}, we get that 
\begin{equation*}
t \int \frac{2c_nx }{(x-\alpha-\lambda_+)^2+\beta^2}  \dd \mu_{w,0}(x) \ge c_1\frac{t}{\sqrt{\alpha+\beta}}
\end{equation*}
for some constant $c_1>0$ that does not depend on $C_1$. On the other hand, we have
\begin{align}
0\le t \int \frac{4c_n \beta^2 x}{\left[ (x-\alpha-\lambda_+)^2+\beta^2 \right]^2} \dd \mu_{w,0}(x) \le C'\frac{t \beta^2}{(\alpha+\beta)^{5/2}} \le C' \frac{t}{\sqrt{\alpha+\beta}}
\end{align}
for some constant $C'>0$ that does not depend on $C_1$. Therefore, we conclude that as long as $C_1$ is chosen large enough, then 
$$\frac12\le 1+t\left[ \int \frac{-2c_n x}{(x-\alpha-\lambda_+)^2+\beta^2} \dd \mu_{w,0}(x)+ \int \frac{4c_n \beta^2 x}{\left[ (x-\alpha-\lambda_+)^2+\beta^2 \right]^2} \dd \mu_{w,0}(x) \right] \le \frac32.$$
Moreover, using \eqref{eq_boundfarawayregion} we can readily show that the rest of the terms on the right-hand side of (\ref{eq_firstpart}) are all of order $\oo(1)$, and that $|\Phi_t'(\zeta)|\lesssim 1$. We omit the details since they are similar to the arguments in the proof of Lemma \ref{eq_edgeclose}. In sum, we obtain that $\re\Phi'_t(\zeta)\sim |\Phi_t'(\zeta)| \sim 1 $ for the current case. 


\vspace{5pt}

\noindent{\bf Case 3:} It remains to consider the case $ c_{1} t^2 \le |\alpha-\xi_+(t)|+\beta \leq C_{1} t^2$ and $\al \ge \xi_+(t) -\tau t^2$.  If $|\alpha-\xi_+(t)|\le c_1 t^2/2$, we have $\beta\ge c_1 t^2/2$. Then using Lemma \ref{lem_immanalysis}, one can check that $|\Phi_t'(\zeta)|=\OO(1)$, \eqref{eq_reducedform2222} still holds, and  
$$4c_n t\int \frac{ \beta^2 x}{\left[ (x-\alpha-\lambda_+)^2+\beta^2 \right]^2} \dd \mu_{w,0}(x) \sim t \frac{\sqrt{|\al| + \beta}}{\beta}\sim 1.$$
Thus we get $\re\Phi'_t(\zeta)\sim |\Phi_t'(\zeta)| \sim 1 $.

In the above proof, we can take the constants such that $\tau \le c_1/2 $. Then we are only left with the regime $\al \ge \xi_+(t)+c_1 t^2/2$ and $ c_{1} t^2 \leq |\alpha-\xi_+(t)|+\beta \leq C_{1} t^2$. In this regime, we have $\zeta=(\al+\lambda_+)+\ii \beta\in \mathcal{D}$ with $\al \gtrsim t^2$ by (\ref{eq_edgebound}). Then using \eqref{eq_boundfarawayregion} and the same arguments as in the proof of Lemma \ref{eq_edgeclose}, we can check that $|\Phi_t'(\zeta)|=\OO(1)$, \eqref{eq_reducedform2222} still holds, and that
$$4c_n t\int \frac{ \beta^2 x}{\left[ (x-\alpha-\lambda_+)^2+\beta^2 \right]^2} \dd \mu_{w,0}(x) \sim t \frac{\beta^2}{(|\al| + \beta)^{5/2}}\sim 1.$$
Thus we get $\re\Phi'_t(\zeta)\sim |\Phi_t'(\zeta)| \sim 1 $. This completes the proof.
\end{proof}

Armed with Lemma \ref{lem_derivativecontrol}, we can prove the following estimates. Recall the notation in \eqref{defn kappa}.

\begin{lemma} \label{lem_xigeneral}
If $\kappa+\eta \leq \tau_1 t^2$ for some sufficiently small constant $\tau_1>0,$ then we have 
\begin{equation}\label{eq_approximate}
t \sqrt{\kappa+\eta} \sim |\xi-\xi_+(t)|,
\end{equation}
which also implies that 
\be\label{Phit'}|\Phi_t'(\zeta)| \sim \min \left\{1, \frac{\sqrt{\kappa+\eta}}{t} \right\}.\ee
In the region $|\kappa+\eta| \geq \tau t^2$ for any constant $\tau>0$, we have 
\begin{equation}\label{eq_derivativebound}
\frac{\partial \alpha}{\partial E} = \frac{\partial \beta}{\partial \eta} \sim 1,
\end{equation}
and
\begin{equation}\label{eq_absbound}
\left| \frac{\partial \alpha}{\partial \eta} \right|=\left| \frac{\partial \beta}{\partial E} \right| \lesssim 1. 
\end{equation}
The above two estimates imply that 
\begin{equation}\label{eq_ccc}
|\alpha|+|\alpha-\xi_+(t)|+\beta \lesssim t^2+\eta+\kappa.
\end{equation}
\end{lemma}
\begin{proof}
If $|\xi -\xi_+(t)| \le c_1t^2$ for some small constant $c_1>0$, then with the same Taylor expansion argument as in the proof of Lemma \ref{lem_edgebehavior}, we can obtain that (recall \eqref{Tayloredge})
\be\label{Tayloredge2}
\xi-\xi_+=\sqrt{\frac{2(z-\lambda_{+,t})}{\Phi_t''(\xi_+)}}\left( 1 - \frac{\Phi_t^{(3)}(\xi_+)}{3\Phi_t''(\xi_+)}\sqrt{\frac{2(z-\lambda_{+,t})}{\Phi_t''(\xi_+)}}+\OO\left(t^{-4}|\xi - \xi_+|^2 \right)\right).
\ee
As long as $c_1$ is small enough, we have
\be\label{eq_smallrelation1} 
|\xi-\xi_+|\sim \sqrt{\frac{|z-\lambda_{+,t}|}{\Phi_t''(\xi_+)}} \sim t|z-\lambda_{+,t}|^{1/2} \sim t\sqrt{\kappa+\eta}.
\ee
where we used the estimates \eqref{eq_derivativePhi} and \eqref{eq_derivativePhi2}. Moreover, with \eqref{Tayloredge2} one can observe that there exists a constant $\tau_1>0$ such that $|\Phi_t^{-1}(z) -\xi_+(t)|\le c_1 t^2$ for all $z$ with $\kappa+\eta\le \tau_1 t^2$. It then concludes \eqref{eq_approximate} together with \eqref{eq_smallrelation1}. Moreover, inserting \eqref{eq_approximate} into \eqref{Phit'0} we get \eqref{Phit'}.

Now we consider the region $|\kappa+\eta| \geq \tau t^2$ for some constant $\tau>0$. By \eqref{eq_approximate}, we have $|\xi -\xi_+(t)| \ge c_\tau t^2$ for some constant $\tau>0$. Moreover, in the proof of Lemma \ref{lem_derivativecontrol} we have shown that $\re\Phi'_t(\zeta)\sim |\Phi_t'(\zeta)| \sim 1 $. Together with 
\eqref{Cauchyderiv}, we get \eqref{eq_derivativebound}, where the first equality comes from Cauchy-Riemann equation. Similarly, we have 
$$ \left| \frac{\partial \alpha}{\partial \eta} \right|=\left| \frac{\partial \beta}{\partial E} \right| \le \frac{1}{|\Phi_t'(\zeta)|} \lesssim 1,$$
which gives \eqref{eq_absbound}. Finally \eqref{eq_ccc} is an easy consequence of \eqref{eq_derivativebound} and \eqref{eq_absbound}.
\end{proof}

\begin{remark}
Besides \eqref{eq_approximate}, we will also use the expansion \eqref{Tayloredge2}, which gives more detailed behavior of $\xi$ near $\xi_+$. In the following proof, whenever we refer to Lemma \ref{lem_xigeneral}, it also includes \eqref{Tayloredge2}.
\end{remark}

Next we collect some useful estimates that are needed in the proof of the local laws. They are established on different spectral domains. 

\begin{lemma}\label{lem_domian1control} 
Fix any constant $C_1>0$. For $z\in \cal D_{\vartheta}$ with $E \leq \lambda_{+,t}+C_1 t^2$, the following estimates hold: 
\begin{equation}\label{eq_dtheta1control}
\min_{i=1}^p |d_i-\zeta| \gtrsim t^2+\eta+t \im m_{w,t} ; 
\end{equation}
\begin{equation}\label{eq_dtheta1extension}
\int \frac{\dd \mu_{w,0}(x)}{|x-\zeta|^a} \lesssim \frac{t+\sqrt{\kappa+\eta}}{(t^2+\im \zeta)^{a-1}}, \quad \text{for any fixed }\ a \geq 2;
\end{equation}
\begin{equation}\label{kappaeta_d1}
t+\sqrt{\kappa+\eta} \lesssim t+\im m_{w,t}. 
\end{equation}
\end{lemma}
\begin{proof}
Using \eqref{eq_defnzeta} and \eqref{rough mt}, we can check that
\begin{equation}\label{zetasimeta+t2} 
\beta=\im \zeta = \im \left[(1+c_nt m_{w,t}(z))^2 z-t(1-c_n)(1+c_nt m_{w,t}(z))\right] \sim \eta+t \im m_{w,t}.
\end{equation}
Since $d_i$'s are real values, we get that $|d_i-\zeta| \geq \im \zeta \gtrsim \eta+t \im m_{w,t}.$
Thus to show \eqref{eq_dtheta1control}, it remains to show that 
\begin{equation}\label{eq_dtheta1controlpf}
\min_{i=1}^p |d_i-\zeta| \geq \tau t^2  
\end{equation}
for some constant $\tau>0$. If $E-\lambda_{+,t} \ge c' t^2$ for some constant $c'>0$, then by Lemma \ref{lem_xigeneral} we have $\al(z) - \xi_+(t) \gtrsim t^2$. Hence we get 
$$\min_{i=1}^p |d_i-\zeta| \ge \min_{i=1}^p |d_i-\al | \gtrsim t^2$$
using (\ref{eq_edgebound}) and that $d_i\le \lambda_+ $. Finally, we are only left with the case $\im \zeta \leq c_1 t^2$ and $E\le \lambda_{+,t} + c_1 t^2$ for some small constant $c_1>0$. In this case, we claim that 
\be\label{check_kappaeta1}
\kappa\leq C c_1 t^2
\ee
for some constant $C>0$ that does not depend on $c_1$. If \eqref{check_kappaeta1} does not hold, then by Lemma \ref{lem_xigeneral} we must have $\im \zeta \ge C'\sqrt{Cc_1}t^2$ for some constant $C'$ that does not depend on $C$ and $c_1$, which gives a contradiction for large enough $C$. Now given \eqref{check_kappaeta1}, we can choose $c_1>0$ small enough such that $|\xi-\xi_t| \leq \tau_1 t^2$ for any small constant $\tau_1>0$ by \eqref{eq_approximate}. Together with (\ref{eq_edgebound}) and $d_i^2\le \lambda_+$, we conclude that 
$$|d_i^2-\zeta|=|d_i^2-(\lambda_+ +\xi_{+,t})-(\xi-\xi_{+,t})| \geq |\xi_+(t)|- |\xi-\xi_{+,t}|\ge \tau t^2.$$

For (\ref{eq_dtheta1extension}), we first suppose that $\im \zeta \leq \tau t^2$ for a small constant $\tau>0.$ As shown in the above proof of (\ref{eq_dtheta1control}), we must have that $\alpha \geq \tau t^2$ as long as $\tau$ is chosen sufficiently small. Thus by (\ref{eq_boundfarawayregion}), we get 
\begin{equation*}
\int \frac{\dd \mu_{w,0}(x)}{|x-\zeta|^a} \lesssim  \frac{1}{(t^2+\im \xi)^{a-3/2}} \lesssim \frac{t}{(t^2+\im \zeta)^{a-1}}. 
\end{equation*}  
Then we consider the case $\im \zeta \geq \tau t^2.$ In this case, (\ref{eq_boundnear}) always serves as an upper bound regardless of the sign of $\al$, and hence
\begin{equation*}
\int \frac{\dd \mu_{w,0}(x)}{|x-\zeta|^a} \lesssim \frac{\sqrt{|\alpha|+\beta}}{\beta^{a-1}} \leq \frac{\sqrt{|\alpha|+\beta}}{(t^2+\im \zeta)^{a-1}}.
\end{equation*}
Then we conclude (\ref{eq_dtheta1extension}) using (\ref{eq_ccc}). 

Finally, the last estimate \eqref{kappaeta_d1} follows from a simple calculation using the fact that $m_{w,t}$ is a Stieltjes transform of a density with square root behavior near the right edge; see Lemma \ref{lem_asymdensitysquare} below.  
\end{proof}


\begin{lemma}\label{lem_domian2control} 
Fix any constant $C_1>0$. For $z\in \cal D_{\vartheta}$ with $E \ge \lambda_{+,t}+C_1 t^2$, the following estimates hold: 
\begin{equation}\label{eq_dtheta2bound1}
\min_{i=1}^p |d_i-\zeta| \gtrsim t^2+\kappa+\eta;
\end{equation} 
\begin{equation}\label{eq_dtheta2bound2}
\int \frac{\dd \mu_{w,0}(x)}{|x-\zeta|^a} \lesssim \frac{1}{(\kappa+\eta+t^2)^{a-3/2}},\quad   \text{for any fixed }\ a \geq 2.
\end{equation}
\end{lemma}
\begin{proof}
From the proof of Lemma \ref{lem_domian1control}, we have seen that $|d_i-\zeta|\gtrsim t^2$. Moreover, since $\im \zeta \gtrsim \eta$, we have $|d_i-\zeta|\gtrsim \eta$. Finally, if $E\ge \lambda_{+,t} + C\eta+t^2$ for some large constant $C>0$, then by Lemma \ref{lem_xigeneral} we have $\al \ge c_1\kappa$ for some constant $c_1>0$. Hence using $d_i\le \lambda_+ \le\lambda_{+,t}$, we get 
$$\min_{i=1}^p |d_i-\zeta| \ge \al \gtrsim \kappa.$$
This concludes \eqref{eq_dtheta2bound1}.
%
%
For (\ref{eq_dtheta2bound2}), if $E\ge \lambda_{+,t} + C\eta+t^2$, then by (\ref{eq_boundfarawayregion}) we have
\begin{equation*}
\int \frac{\dd \mu_{w,0}(x)}{|x-\zeta|^a} \lesssim \frac{1}{\kappa^{a-3/2}} \lesssim  \frac{1}{(\kappa+\eta+t^2)^{a-3/2}}.
\end{equation*} 
On the other hand, suppose $C_1 t^2\le E - \lambda_{+,t}\le   C\eta+t^2$. Then we have $\al\gtrsim t^2$ by Lemma \ref{lem_xigeneral} and $\beta \gtrsim \eta$ by \eqref{zetasimeta+t2}. Thus using (\ref{eq_boundfarawayregion}),  we get  
\begin{equation*}
\int \frac{\dd \mu_{w,0}(x)}{|x-\zeta|^a} \lesssim \frac{1}{(t^2+\eta)^{a-3/2}} \lesssim  \frac{1}{(\kappa+\eta+t^2)^{a-3/2}},
\end{equation*} 
where we used $\kappa\le C\eta+t^2$ in the second step. This concludes  \eqref{eq_dtheta2bound2}.
\end{proof}

\begin{lemma}\label{lem_domian12control} 
 If $\kappa+\eta \leq \tau_1 t^2$ for some sufficiently small constant $\tau_1>0,$ then we have 
\begin{equation}\label{eq_dtheta12bound}
|m_{w,0}''(\zeta)|\sim t^{-3} .
\end{equation}
\end{lemma}
\begin{proof}
 If $\kappa+\eta \leq \tau_1 t^2$, then by \eqref{eq_approximate} we have $\zeta\in \cal D$ with $\al \gtrsim t^2$ and $|\xi|\sim t^2$. Thus using \eqref{eq_boundfarawayregion} we get
$$ \int \frac{\dd \mu_{w,0}(x)}{|x-\zeta|^3}\sim t^{-3}.$$ 
 Furthermore, by \eqref{eq_approximate} we have $\al \ge C_1\sqrt{\tau_1}\im \zeta $ for some constant $C_1>0$ that does not depend on $\tau_1$. As long as $\tau_1$ is taken sufficiently small, we have $-\re {(x-\zeta)^{-3}}\sim {|x-\zeta|^{-3}}$ for all $x\in \supp(\mu_{w,0}).$  Thus we get
 $$|m_{w,0}''(\zeta)| = 2\left|\int \frac{\dd \mu_{w,0}(x)}{(x-\zeta)^3}\right|\sim  \int \frac{\dd \mu_{w,0}(x)}{|x-\zeta|^3}\sim t^{-3},$$
which concludes the proof.
\end{proof}

\subsection{Qualitative properties of $\rho_{w,t}$ and $m_{w,t}$}\label{sec_properties}

%

The following lemma 
describes the square root behavior of $\rho_{w,t}$ around the edge $\lambda_{+,t}$. 
\begin{lemma}\label{lem_asymdensitysquare} For $|E-\lambda_{+,t}| \leq 3c_V/4,$ the asymptotic density $\rho_{w,t}$ satisfies 
\begin{equation}\label{sqrtdensity}
\rho_{w,t}(E) \sim \sqrt{(\lambda_{+,t}-E)_+}  .
\end{equation}
Moreover, for $-\tau t^2 \le E-\lambda_{+,t} \leq 0$ for some sufficiently small constant $\tau>0$, we have 
\begin{equation}\label{sqrtdensity2}
\rho_{w,t}(E)=\frac{1}{\pi}\sqrt{\frac{2(\lambda_{+,t}-E)}{[4\lambda_{+,t}\xi_+(t) + (1-c_n)^2 t^2] c_n^2 t^2\Phi^{''}(\xi_+(t))}} \left(1+\OO\left(\frac{|E-\lambda_{+,t}|}{t^2} \right)\right).
\end{equation}
Recall that by \eqref{eq_derivativePhi2} we have $t^2\Phi^{''}(\xi_+(t))| \sim 1. $
\end{lemma} 
\begin{proof}
By \eqref{zetasimeta+t2}, we have 
\be\label{betaEsimf}
\beta(E)=\im \zeta(E)\sim t\im m_{w,t}(E) =t \pi  \rho_{w,t}(E). 
\ee
Then \eqref{sqrtdensity} follows from \eqref{betasimal} and \eqref{alEsqrt}. For \eqref{sqrtdensity2}, we use \eqref{eq_banotherform} and $b(z)=1+c_nt m_{w,t}(z)$ to get that
\begin{equation}\label{mwt_zetat}
m_{w,t}(E)=\frac{t(1-c_n)+\sqrt{t^2(1-c_n)^2+4 (\al(E) + \ii \beta(E) + \lambda_+) E}}{2Ec_nt} -\frac{1}{c_nt} ,
\end{equation}
Taking the imaginary part of \eqref{mwt_zetat} and using \eqref{Tayloredge}, we can conclude \eqref{sqrtdensity2}. 
\end{proof}

Lemma \ref{lem_asymdensitysquare} immediately implies the following estimates on $\im m_{w,t}.$
\begin{lemma} \label{lem_asymdensitysquare2}
We have the following estimates for $z= E+\ii\eta$ with $\lambda_{+,t} - 3c_V/4 \le E\le \lambda_{+,t} + 3c_V/4$ and $0\le \eta \le 10$: 
\begin{equation}\label{eq_imasymptoics}
|m_{w,t}(z)|\lesssim 1,\quad  \im m_{w,t}(z) \sim
\begin{dcases}
\sqrt{\kappa+\eta}, & \lambda_{+,t} - 3c_V/4 \le  E\leq \lambda_{+,t} \\
\frac{\eta}{\sqrt{\kappa+\eta}}, &\lambda_{+,t}\le  E  \le \lambda_{+,t} + 3c_V/4
\end{dcases}.
\end{equation} 
\end{lemma}
\begin{proof}
\eqref{eq_imasymptoics} can be derived easily from \eqref{def mwt} combined with the square root behavior of $\rho_{w,t}$ in \eqref{sqrtdensity}. 
\end{proof}

We also need to control the derivative $\partial_z m_{w,t}(z)$. First by \eqref{def mwt}, we have the trivial estimate
\begin{equation}\label{eq_trivialbounderivative}
\left| \partial_z m_{w,t}(z) \right| = \left|\int \frac{\dd \mu_{w,t}(x)}{(x-z)^2}\right|\leq \frac{\im m_{w,t}}{\eta}.
\end{equation}
Moreover, we claim the following estimates.

\begin{lemma} \label{lem_partialm}
For $\kappa+\eta \leq t^2,$ we have 
\begin{equation}\label{eq_bound1}
\left| \partial_z m_{w,t}(z) \right| \lesssim (\kappa+\eta)^{-1/2}.
\end{equation}
Moreover, if $\kappa+\eta \geq t^2$,  we have that for $E \geq \lambda_{+,t},$ 
\begin{equation}\label{eq_bound2}
|\partial_z m_{w,t}(z)|   \lesssim (\kappa+\eta)^{-1/2},
\end{equation}
and for $E \leq \lambda_{+,t}$, 
\begin{equation}\label{eq_bound3}
|\partial_z m_{w,t}(z)| \lesssim \frac{ \sqrt{\kappa+\eta}}{t \sqrt{\kappa+\eta}+\eta}. 
\end{equation}
\end{lemma}
\begin{proof}
By equation \eqref{simplePhizeta}, we have $\partial_z \zeta = [\Phi_t'(\zeta)]^{-1}$. Then using the definition of $\zeta$ in \eqref{eq_defnzeta}, we can solve that
$$\partial_z m_{w,t}(z)=\frac{[\Phi_t'(\zeta)]^{-1}- b^2 }{\left[ 2bz - (1-c_n)t\right]c_nt}. $$
Then using \eqref{Phit'}, we get that
\begin{equation}\label{pfpartialm}
\left| \partial_z m_{w,t}(z) \right| \lesssim \max \left\{ \frac{1}{\sqrt{\kappa+\eta}}, \frac{1}{t} \right\},
\end{equation}
 which concludes \eqref{eq_bound1} for $\kappa+\eta \leq t^2$. 
The bound (\ref{eq_bound2}) follows directly from (\ref{eq_imasymptoics}) and (\ref{eq_trivialbounderivative}). Similarly, with (\ref{eq_imasymptoics}) and (\ref{eq_trivialbounderivative}), we can get \eqref{eq_bound3} when $t\sqrt{\kappa +\eta}\le \eta$. If $t\sqrt{\kappa +\eta}\ge \eta$, we use \eqref{pfpartialm} 
to get
$$\left| \partial_z m_{w,t}(z) \right| \lesssim \max \left\{ \frac{1}{\sqrt{\kappa+\eta}}, \frac{1}{t} \right\}\lesssim \frac{ \sqrt{\kappa+\eta}}{t \sqrt{\kappa+\eta}+\eta}.$$
This concludes (\ref{eq_bound3}). 
\end{proof}

\section{Proof of Theorem \ref{thm_local}}\label{sec largelocal}

Since the multivariate Gaussian distribution is rotationally invariant under orthogonal transforms, for any $t>0$ we have the following equality in distribution 
\be\label{Yt=Wt}Y_t=Y+\sqrt{t}X\stackrel{d}{=}O_1W_t O_2^\top,\quad W_t:= W+\sqrt{t}X,\ee
where we used the SVD of $Y$ in \eqref{SVD}. As in Definition \ref{resol_not}, we define the following resolvents 
 \begin{equation}\label{eqn_defGd}
R(z)\equiv R(W_t , z):=(z^{1/2}\wt H_t - z)^{-1}, \quad \wt H_t (W_t):= \begin{pmatrix} 0 & W_t \\ W_t^\top & 0\end{pmatrix},
\end{equation}
and
\begin{equation}\label{def_greend}
  \mathcal R\equiv \cal R(W_t,z) :=\left(W_tW_t^\top -z\right)^{-1} , \ \ \ \uR \equiv \uR(W_t,z):=\left(W_t^\top W_t-z\right)^{-1} .
\end{equation}
With slight abuse of notations, we still denote 
\be\label{defn_md}
m(z)\equiv m(W_t,z):= \frac{1}{p} \mathrm{Tr} \, \mathcal R(z),\quad  \um(z)\equiv \um (W_t,z):= \frac{1}{n} \mathrm{Tr} \, \uR(z).
\ee
By \eqref{Yt=Wt}, we know that they have the same distribution as the original $m(z)$ and $\um(z)$ defined in \eqref{defn_m} and \eqref{defn_m2}. Moreover, by \eqref{Yt=Wt} we have
$$G(z)\stackrel{d}{=} \begin{pmatrix} O_1 & 0 \\ 0 & O_2\end{pmatrix}R(z)\begin{pmatrix} O_1^\top & 0 \\ 0 & O_2^\top \end{pmatrix}.$$
Correspondingly, we can obtain from \eqref{defnPi} the asymptotic limit of $R(z)$ as
\be \label{defnPi2}
\Pi^w(z):=\begin{bdmatrix} \frac{-(1+c_n t m_{w,t})}{z(1+c_n t m_{w,t})(1+t \underline m_{w,t} )-WW^\top} &  \frac{-z^{-1/2} }{z(1+c_n t m_{w,t})(1+t \underline m_{w,t} )-WW^\top}W \\ 
W^\top \frac{-z^{-1/2} }{z(1+c_n t m_{w,t})(1+t \underline m_{w,t} )-WW^\top} & \frac{-(1+t \underline m_{w,t} )}{z(1+c_n t m_{w,t})(1+t \underline m_{w,t} )-W^\top W}  
\end{bdmatrix}.
\ee
To prove Theorem \ref{thm_local}, it suffices to study the resolvent $R(z)$.

\subsection{Basic tools}
In this subsection, we introduce more notations and collect some basic tools that will be used in the proof. 
First as in \eqref{green2}, using Schur complement formula we can get
\begin{equation} \label{green}
R= \left( {\begin{array}{*{20}c}
   { \mathcal R} & z^{-1/2} \mathcal RW_t \\
   {z^{-1/2} W_t^\top \mathcal R} & { \uR }  \\
\end{array}} \right)= \left( {\begin{array}{*{20}c}
   { \mathcal R} & z^{-1/2} W_t \uR   \\
   z^{-1/2} {\uR}W_t^\top & { \uR }  \\
\end{array}} \right).
\end{equation}
For simplicity of notations, we define the index sets
$$\mathcal I_1:=\{1,...,p\}, \ \ \mathcal I_2:=\{p+1,...,p+n\}, \ \ \mathcal I:=\mathcal I_1\cup\mathcal I_2.$$
We shall consistently use the latin letters $i,j\in\mathcal I_1$, greek letters $\mu,\nu\in\mathcal I_2$, and $\fa,\fb\in\mathcal I$. 
For simplicity, given a vector $\mathbf v\in \mathbb C^{\mathcal I_{1,2}}$, we always identify it with its natural embedding in $\C^{\cal I}$. For example, we shall identify $\mathbf v\in \mathbb C^{\mathcal I_1}$ with $\left( {\begin{array}{*{20}c}
   {\mathbf v}  \\
   \mathbf 0_{n} \\
\end{array}} \right)$.


\begin{definition}[Minors]
For any $ \cal J \times \cal J$ matrix $\cal A$ and $\mathbb T \subseteq \mathcal J$, where $\cal J$ and $\mathbb T$ are some index stes, we define the minor $\cal A^{(\mathbb T)}:=(\cal A_{ab}:a,b \in \mathcal J\setminus \mathbb T)$ as the $ (\cal J\setminus \mathbb T)\times (\cal J\setminus \mathbb T)$ matrix obtained by removing all rows and columns indexed by $\mathbb T$. Note that we keep the names of indices when defining $\cal A^{(\mathbb T)}$, i.e. $(\cal A^{(\mathbb{T})})_{ab}= \cal A_{ab}$ for $a,b \notin \mathbb{{T}}$. Correspondingly, we define the resolvent minor as
\begin{align*}
R^{(\mathbb T)}(z):&=\left[ z^{1/2}\wt H_t^{(\mathbb T)}(z) -z\right]^{-1}= \left( {\begin{array}{*{20}c}
   { \mathcal R^{(\mathbb T)}} & z^{-1/2} \mathcal R^{(\mathbb T)} W_t^{(\mathbb T)}  \\
   {z^{-1/2} (W_t^{(\mathbb T)})^\top \mathcal R^{(\mathbb T)}} & { \uR^{(\mathbb T)} }  \\
\end{array}} \right) ,
\end{align*}
and the partial traces
$$m^{(\mathbb T)}:=\frac{1}{p}\sum_{i\in \cal I_1}^{(\mathbb T)} R_{ii}^{(\mathbb T)},\quad \um^{(\mathbb T)}:= \frac{1}{n}\sum_{\mu \in \cal I_2 }^{(\mathbb T)} R_{\mu\mu}^{(\mathbb T)},$$
where we used the notation $\sum_{a}^{(\mathbb T)} := \sum_{a\notin \mathbb T} .$ For $\mathbb T\subset \mathcal I_1$, we define the subset $[\mathbb T]:=\{\fa \in \mathcal I: \fa \in \mathbb T \text{ or } \fa + p \in \mathbb T\}$. Then we define the minor $\wt H^{[\mathbb T]}:=\wt H^{([\mathbb T])}$, and correspondingly $R^{[\mathbb T]}:= R^{([\mathbb T])}$.
\end{definition}
 For convenience, we will adopt the convention that for any minor $\cal A^{(\mathbb T)}$ defined as above, $\cal A^{(\mathbb T)}_{ab} = 0$ if $a \in \mathbb T$ or $b \in \mathbb T$. Moreover, we will abbreviate $\overline i:= i+p\in \cal I_2$ for $i\in \cal I_1$, $\overline \mu:= \mu-p\in \cal I_1$ for $p+1\le \mu \le 2p$, $(\{\fa\})\equiv (\fa)$, $[i]\equiv [\{i\}]$, $(\{\fa, \fb\})\equiv (\fa\fb)$ and $[\{i, j\}]\equiv [ij]$.


For an $\cal I \times \mathcal I$ matrix $\cal A$ and $i,j \in \mathcal I_1$, we define the $2\times 2$ minors as
\begin{equation}\nonumber
 \cal A_{[ij]} :=\left( {\begin{array}{*{20}c}
   {\cal A_{ij} } & {\cal A_{i\overline j} }  \\
   {\cal A_{ \overline i j} } & {\cal A_{ \overline i \overline j} }  \\
\end{array}} \right),
\end{equation}
Moreover, for $\fa\in \cal I\setminus \{i,\overline i\}$ we denote
\begin{equation}\nonumber 
 \cal A_{[i]\fa}=\left( {\begin{array}{*{20}c}
   {\cal A_{i \fa} }  \\
   {\cal A_{\overline i \fa} }\\
\end{array}} \right),\quad \cal A_{\fa[i]}=\left(  {\cal A_{\fa i} }, {\cal A_{\fa\overline i} } .
 \right), \quad.
\end{equation}
Now we record the following resolvent identities obtained from Schur complement formula.


\begin{lemma}\label{lemm_resolvent}
The following resolvent identities hold.
\begin{itemize}
\item[(i)] For $i\in \cal I_1$, we have 
 \begin{equation}\label{eq_res11}
   (R_{[ii]})^{ - 1}  = \begin{pmatrix} - z & z^{1/2}(d_i^{1/2} + \sqrt{t}X_{i\overline i}) \\ z^{1/2}(d_i^{1/2} + \sqrt{t}X_{i\overline i})  & -z \end{pmatrix}  - zt \begin{pmatrix}( XR^{[i]} X^\top)_{ii} & ( X R^{[i]}X)_{i \overline i} \\  (X^\top R^{[i]}  X^\top)_{\overline i i} &  (X^\top R^{[i]}  X)_{\overline i \overline i} \end{pmatrix}   .
    \end{equation}
    For $2p+1\le \mu \le p+n$, we have
\begin{equation}
\frac{1}{{R_{\mu \mu } }} =  - z  - zt\left( {X^\top  R^{\left( \mu  \right)} X} \right)_{\mu \mu }.\label{resolvent2}
\end{equation}

\item[(ii)]  For $i\ne j \in \mathcal I_1$, we have 
       \begin{align}
   R_{[ij]}  &= - z^{1/2}\sqrt{t}  R_{[ii]} \begin{pmatrix} (X R^{[i]})_{ij} &  (X R^{[i]})_{i \overline j} \\ (X^{\top} R^{[i]})_{\overline i j} & (X^{\top} R^{[i]})_{\overline i  \overline j} \end{pmatrix}   = -  z^{1/2}\sqrt{t}  \begin{pmatrix} (R^{[j]} X^\top)_{ij} & (R^{[j]} X)_{i \overline j}\\ (R^{[j]} X^\top)_{\overline i j } & (R^{[j]} X)_{\overline i \overline j }\end{pmatrix} R_{[jj]} 
   \nonumber\\
 & = zt R_{[ii]} \begin{pmatrix} (X R^{[ij]} X^\top )_{ij} & (X R^{[ij]}X)_{i \overline j} \\ (X^{\top} R^{[ij]} X^\top )_{\overline ij} & (X^{\top} R^{[ij]} X)_{\overline i \overline j}\end{pmatrix}R_{[jj]}^{[i]}. \label{eq_res21}
    \end{align}
 For $2p+1\le \mu \ne \nu \le p+n$, we have
 \begin{equation}
 R_{\mu \nu }  = - z^{1/2} \sqrt{t} R_{\mu \mu }  \left( {X^\top  R^{\left( {\mu} \right)} } \right)_{\mu \nu} = - z^{1/2} \sqrt{t}  \left( { R^{\left( {\nu} \right)} X } \right)_{\mu \nu}R_{\nu \nu } = zt R_{\mu \mu } R^{(\mu)}_{\nu \nu }  \left( {X^\top  R^{\left( {\mu\nu} \right)} X} \right)_{\mu \nu} . \label{resolvent3}
\end{equation}

\item[(iii)] For $i \in \mathcal I_1$ and $2p+1\le \mu \le p+n$, we have 
    \begin{align}
   R_{[i]\mu} = - z^{1/2}\sqrt{t} R_{[ii]} \begin{pmatrix}  (X R^{(i\overline i \mu)})_{ i \mu}   \\ ( X^\top R^{(i\overline i \mu)})_{\overline i\mu} \end{pmatrix} =zt R_{[ii]}R_{\mu\mu}^{[i]} \begin{pmatrix}  (X R^{(i\mu\overline \mu)}X)_{  i\mu}   \\ ( X^{\top}R^{(i\mu\overline \mu)}X)_{\overline i \mu } \end{pmatrix} .\label{eq_res23}
       \end{align}

 \item[(vi)]
 For $\fa \in \mathcal I$ and $\fb,\mathfrak c \in \mathcal I \setminus \{\fa\}$,
\begin{equation}
R_{\fb\mathfrak c}  = R_{\fb\mathfrak c}^{\left( \fa \right)} + \frac{R_{\fb\fa} R_{\fa\mathfrak c}}{R_{\fa\fa}}, \quad \frac{1}{{R_{\fb\fb} }} = \frac{1}{{R_{\fb\fb}^{(\fa)} }} - \frac{{R_{\fb\fa} R_{\fa\fb} }}{{R_{\fb\fb} R_{\fb\fb}^{(\fa)} R_{\fa\fa} }}. \label{resolvent8}
\end{equation}
  For $i \in \mathcal I_1 $ and $\fa,\fb \in \mathcal I\setminus \{i,\overline i\}$, we have 
    \begin{equation}\label{eq_res3}
     R_{\fa \fb} =R_{\fa \fb}^{[i]} +R_{\fa[i]}R_{[ii]}^{-1}R_{[i]\fb}, 
     \quad R_{\fa\fa}^{-1}= (R^{[i]}_{\fa\fa})^{ -1} - R_{\fa\fa}^{ - 1} (R^{[i]}_{\fa\fa})^{ -1} R_{\fa[i]}R_{[ii]}^{-1}R_{[i]\fa}.
    \end{equation}

  \item[(vii)] All of the above identities hold for $R^{[\mathbb T]}$ instead of $R$ for any subset $\mathbb T\subset \mathcal I_1$.

  \end{itemize}
\end{lemma}

Using the spectral decomposition of $G(z)$, it is easy to prove the following estimates and identities, which also hold for $R(z)$ as a special case. The reader can also refer to \cite[Lemma 4.6]{Anisotropic}, \cite[Lemma 3.5]{XYY_circular} and \cite[Lemma 5.4]{yang20190} for the proof.

\begin{lemma}\label{Ward_id}
For $z=E+\ii\eta$ such that $|z|\sim 1$, we have that for some constant $C>0$,
\begin{equation}
\left\| G(z) \right\| \le C\eta ^{ - 1} ,\quad \left\| {\partial _z G}(z) \right\| \le C\eta ^{ - 2}. 
\label{eq_gbound}
\end{equation}
Furthermore, we have the following identities:
\begin{align}
 \sum_{i \in \mathcal I_1 }  \left| {G_{ij} } \right|^2   = \frac{\Im\, G_{jj}}{\eta} , \quad 
&\sum_{\mu  \in \mathcal I_2 } {\left| {G_{\mu\nu} } \right|^2 }  = \frac{{\Im \, G_{\nu\nu} }}{\eta}, \label{eq_gsq1}\\ 
 \sum_{i \in \mathcal I_1 } {\left| {G_{i\mu} } \right|^2 }  = |z|^{-1}{G}_{\mu \mu}  + \frac{\overline z}{|z|}\frac{\Im \, G_{\mu\mu} }{\eta} , \quad &\sum_{\mu \in \mathcal I_2 } {\left| {G_{i \mu} } \right|^2 } =|z|^{-1} {G}_{ii} + \frac{\overline z}{|z|} \frac{\Im\, G_{ii}  }{\eta} .\label{eq_gsq3} 
 \end{align}
All of the above estimates remain true for $G^{(\mathbb T)}$ instead of $G$ for any $\mathbb T \subseteq \mathcal I$.
\label{lemma_Im}
\end{lemma}

As a consequence of Lemma \ref{Ward_id}, we can derive the following estimate. For its proof, we refer the reader to \cite[Lemma 5.5]{yang20190}.
\begin{lemma}
For any $\mathbb T \subseteq \mathcal I$ and $z=E+\ii\eta$ such that $|z|\sim 1$, we have
\begin{equation}\label{m_T}
\big| {m  - m^{\left( \mathbb T \right)} } \big| = c_n^{-1} \big| {\mm  - \mm^{\left( \mathbb T \right)} } \big|  \le \frac{{C\left| \mathbb T \right|}}{{n\eta }}, 
\end{equation}
for some constant $C>0$.
\end{lemma}

The following lemma gives standard large deviation bounds. 

\begin{lemma}[Lemmas B.2-B.4 of \cite{Delocal}]\label{largedeviation}
Let $(x_i)$, $(y_j)$ be independent families of centered and independent random variables, and $(A_i)$, $(B_{ij})$ be families of deterministic complex numbers. Suppose the entries $ x_i$ and $ y_j$ have variance at most $n^{-1}$ and finite moments up to any order, i.e. for any fixed $k\in \N$, there exists a constant $C_k>0$ such that
$$\max_i \E |\sqrt{n}x_i|^k \le C_k,\quad \max_i \E |\sqrt{n}y_i|^k \le C_k .$$
Then the following large deviation bounds hold:
\begin{align*}
 \Big\vert \sum_i A_i x_i \Big\vert \prec  n^{-1/2}\Big(\sum_i |A_i|^2 \Big)^{1/2} , \quad &\Big\vert \sum_{i,j} x_i B_{ij} y_j \Big\vert \prec n^{-1}\Big(\sum_{i, j} |B_{ij}|^2\Big)^{{1}/{2}} , \\
  \Big\vert \sum_{i} \overline x_i B_{ii} x_i - \sum_{i} (\mathbb E|x_i|^2) B_{ii}  \Big\vert  \prec n^{-1}\Big( \sum_i |B_{ii}|^2\Big) ,\quad &\Big\vert \sum_{i\ne j} \overline x_i B_{ij} x_j \Big\vert  \prec  n^{-1}\Big(\sum_{i\ne j} |B_{ij}|^2\Big)^{{1}/{2}} .
\end{align*}
\end{lemma}  

The following lemma collects basic properties of stochastic domination $\prec$, which will be used tacitly throughout the proof.

\begin{lemma}[Lemma 3.2 in \cite{isotropic}]\label{lem_stodomin}
Let $\xi$ and $\zeta$ be families of nonnegative random variables.
\begin{enumerate}
\item Suppose that $\xi (u,v)\prec \zeta(u,v)$ uniformly in $u\in U$ and $v\in V$. If $|V|\le n^C$ for some constant $C$, then $\sum_{v\in V} \xi(u,v) \prec \sum_{v\in V} \zeta(u,v)$ uniformly in $u$.

\item If $\xi_1 (u)\prec \zeta_1(u)$ and $\xi_2 (u)\prec \zeta_2(u)$ uniformly in $u\in U$, then $\xi_1(u)\xi_2(u) \prec \zeta_1(u)\zeta_2(u)$ uniformly in $u$.

\item Suppose that $\Psi(u)\ge n^{-C}$ is deterministic and $\xi(u)$ satisfies $\mathbb E\xi(u)^2 \le n^C$ for all $u$. Then if $\xi(u)\prec \Psi(u)$ uniformly in $u$, we have $\mathbb E\xi(u) \prec \Psi(u)$ uniformly in $u$.
\end{enumerate}
\end{lemma}


We define the random error
\begin{equation}\label{eqn_randomerror}
\theta(z):= |m(z)-m_{w,t}(z)|=c_n^{-1}|\mm(z)-\mm_{w,t}(z)| .
\end{equation}
Moreover, we define the random control parameters 
\begin{equation}\label{eq_defpsitheta}
\Psi _\theta(z)  : =  \sqrt {\frac{{\Im \, m_{w,t}(z)  + \theta(z) }}{{n\eta }}} + \frac{1}{n\eta},
\end{equation}
and
\begin{equation}\nonumber
 \Lambda _o : =  \max_{i \ne j \in \mathcal I_1} \left\| R_{[ii]}^{-1} {R_{[ij]}\left(R_{[jj]}^{[i]}\right)^{-1}  } \right\| +  \max_{\mu \ne \nu \ge 2p+1 } \left|R_{\mu \mu } ^{-1}R_{\mu\nu}\left(R^{(\mu)}_{\nu \nu }\right)^{-1}\right|+ \max_{i\in \cal I_1 ,  \mu\ge 2p+1 }  \left\| R_{[ii]}^{-1}   R_{[i]\mu} \left(R_{\mu\mu}^{[i]} \right)^{-1}\right\|  ,
 \ee
 which controls the size of off-diagonal entries. In analogy to \cite[Section 3]{EKYY1} and \cite[Section 5]{Anisotropic}, we introduce the $Z$ variables
\begin{equation*}
  Z_{\mu} :=(1-\mathbb E_{\mu})\big(R_{\mu\mu} \big)^{-1}, \quad \mu\ge 2p+1,
\end{equation*}
where $\mathbb E_{\mu}[\cdot]:=\mathbb E[\cdot\mid \wt H_t^{(\mu)}]$ is the partial expectation over the randomness of the $\mu$-th row and column of $\wt H_t$. By (\ref{resolvent2}), we have that for $\mu\ge 2p+1$, 
\begin{equation}
Z_\mu =zt (\mathbb E_{\mu} - 1) \left( {X^\top R^{\left( \mu \right)} X} \right)_{\mu\mu} =zt \sum_{i,j\in \cal I_1} R^{(\mu)}_{ij} \left(\frac{1}{n}\delta_{ij} - X_{i\mu}X_{j\mu}\right). \label{Zmu}
\end{equation}
We also introduce the following matrix value $Z$ variables:
\begin{equation*}
  Z_{[i]} :=(1-\bbE_{[i]})\left(R_{[ii]} \right)^{-1} ,\quad i \in \cal I_1,
\end{equation*}
where $\mathbb E_{[i]}[\cdot]:=\mathbb E[\cdot\mid \wt H_t^{[i]}]$ is the partial expectation over the randomness of the $i$-th and $\overline i$-th rows and columns of $\wt H_t$. By (\ref{eq_res11}), we have
\begin{equation}\label{Zii}
Z_{[i]} =\sqrt{zt} \begin{pmatrix} 0 &  X_{i\overline i} \\  X_{i\overline i}  & 0 \end{pmatrix} + zt \begin{pmatrix}\sum_{\mu,\nu \in \cal I_2} R^{[i]}_{\mu\nu} (n^{-1}\delta_{\mu\nu} -X_{i\mu  }X_{i\nu}) & \sum_{j \in \cal I_1, \mu\in \cal I_2} R^{[i]}_{\mu j} X_{i\mu}X_{j\overline i} \\ \sum_{j \in \cal I_1, \mu\in \cal I_2}R^{[i]}_{j\mu} X_{i \mu} X_{j\overline i } & \sum_{ j,k \in \cal I_1} R^{[i]}_{jk} (n^{-1}\delta_{jk} - X_{j\overline i}X_{k\overline i}) \end{pmatrix} .
\end{equation}
Now using Lemma \ref{largedeviation}, we can prove the following large deviation estimates on the $Z$ variables and $\Lambda_o$.


\begin{lemma}\label{Z_lemma}
Define the events $\Xi_0:=\{\theta=\OO(1)\}$. For $z\in \cal D_{\vartheta}$, we have
\begin{align}
 \|Z_{[i]}\| + |Z_\mu| \prec  t\Psi_\theta +t^{1/2}n^{-1/2}, \quad i \in \cal I_1, \ \mu\ge 2p+1,\label{Zestimate1}
\end{align}
and
\begin{align}
\mathbf 1(\Xi_0) \Lambda_o  \prec t\Psi_\theta .\label{Zestimate2}
\end{align}
Moreover, for any constant $c_0>0$ we have
\begin{align}
\mathbf 1(\eta\ge c_0)\left(  \|Z_{[i]}\| + |Z_\mu| +\Lambda_o\right)\prec  t^{1/2}n^{-1/2}, \quad i \in \cal I_1, \ \mu\ge 2p+1,\label{Zestimate3}
\end{align}
\end{lemma}
\begin{proof}
 Applying Lemma \ref{largedeviation} to $Z_{\mu}$ in (\ref{Zmu}), we get that 
\begin{equation}\label{estimate_Zi}
\begin{split}
\left| Z_{\mu}\right| & \prec \frac{t}{n} \left( \sum_{i,j} {\left| R_{ij}^{(\mu)}  \right|^2 }  \right)^{1/2} =  \frac{t}{n}\left( {\sum_i \frac{ \im R_{ii}^{(\mu)} }{\eta} } \right)^{1/2}\le t\sqrt { \frac{ \Im\, m^{(i)}  } {n\eta} } \lesssim t\sqrt { \frac{ \Im\, m } {n\eta} } + \frac{t}{n\eta}\le t\Psi_\theta(z),
\end{split}
\end{equation}
where we used (\ref{eq_gsq1}) in the second step, \eqref{m_T} in the fourth step and  (\ref{eqn_randomerror}) in the last step. With similar arguments, we get 
\begin{equation}\label{estimate_Zmu1}
|(Z_{[i]})_{11}|+ |(Z_{[i]})_{22}|\prec  t\Psi_\theta(z).
\ee
For the $(1,2)$ and $(2,1)$-th entries of $Z_{[\mu]}$, we have that on event $\Xi_0$,
\begin{equation}\label{estimate_Zmu2}
\begin{split}
|(Z_{[i]})_{12}|= |(Z_{[i]})_{21}|&\prec \frac{t^{1/2}}{n^{1/2}}+  \frac{t}{n} \left( \sum_{j,\mu} {\left| R_{j\mu}^{[i]}  \right|^2 }  \right)^{1/2}\le \frac{t^{1/2}}{n^{1/2}}+  t \left( |z|^{-1}\frac{m^{[i]}}{n} + \frac{\overline z}{|z|}\frac{\im m^{[i]}}{n\eta}\right)^{1/2}  \\
&\lesssim \frac{t^{1/2}}{n^{1/2}} + \frac{t}{n\eta}+  t\left( \frac{|m|}{n} +  \frac{\im m }{n\eta}\right)^{1/2} \prec \frac{t^{1/2}}{n^{1/2}}+ t\Psi_\theta(z),
\end{split}
\ee
where we used $X_{i\overline i}\prec n^{-1/2}$ by Markov's inequality in the second step, \eqref{eq_gsq3} in the third step, \eqref{m_T} in the fourth step, $\mathbf 1(\Xi_0)|m|=\OO(1)$ and (\ref{eqn_randomerror}) in the last step. This concludes \eqref{Zestimate1}. For \eqref{Zestimate2}, by \eqref{eq_res21}-\eqref{eq_res23}, each term  
$$R_{[ii]}^{-1} {R_{[ij]}\left(R_{[jj]}^{[i]}\right)^{-1}  },\quad \text{or}\quad R_{\mu \mu } ^{-1}R_{\mu\nu}\left(R^{(\mu)}_{\nu \nu }\right)^{-1},\quad \text{or}\quad  R_{[ii]}^{-1}   R_{[i]\mu} \left(R_{\mu\mu}^{[i]} \right)^{-1},$$
can be bounded in exactly the same way as above.  
Finally, for $\eta\ge c_0$ we have $\|R(z)\|=\OO(1)$ by \eqref{eq_gbound}, so $\Xi_0$ holds. Then we conclude \eqref{Zestimate3} immediately using \eqref{Zestimate1} and \eqref{Zestimate2}. 
\end{proof}

%

\subsection{The averaged local laws}\label{sec region11}
In this subsection, we mainly focus on the proof of the averaged local laws \eqref{averin} and \eqref{averout}, from which the anisotropic local law follows using a standard argument. We divide our proof into two parts according to whether $E\le \lambda_{+,t}+Ct^2$ or $E> \lambda_{+,t}+Ct^2$ for some constant $C>0$. More precisely, we define the following spectral domains for some large constant $C_1>0$:
\begin{equation}\label{eq_dtheta12}
\mathcal{D}_{\vartheta}^1:= \mathcal{D}_{\vartheta} \cap \{z=E+\ii \eta: E \leq \lambda_{+,t}+C_1 t^2 \}, \quad \mathcal{D}_{\vartheta}^2:= \mathcal{D}_{\vartheta} \cap \{z=E+\ii \eta: E > \lambda_{+,t}+C_1 t^2 \}.
\end{equation}
%
We first prove the averaged local laws on the region $\cal D_{\vartheta}^1$.
\begin{proposition}\label{region11}
Under the assumptions of Theorem \ref{thm_local}, \eqref{averin} holds uniformly in $z\in \cal D_{\vartheta}^1$, and \eqref{averout} holds uniformly in $z\in D_{\vartheta}^1\cap \cal D^{out}_\vartheta$. 
\end{proposition}

Recall the equation in \eqref{simplePhizeta}. Using Lemma \ref{lem_domian1control} and Lemma \ref{lem_domian12control}, we now provide some deterministic estimates on the derivatives of $\Phi_t$.
\begin{lemma}\label{lem derivPhi}
 For $\mathfrak z\in \C_+$ satisfying 
\be\label{choosem1} |\mathfrak z-m_{w,t}|\le \frac{t+\im m_{w,t}}{(\log n)^2},\ee
and $u\equiv u(\mathfrak z):= (1+c_n t \mathfrak z)^2 z - (1+c_n t \mathfrak z)t(1-c_n)$, we have
\be\label{firstderiv}
|\Phi_t'(u)| \sim  \min \left\{1, \frac{\sqrt{\kappa+\eta}}{t} \right\},
\ee
and for  $k= 2, 3,$
\be\label{secondderiv}
  |\Phi_t^{(k)}(u)| \lesssim \frac{t^2 + t\sqrt{\kappa+\eta}}{(t^2 + \im \zeta_t )^{k}}.
\ee
Moreover, if $\kappa+\eta\le \tau_1t^2$ for some small enough constant $\tau_1>0$, then we have
\be\label{secondderivsim}
  |\Phi_t^{(2)}(u)| \sim t^{-2}.
\ee
\end{lemma}
\begin{proof}
By \eqref{eq_dtheta1control}, we have that for $\mathfrak z$ satisfying \eqref{choosem1},
\be\label{mclose1}
|u-\zeta_t| \lesssim t|\mathfrak z-m_{w,t}| \lesssim \frac{\min_{i=1}^p |d_i-\zeta_t | }{(\log n)^2}.\ee
With this estimate, we can repeat the proof for \eqref{Phit'} and get  \eqref{firstderiv}.

Using \eqref{eq_dtheta1extension} and \eqref{mclose1}, we can bound that for any fixed $k\ge 1$,
\be\label{mkdtheta1}|m_{w,0}^{(k)}(u)| \lesssim \frac{t + \sqrt{\kappa+\eta}}{(t^2 + \im \zeta )^{k}}.\ee
From the definition \eqref{eq_subcompansion}, we can calculate directly that (recall \eqref{Phit''0})
\be\label{Phit''}
\begin{split}
\Phi_t''(u)&=- 2c_nt m''_{w,0}(u)\cdot u(1-c_nt m_{w,0}(u))- 4c_nt m'_{w,0}(u)\cdot (1-c_nt m_{w,0}(u))  \\
&+ 2u[c_nt m'_{w,0}(u)]^2 -c_n(1-c_n)t^2m''_{w,0}(u)  .
\end{split}
\ee
Using \eqref{mkdtheta1}, we can conclude \eqref{secondderiv} for $k=2$. The proof of \eqref{secondderiv} for $k=3$ is similar. Finally, using \eqref{Phit''} and Lemma \ref{lem_domian12control}, we can conclude \eqref{secondderivsim}.\end{proof}


The core of the proof for Proposition \ref{region11} is an analysis of the self-consistent equation. We define that 
\be\label{defn ut} u_{t}:=u(m)= (1+c_n t m)^2 z - (1+c_n t m)t(1-c_n),\ee
and the event
$$\Xi:=\left\{ \theta \le \frac{t+\im m_{w,t}}{(\log n)^2}\right\}.$$
Note that on this event, \eqref{mclose1} holds for $u=u_{t}$. Recalling the definition of $\Pi^w$ in \eqref{defnPi2}, it is easy to check that 
\be\label{pi_i}
\Pi^w_{[ii]}(z)=\begin{pmatrix}  - z(1+t \underline m_{w,t} ) &  z^{1/2}d_i^{1/2} \\ z^{1/2}d_i^{1/2}   & -z(1+c_n t m_{w,t}) \end{pmatrix}^{-1} .
\ee
\begin{lemma}\label{lemm_selfcons_weak}
For $z\in \cal D_{\vartheta}^1$ and on the event $\Xi$, we have that for any constant $\e>0$,
\begin{equation}
 \left| \Phi_t'(\zeta_t) (u_{t}-\zeta_t) + \frac12\Phi_t''(\zeta_t) (u_{t}-\zeta_t)^2  \right|\le  \frac{\theta^2}{\log n}\frac{t^2}{t^2 + \im \zeta_t}+n^\e  \frac{t^2+t\sqrt{\kappa+\eta}}{t^2+\im \zeta_t }t\Psi_\theta  +  n^\e \frac{tn^{-1}}{t^2+\im \zeta_t }   \label{selfcons_lemm}
\end{equation}
with high probability.
Moreover, we have the finer estimate: for any constant $\e>0$,
\begin{equation}
\left| \Phi_t'(\zeta_t) (u_{t}-\zeta_t) + \frac12\Phi_t''(\zeta_t) (u_{t}-\zeta_t)^2  \right|\le \frac{\theta^2}{\log n} \frac{ t^2}{t^2 + \im \zeta_t} + n^\e t \| [ Z]\| + n^\e \frac{t\left( t \Psi^2_\theta + n^{-1} \right)}{t^2+\im \zeta_t } , \label{selfcons_improved}
\end{equation}
with high probability, where
\begin{equation}\label{def_Zaver}
[Z]:=\frac{1}{p}\sum_{i\in \mathcal I_1}\Pi^w_{[ii]}Z_{[i]}\Pi^w_{[ii]}.
\end{equation}
\end{lemma}

\begin{proof}

 In the following proof, we always assume that event $\Xi$ holds. Using \eqref{resolvent2}, we get that for $\mu\ge 2p+1$,
\be\label{Gmumu} \frac{1}{R_{\mu\mu}} = - z- z c_n t m^{(\mu)} + Z_\mu = - z- z c_n t m_{w,t} + \e_\mu ,
\ee
where 
$$\e_\mu:= Z_\mu + z c_n t (m -m^{(\mu)}) + z c_n t (m_{w,t} -m).$$
On the other hand, using \eqref{eq_res11} we get that for $i\in \cal I_1$,
\begin{equation}\label{Gii}
\begin{split}
(R_{[ii]})^{-1} = \begin{pmatrix}- z(1+t\mm_{w,t} ) & z^{1/2}d_i \\ z^{1/2}d_i & -z(1+c_n tm_{w,t}) \end{pmatrix} +\e_{[i]} ,  
   \end{split}
\end{equation}
where 
\begin{align*} 
\e_{[i]}&:= Z_{[i]}+ zt \left( {\begin{array}{*{20}c}
   { \mm-\mm^{[i]} } & {0}  \\
   {0} & {c_n (m-m^{[i]}) } \end{array}} \right)+ zt \left( {\begin{array}{*{20}c}
   { \mm_{w,t}-\mm } & {0}  \\
   {0} & {c_n (m_{w,t}-m) } \end{array}} \right)  \\
   &= Z_{[i]}+ zc_nt\left[ (m-m^{[i]}) + (m_{w,t}-m)\right]\left( {\begin{array}{*{20}c}
   {1 } & {0}  \\
   {0} & {1} \end{array}} \right) ,
   \end{align*}
 where  we used \eqref{m21} in the second step. With \eqref{m_T} and \eqref{Zestimate1}, we can bound that 
 \be\label{eiemu} \|\e_{[i]}\|+|\e_{[\mu]}| \lesssim  t\theta + \OO_\prec\left(  t\Psi_\theta +t^{1/2}n^{-1/2} \right).\ee
 Together with \eqref{mclose1}, we get that with high probability,
$$\frac{\|\e_{[i]}\|}{|d_i-\zeta_t|}=\OO\left( (\log n)^{-2} \right) \ \Rightarrow \ \left\| \Pi^w_{[ii]}\e_{[i]}\right\|=\OO\left( (\log n)^{-2} \right).$$
Thus taking the matrix inverse of \eqref{Gii}, we get that
\be\label{Gii solv}
R_{[ii]}=\Pi_{[ii]}^w \left[ 1+\OO\left(\frac{\|\e_{[i]}\|}{|d_i-\zeta_t|}\right)\right] = \Pi_{[ii]}^w \left[ 1+\OO\left( (\log n)^{-2} \right)\right] 
\ee
with high probability on the event $\Xi$. On the other hand, by \eqref{m_T} we have that for $z\in \cal D_{\vartheta}^1$,
$$ \mathbf 1(\Xi) |m^{[i]}-m_{w,t}| \le \frac{C}{n\eta} +\frac{t+\im m_{w,t}}{(\log n)^2} \le 2\frac{t+\im m_{w,t}}{(\log n)^2}.$$
Thus repeating the above proof, we can obtain a similar estimate for $R_{[jj]}^{[i]}$ as in \eqref{Gii solv} for any $j\ne i$:
\be\label{Gii solv2}
R_{[jj]}^{[i]} = \Pi_{[jj]}^w \left[ 1+\OO\left( (\log n)^{-2} \right)\right] 
\ee
with high probability on the event $\Xi$. 

We can also write \eqref{Gii} as
\begin{equation}\label{Gii2}
\begin{split}
(R_{[ii]})^{-1} = \begin{pmatrix}- z(1+t\mm ) & z^{1/2}d_i^{1/2} \\ z^{1/2}d_i^{1/2} & -z(1+c_n tm) \end{pmatrix} +\wt \e_{[i]} ,  
   \end{split}
\end{equation}
where 
\begin{align}\label{bdd wte} 
\wt \e_{[i]}&:= Z_{[i]}+ zc_nt (m-m^{[i]})  \left( {\begin{array}{*{20}c}
   {1 } & {0}  \\
   {0} & {1} \end{array}} \right)  \prec t\Psi_\theta +\frac{t^{1/2}}{n^{1/2}}.
   \end{align}
Now taking the matrix inverse of both sides of \eqref{Gii2}, we get
\be\label{Gii22}
R_{[ii]}=\pi_{[i]} - \pi_{[i]} \wt\e_{[i]}\pi_{[i]} + \OO(\| \pi_{[i]} \|^3\|\wt\e_{[i]}\|^2),
\ee
where we denoted $\pi_{[i]}$ as
\be\label{pi_i}
\pi_{[i]}(z):=\begin{pmatrix}  - z(1+t  \mm ) &  z^{1/2}d_i^{1/2} \\ z^{1/2}d_i^{1/2}   & -z(1+c_n t m) \end{pmatrix}^{-1} .
\ee
With \eqref{mclose1} and \eqref{eq_dtheta1control}, we can check that on $\Xi$,
\be\label{pinorm}
\|\pi_{[i]}-\Pi^w_{[ii]}\|\lesssim \frac{t\theta}{|d_i-\zeta_t|^2} \lesssim \frac{t\theta}{(t^2+ \im \zeta_t)^2},\quad \|\pi_{[i]}\|\lesssim \|\Pi^w_{[ii]}\|\lesssim \frac{1}{|d_i-\zeta_t|} .
\ee
By \eqref{eq_res3}, we have 
$$ m_1-m_{1}^{[i]} = \frac1p\sum_{j}\left(R_{[ji]}R_{[ii]}^{-1}R_{[ij]}\right)_{11}.$$
Then using \eqref{Zestimate2}, \eqref{Gii solv} and \eqref{Gii solv2}, we get that 
\begin{align} 
|m_1-m_{1}^{[i]}| &\lesssim \frac1n \sum_{j\in \cal I_1} \left\| R_{[ji]}R_{[ii]}^{-1}R_{[ij]}\right\| \lesssim \frac{1}{n}\|R_{[ii]}\| + \frac1n\sum_{j\ne i}\Lambda_o^2 \|R_{[ii]}\| \| R_{[jj]}^{[i]}\|^2 \nonumber \\
&\prec \frac1{|d_i - \zeta_t|}\left(  \frac{1}{n} + \frac{t^2 \Psi_\theta^2}{n}\sum_{j\ne i}\frac{1}{|d_j - \zeta_t|^2} \right) \lesssim \frac{1}{t^2 + \im \zeta_t} \left( n^{-1} + t \Psi_\theta^2\right),\label{bdd m-mi}
\end{align}
where in the last step we used \eqref{eq_dtheta1control}, \eqref{eq_dtheta1extension} and that 
\be\label{loweretat}t^2 + \im \zeta_t \gtrsim t^2 + \eta+t \im m_{w,t}\gtrsim t^2+\eta + t\sqrt{\kappa+\eta}\ee  
by \eqref{kappaeta_d1}. Thus taking average of \eqref{Gii22} over $i\in \cal I_1$ and using \eqref{bdd wte},  \eqref{pinorm}, \eqref{bdd m-mi} and \eqref{eq_dtheta1extension} with $a=3$, we obtain that  
\be\label{Gii3}
\frac{1}{p}\sum_{i\in \cal I_1}R_{[ii]}=\frac{1}{p}\sum_{i\in \cal I_1} \pi_{[i]} - [Z] + \OO_\prec\left[  \frac{1}{t^2+\im \zeta} \left(\theta \left( t\Psi_\theta +\frac{t^{1/2}}{n^{1/2}}\right)+ t \Psi^2_\theta + \frac 1n \right)\right].
\ee
In particular, the $(1,1)$-th entry of \eqref{Gii3} gives that
\be\label{Giim1t}
m=\frac{1}{p} \sum_{i=1}^p \frac{1+c_n tm }{d_i-(1+c_n tm)^2z+t(1-c_n)(1+c_n tm )} + \OO_\prec( \cal E), \ee
where we used \eqref{m21} and introduced the notation
$$\cal E:= \| [ Z]\| + \frac{1}{t^2+\im \zeta} \left(\theta \left( t\Psi_\theta +\frac{t^{1/2}}{n^{1/2}}\right)+ t \Psi^2_\theta + \frac 1n \right).
$$
Using the definition of $m_{w,0}=m_V$, we can rewrite \eqref{Giim1t} as 
\be\nonumber
 m =(1+c_n tm )m_{w,0}(u_{t}) + \OO_\prec( \cal E) \Rightarrow m=\frac{m_{w,0}(u_{t}) + \OO_\prec( \cal E)}{1-c_n tm_{w,0}(u_{t})}.
\ee
Plugging it into \eqref{defn ut}, we get
\be\label{Gii4} z = \Phi_t(u_{t})+ \OO_\prec( t\cal E).\ee
Now subtracting the equation $z=\Phi_t(\zeta_{t})$ from \eqref{Gii4} and performing Taylor expansion, we get that
\be\label{Gii5} 
 \left|\Phi_t'(\zeta_t) (u_{t}-\zeta_t) + \frac12\Phi_t''(\zeta_t) (u_{t}-\zeta_t)^2\right|  \lesssim   \frac{t^2 +t \sqrt{\kappa+\eta}}{(t^2 + \im \zeta_t )^{3}} \cdot t^3 \theta^3 + \OO_\prec\left( t\cal E\right),
 \ee
where we used \eqref{secondderiv} and that $|u_{t}-\zeta_t|\lesssim t\theta$. We can bound $t\cal E$ as
\begin{align}\label{Giierr1}
t\cal E &\prec  t \| [ Z]\| + \frac{t}{t^2+\im \zeta } \left[ n^\e \left(t \Psi^2_\theta +  n^{-1} \right)+ n^{-\e}t\theta^2\right] 
\end{align}
for any constant $\e>0$, where we used that
$$t^{1/2}\theta \left( t^{1/2}\Psi_\theta + {n^{-1/2}}\right) \le  n^{-\e}{t\theta^2}  + n^\e \left( t^{1/2}\Psi_\theta + {n^{-1/2}}\right)^2.$$
On the other hand, using the definition of the event $\Xi$ and \eqref{loweretat} we can bound
\begin{align}\label{Giierr2}
  \frac{t^2 +t \sqrt{\kappa+\eta}}{(t^2 + \im \zeta_t )^{3}} \cdot t^3 \theta^3& \lesssim  \frac{ t^3\theta^2}{(t^2 + \im \zeta_t)^2}\frac{t+\im m_{w,t}}{(\log n)^2} \lesssim \frac{\theta^2}{(\log n)^2} \frac{ t^2}{t^2 + \im \zeta_t}.  
\end{align}
Plugging \eqref{Giierr1} and \eqref{Giierr2} into \eqref{Gii5}, we conclude \eqref{selfcons_improved}.

For \eqref{selfcons_lemm}, we need to bound 
$$[Z]=\frac1p\sum_{i\in \cal I_1}\Pi_{[ii]}^wA_{[i]}\Pi^w_{[ii]}+ \frac1p\sum_{i\in \cal I_1}\Pi^w_{[ii]}B_{[i]}\Pi^w_{[ii]},$$
where 
$$ A_{[i]}:= \sqrt{zt} \begin{pmatrix} 0 &  X_{i\overline i} \\  X_{i\overline i}  & 0 \end{pmatrix},\quad B_{[i]}:= Z_{[i]}-A_{[i]}.$$
The proof of Lemma \ref{Z_lemma} already gives that $\|A_{[i]}\|\prec t^{1/2} n^{-1/2}$ and $\|B_{[i]}\| \prec t\Psi_\theta.$ Using Lemma \ref{largedeviation} and \eqref{eq_dtheta1extension}, we obtain that 
\begin{align}\label{ZAZ}
\frac1p\sum_{i\in \cal I_1}\Pi_{[ii]}^wA_{[i]}\Pi^w_{[ii]} \prec n^{-1}\left( \frac1p\sum_i \frac{t}{|d_i -\zeta_t|^4}\right)^{1/2} \lesssim \frac{(t^2+t\sqrt{\kappa+\eta})^{1/2}}{n(t^2 + \im \zeta_t)^{3/2}}\lesssim \frac{1}{n(t^2 + \im \zeta_t) },
\end{align}
and 
\begin{align}\label{ZBZ}
\frac1p\sum_{i\in \cal I_1}\Pi_{[ii]}^wB_{[i]}\Pi_{[ii]}^w \prec \frac1p\sum_i \frac{t\Psi_{\theta} }{|d_i -\zeta_t|^2}\prec  \frac{t^2+t\sqrt{\kappa+\eta}}{t^2+\im \zeta_t }\Psi_\theta .
\end{align}
Plugging these two estimates into \eqref{selfcons_improved} and using  $ { \Psi_\theta  }  \ll t+\sqrt{\kappa+\eta} $ for $z\in \cal D_{\vartheta}^1$, we conclude \eqref{selfcons_lemm}.
\end{proof}

Combining Lemma \ref{lem_domian1control}, Lemma \ref{lem derivPhi} and Lemma \ref{lemm_selfcons_weak}, we can conclude the proof of Proposition \ref{region11} using the same argument as the one for Proposition 5.3 of \cite{edgedbm}. We only give an outline of the proof without writing down the details. 

\begin{lemma}[Weak averaged local law]\label{alem_weak} 
There exists a constant $\tau_1>0$ such that the following estimates hold for $z\in \cal D_{\vartheta}^1$: if $\kappa+\eta\ge \tau_1t^2$, then 
$$\theta \prec \Psi(z), $$
where $\Psi$ was defined in \eqref{eq_defpsi}; if $\kappa+\eta\le \tau_1 t^2$, then 
$$\theta \prec t^{2/3}{(n\eta)^{-1/3}}.$$
\end{lemma}
\begin{proof}
The proof of this lemma is the same as the one for Proposition 5.6 of \cite{edgedbm}. Roughly speaking, it follows from the self-consistent equation estimate \eqref{selfcons_lemm}, which corresponds to estimate (5.21) of \cite{edgedbm}. In the proof, we shall need some deterministic estimates on the coefficients $\Phi'_t(\zeta_t)$ and $\Phi''_t(\zeta_t)$, which are provided by Lemma \ref{lem derivPhi}. In particular, when $\kappa+\eta\le \tau_1 t^2$, we will need to use the estimate \eqref{secondderivsim}.
\end{proof}

To get the strong local laws in \eqref{averin} and \eqref{averout}, we need a stronger bound on $[Z]$ in (\ref{selfcons_improved}), which is given by the following fluctuation averaging lemma. 

\begin{lemma}[Fluctuation averaging] \label{abstractdecoupling}
Suppose that $\theta\prec \Phi$, where $\Phi$ is a positive, $n$-dependent deterministic function on $\cal D_{\vartheta}^1$ satisfying that
$$\frac{1}{n\eta} \le \Phi \le  \frac{t+\im m_{w,t}}{(\log n)^2}.$$  
Then for $z \in \cal D_{\vartheta}^1$ we have
\begin{equation}\label{flucaver_ZZ}
\left|[Z]\right|  \prec  \frac{1}{n\eta}\frac{\im m_{w,t} +\Phi}{t+\im m_{w,t} +\Phi}.
\end{equation}
\end{lemma}
\begin{proof}
The proof is similar to the one for Lemma 5.8 of \cite{edgedbm}, where the only difference is that \cite{edgedbm} dealt with an average of scalar value $Z$ variables, while here $[Z]$ is an average of matrix value $Z$ variables. However, this only brings into some minor notational changes. For a fluctuation averaging estimate for matrix value $Z$ variables, one can also refer to \cite[Lemma 4.9]{XYY_circular}. 
\end{proof}

Now we have all the ingredients for the proof of Proposition \ref{region11}.

\begin{proof}[Proof of Proposition  \ref{region11}]
The proof of Proposition  \ref{region11} follows from an iteration argument using the stronger self-consistent equation estimate (\ref{selfcons_improved}), Lemma \ref{alem_weak} and Lemma \ref{abstractdecoupling}. The proof is exactly the same as the one in Section 5.1.4 of \cite{edgedbm}, so we omit the details.
\end{proof}

%

Next we prove the averaged local law on the region $ \cal D_{\vartheta}^2$. 

\begin{proposition}\label{region12}
Under the assumptions of Theorem \ref{thm_local}, \eqref{averout} holds uniformly in $z\in D_{\vartheta}^2$. 
\end{proposition}

Our proof of Proposition \ref{region12} follows the same strategy as for the proof of Proposition \ref{region11}.  In this case, using the estimates in Lemma \ref{lem_domian2control} we can prove the following version of Lemma \ref{lem derivPhi}.

\begin{lemma}\label{lem derivPhi2}
For $\mathfrak z\in \C_+$ satisfying 
\be\label{choosem2} |\mathfrak z-m_{w,t}|\le \frac{t}{(\log n)^2},\ee
and $u= (1+c_n t \mathfrak z)^2 z - (1+c_n t \mathfrak z)t(1-c_n)$, we have
\be\label{firstsecondderiv}
|\Phi_t'(u)| \sim 1,\quad |\Phi_t^{(2)}(u)| \lesssim \frac{t}{(t^2 +\kappa+\eta)^{3/2}}.
\ee
\end{lemma}
\begin{proof}
By \eqref{eq_dtheta2bound1}, we have that for $\mathfrak z$ satisfying \eqref{choosem2}, \eqref{mclose1} holds. Using this estimate, we can repeat the proof for \eqref{Phit'} and get the first estimate in \eqref{firstsecondderiv}. Then using \eqref{Phit''} and \eqref{eq_dtheta2bound2}, we can obtain the second estimate in \eqref{firstsecondderiv}. \end{proof}

\nc

Next we prove the following counterpart of Lemma \ref{lemm_selfcons_weak}. We define the event
$$\Xi:=\left\{ \theta \le \frac{t}{(\log n)^2}\right\}.$$
\begin{lemma}\label{lemm_selfcons_weak2}
For $z\in \cal D_{\vartheta}^2$ and on the event $\Xi$, we have that for any constant $\e>0$,
\begin{equation}
 \left| \Phi_t'(\zeta_t) (u_{t}-\zeta_t) \right|\lesssim  \theta^2 +n^\e t  \Psi_\theta  +  \frac{n^\e t  }{n(\kappa+\eta) }  \label{selfcons_lemm2}
\end{equation}
with high probability. Moreover, we have the finer estimate: for any constant $\e>0$,
\begin{equation}
 \left| \Phi_t'(\zeta_t) (u_{t}-\zeta_t)\right|\lesssim  \theta^2 + n^\e t \| [ Z]\| + \frac{n^\e t}{(n\eta)^2 \sqrt{\kappa+\eta}}+\frac{n^\e t}{n(\kappa+\eta)}  \label{selfcons_improved2}
\end{equation}
with high probability. 
\end{lemma}
\begin{proof}
In this case, with \eqref{eq_dtheta2bound2} the estimate \eqref{Gii3} becomes 
\be\label{Gii3(2)}
\frac{1}{p}\sum_{i\in \cal I_1}R_{[ii]}=\frac{1}{p}\sum_{i\in \cal I_1} \pi_{[i]} - [Z] + \OO_\prec\left[  \frac{t}{(t^2+\kappa+ \eta)^{3/2}} \left(\theta \left( t\Psi_\theta +\frac{t^{1/2}}{n^{1/2}}\right)+ t \Psi^2_\theta + \frac 1n \right)\right].
\ee
Then estimate \eqref{Giim1t} now holds with 
$$\cal E:= \| [ Z]\| + \frac{t}{(t^2+\kappa+ \eta)^{3/2}}\left(\theta \left( t\Psi_\theta +\frac{t^{1/2}}{n^{1/2}}\right)+ t \Psi^2_\theta + \frac 1 n \right).
$$
Repeating the proof below \eqref{Giim1t}, we can obtain \eqref{Gii4}. Again performing Taylor expansion and using \eqref{firstsecondderiv}, we obtain that for any constant $\e>0$,
\be \nonumber 
 \left|\Phi_t'(\zeta_t) (u_{t}-\zeta_t) \right|  \lesssim \frac{t^3}{(t^2 +\kappa+\eta)^{3/2}} \theta^2 + \OO_\prec\left( t\cal E\right) \lesssim \theta^2 + \OO_\prec\left( t \| [ Z]\| + \frac{n^{\e}t}{(t^2+\kappa+ \eta)^{3/2}}\left(t^2 \Psi^2_\theta + \frac tn \right)\right)
 \ee
 with high probability. This concludes \eqref{selfcons_improved2} using the definition of $\Psi_\theta$,  the estimate \eqref{eq_imasymptoics}, and $\kappa\ge C_1t^2$ for $z\in \cal D_\vartheta^2$. For the term $[Z]$, we can bound it as in \eqref{ZAZ} and \eqref{ZBZ}: 
\begin{align}\nonumber
\frac1p\sum_{i\in \cal I_1}\Pi^w_{[ii]}A_{[i]}\Pi^w_{[ii]} \prec n^{-1}\left( \frac1p\sum_i \frac{t}{|d_i -\zeta_t|^4}\right)^{1/2} \lesssim \frac{1}{n}\left( \frac{t}{(t^2+\kappa+\eta)^{5/2}}\right)^{1/2} \lesssim \frac{1}{n(t^2+\kappa+\eta)},
\end{align}
and 
\begin{align}\nonumber
\frac1p\sum_{i\in \cal I_1}\Pi^w_{[ii]}B_{[i]}\Pi^w_{[ii]} \prec \frac1p\sum_i \frac{t\Psi_{\theta} }{|d_i -\zeta_t|^2}\lesssim  \frac{t\Psi_\theta}{(t^2+\kappa+\eta)^{1/2}}\lesssim \Psi_\theta .
\end{align}
Plugging these two estimates into \eqref{selfcons_improved2} we conclude \eqref{selfcons_lemm2}.
\end{proof}

From Lemma \ref{lemm_selfcons_weak2}, we can then derive  a similar weak local law as in Lemma \ref{alem_weak}.
\begin{lemma}[Weak local law]\label{alem_weak2} 
For all $z\in \cal D_{\vartheta}^2$, we have
$$\theta \prec \Psi(z). $$ 
\end{lemma}
\begin{proof}
The proof of this lemma is similar to the one for Proposition 5.6 of \cite{edgedbm}. In fact, it is much simpler because the proof uses the estimate \eqref{selfcons_lemm2}, which takes a much simpler form than \eqref{selfcons_lemm}. We omit the details.
\end{proof}

Similarly to Lemma \ref{abstractdecoupling}, we have the following fluctuation averaging estimate on $[Z]$. 

\begin{lemma}[Fluctuation averaging] \label{abstractdecoupling2}
Suppose that $\theta\prec \Phi$, where $\Phi$ is a positive, $n$-dependent deterministic function on $\cal D_{\vartheta}^2$ satisfying that
$$\frac{1}{n\eta} \le \Phi \le  \frac{t}{(\log n)^2}.$$   
Then for all $z \in \cal D_{\vartheta}^2$ we have
\begin{equation}\label{flucaver_ZZ2}
\left|[Z]\right|  \prec  \frac{1}{n(\kappa+\eta)}+ \frac{\Phi}{n\eta\sqrt{\kappa+\eta}}.
\end{equation}
\end{lemma}
\begin{proof}
The proof is similar to the one for Lemma 5.13 of \cite{edgedbm}. We omit the details. 
\end{proof}

Now we give the proof of Proposition \ref{region12}.

\begin{proof}[Proof of Proposition  \ref{region12}]
With (\ref{selfcons_improved2}), \eqref{firstsecondderiv} and Lemma \ref{abstractdecoupling2}, we get
\be\label{simpleiter1} \theta\prec  \frac{n^\e }{(n\eta)^2\sqrt{\kappa+\eta}}+\frac{n^\e }{n(\kappa+\eta)} + \frac{n^\e\Phi}{n\eta\sqrt{\kappa+\eta}}.\ee
Iterating this estimate, we get $\theta\prec n^\e /(n\eta)$. Plugging $\Phi=n^\e /(n\eta)$ into \eqref{simpleiter1}, we conclude Proposition \ref{region12} since $\e$ can be arbitrarily small.
\end{proof}

\subsection{The anisotropic local law}\label{sec region22}

In this subsection, we complete the proof Theorem \ref{thm_local}. Note that Proposition \ref{region11} and Proposition \ref{region12} already conclude the averaged local laws \eqref{averin} and \eqref{averout} by \eqref{Yt=Wt}. It remains to prove the anisotropic local law \eqref{aniso_law}. Since most arguments are standard, we only give an outline of the proof and refer the reader to relevant references for more details. 

By \eqref{Yt=Wt}, we see that to prove \eqref{aniso_law}, it suffices to prove that for any deterministic unit vectors $\mathbf u, \mathbf v \in \mathbb R^{p+n}$,
\begin{equation}\label{aniso_laww}
\left|  \mathbf u^\top (\Pi^w(z))^{-1}\left[R(z)-\Pi^w(z)\right] (\Pi^w(z))^{-1}\mathbf v   \right|\prec  t\Psi(z) + \frac{t^{1/2}}{n^{1/2}}.
\end{equation}
We first prove an entrywise version of this result, that is, for any $\fa,\fb\in \cal I$,
 \begin{equation}\label{entry_laww}
\left|  \left[(\Pi^w(z))^{-1}\left[R(z)-\Pi^w(z)\right] (\Pi^w(z))^{-1}\right]_{\fa\fb}  \right|\prec  t\Psi(z) + \frac{t^{1/2}}{n^{1/2}}.
\end{equation}

First, plugging \eqref{averin} into \eqref{eiemu}, we get
 \be\nonumber \|\e_{[i]}\|+|\e_{[\mu]}| \prec  t\Psi  +t^{1/2}n^{-1/2}  .\ee 
 Note that by \eqref{eq_dtheta1control} and \eqref{eq_dtheta2bound1}, we can check that (recall \eqref{boundpsi})
 $$ \|\Pi_{[ii]}^w(z)\| \cdot \left(t\Psi  +t^{1/2}n^{-1/2} \right)\le \frac{t\Psi  +t^{1/2}n^{-1/2} }{\min_i |d_i-\zeta_t|}\lesssim n^{-\vartheta/2},\quad z\in \cal D_\vartheta.$$
Now using \eqref{Gmumu} and \eqref{Gii}, we obtain that for $\mu \ge 2p+1$,
\be\label{Rmumu}
R_{\mu\mu}=\frac{1}{ - z- z c_n t m_{w,t} +\OO_\prec ( t\Psi  +t^{1/2}n^{-1/2} )}= \Pi_{\mu\mu}^w + \OO_\prec ( t\Psi  +t^{1/2}n^{-1/2} ),
\ee
 and that for $i\in \cal I_1$,
 \begin{align}
 (\Pi_{[ii]}^w)^{-1} \left( R_{[ii]}- \Pi_{[ii]}^w\right)(\Pi_{[ii]}^w)^{-1}&=(\Pi_{[ii]}^w)^{-1}\left[\left( (\Pi_{[ii]}^w)^{-1} +\OO_\prec ( t\Psi  +t^{1/2}n^{-1/2} ) \right)^{-1}-\Pi_{[ii]}^w\right](\Pi_{[ii]}^w)^{-1} \nonumber\\
 &\prec t\Psi  +t^{1/2}n^{-1/2}.\label{Rii}
 \end{align}
 These two estimates give the diagonal estimates in \eqref{entry_laww}. Combining \eqref{Rmumu} and \eqref{Rii} with \eqref{Zestimate2}, we can also obtain the off-diagonal estimates, which conclude the entrywise local law \eqref{entry_laww}.
 
Next we prove the anisotropic local law \eqref{aniso_law} using \eqref{entry_laww}. Due to the polarization identity 
$$\bu^\top M \bv = \frac{1}{2}\left(\bu+\bv\right)^\top M \left(\bu+\bv\right)-  \frac{1}{2}\left(\bu-\bv\right)^\top M \left(\bu-\bv\right),$$
for any symmetric matrix $M$, it suffices to take $\bu=\bv$ in \eqref{aniso_laww}. For any vector $\bu\in \R^{p+n}$ and $i\in \cal I_1$, we denote $u_{[i]}:=\begin{pmatrix}u_i \\ u_{\overline i} \end{pmatrix}$. Then with \eqref{entry_laww}, we have that for any deterministic unit vector $\bu\in \R^{p+n}$,
\begin{equation}\label{eq_iso_1}
\begin{split}
\left|  \bu ^\top (\Pi^w)^{-1}(R(z)-\Pi^w(z))(\Pi^w)^{-1}\bu \right| \prec t\Psi  +t^{1/2}n^{-1/2} + \Big| \sum_{i\ne j \in \cal I_1}  u^\top_{[i]}  \left[(\Pi^w)^{-1}R(\Pi^w)^{-1}\right]_{[ij]} u_{[j]}  \Big| \\
+  \Big| {\sum_{\mu \ne \nu \ge 2p+1} {u_{ \mu} \left[(\Pi^w)^{-1}R(\Pi^w)^{-1}\right]_{ \mu\nu } u_{\nu} } } \Big|   +2 \Big| {\sum_{i\in \cal I_1 ,\mu \ge 2p+1 }  u_{[i]}^\top \left[(\Pi^w)^{-1}R(\Pi^w)^{-1}\right]_{[i] \mu} u_{ \mu } } \Big| .
\end{split}
\ee
We need to bound each sum on the right hand side. With Markov's inequality, it reduces to proving the following high moment bounds: for any fixed $a\in \N$,
\begin{align}
& \mathbb E\Big|  \sum_{i\ne j \in \cal I_1}  u^\top_{[i]}  \left[(\Pi^w)^{-1}R(\Pi^w)^{-1}\right]_{[ij]} u_{[j]} \Big|^{2a} \prec \left(t\Psi  +t^{1/2}n^{-1/2}\right)^{2a}; \label{eqiso1} \\
& \mathbb E\Big| {\sum_{\mu \ne \nu \ge 2p+1} {u_{ \mu} \left[(\Pi^w)^{-1}R(\Pi^w)^{-1}\right]_{ \mu\nu } u_{\nu} } }\Big|^{2a} \prec \left(t\Psi  +t^{1/2}n^{-1/2}\right)^{2a}; \label{eqiso2}  \\
&\mathbb E \Big| {\sum_{i\in \cal I_1 ,\mu \ge 2p+1 }  u_{[i]}^\top \left[(\Pi^w)^{-1}R(\Pi^w)^{-1}\right]_{[i] \mu} u_{ \mu } }  \Big|^{2a} \prec \left(t\Psi  +t^{1/2}n^{-1/2}\right)^{2a}.\label{eqiso3} 
\end{align}
The proof of these estimates is based on a polynomialization method developed in \cite[section 5]{isotropic}. The core of the proof is some combinatorial results, which have been discussed in details in \cite{isotropic,XYY_circular,XYY_VESD,yang_CCA3}. In particular, the setting in \cite{XYY_circular} is most close to the setting here, where the main difference is that the $d_i$'s are all replaced with a fixed complex number $w\in \C$ in \cite{XYY_circular}. However, in the proof we only need the bound $|w|=\OO(1)$ in the proof, which also holds for $d_i$'s, i.e. $\max_i |d_i|=\OO(1)$. So we omit the details. This concludes the proof of \eqref{aniso_laww}, which further implies \eqref{aniso_law}.

\section{Proof of Theorem \ref{thm_local2}}\label{sec localregular}

As in the proof of Theorem \ref{thm_local}, in order to prove Theorem \ref{thm_local2}, it suffices to use the resolvent $R(z)$ in \eqref{eqn_defGd} and prove \eqref{averin2} and \eqref{averout2} for $m$ and $\mm$ in \eqref{defn_md}. 

Corresponding to \eqref{eq_defnzeta} and \eqref{simplePhizeta}, we define the following functions with $m_{c,t}$:
\be\label{eq_zetaex}
\zeta_c\equiv \zeta_c(t):= z(1+c_ntm_{c,t})^2 -(1-c_n)t(1+c_ntm_{c,t}),\ee
and
$$\Phi_c(\zeta)\equiv \Phi_{c}(\zeta,t):=\zeta(1-c_nt m_{c}(\zeta))^2+(1-c_n)t(1-c_ntm_{c}(\zeta)).$$
First with Theorem \ref{thm_local}, we can prove that Theorem \ref{thm_local2} holds for $t\gg n^{-1/3}$.
\begin{proposition}\label{prop_local}
For any fixed constants $\e_0>0$, Theorem \ref{thm_local2} holds for all $t\ge n^{-1/3+\e_0}$.
\end{proposition}
\begin{proof}
  As a consequence of the square root behavior of $\rho_c$ around $\lambda_+$ in \eqref{sqrtrhoc}, we get that for any $z$ such that $\lambda_+ - 3c_V/4\le E \le \lambda_+ - 3c_V/4$ and $0< \eta \le 10$, 
\begin{equation}\label{Immc_c}
 \im m_c(z) \sim \begin{cases}
    {\eta}/{\sqrt{\kappa+\eta}}, & \text{ if } E\geq \lambda_+ \\
    \sqrt{\kappa+\eta}, & \text{ if } E \le \lambda_+\\
  \end{cases}.
\end{equation}
Let $\eta_*=n^{-2/3+\delta}$ for a sufficiently small constant $\delta>0$. Combining \eqref{Immc_c} with Assumption \ref{assm regularW}, we obtain that $W$ is $\eta_*$-regular. Now using Theorem \ref{thm_local}, we get that for all $t\ge n^{-1/3+\e_0}$, the averaged local laws \eqref{averin} and \eqref{averout} hold. It remains to show that $m_{w,t}$ and $m_{c,t}$ are close to each other. More precisely, we will show that for $z\in \cal D^c_\vartheta$, 
\be\label{averin3}
|m_{w,t}-m_{c,t}|\prec \frac1{n\eta},\quad \text{for }\ \ E\le\lambda_{c,t},
\ee
and
\be\label{averout3}
|m_{w,t}-m_{c,t}|\prec \frac{1}{n(\kappa+\eta)}+\frac{1}{(n\eta)^2\sqrt{\kappa+\eta}},\quad \text{for }\ \ E\ge\lambda_{c,t}.
\ee
Combining these estimates with \eqref{averin} and \eqref{averout}, we conclude Proposition \ref{prop_local}.

 Recall that $\zeta_t$ and $\zeta_{c}$ satisfy $\Phi_t(\zeta_t)=z$ and $\Phi_c(\zeta_c)=z$, respectively. Thus we have the equation
\be\label{deterself}
0=\Phi_t(\zeta_t)- \Phi_c(\zeta_c).
\ee
We first consider the case $\eta\ge c_0$ for some constant $c_0>0$. In this case, using Assumption \ref{assm regularW} it is easy to check the following estimates:
$$|\zeta_t|\sim \im \zeta_t \sim 1, \quad |\zeta_c|\sim \im \zeta_c\sim  1,\quad | \zeta_t - \zeta_c|\gtrsim t|m_{w,t}-m_{c,t}|,$$
and
$$ |\Phi_t(\zeta_t)-\Phi_t(\zeta_c)|\gtrsim |\zeta_t-\zeta_c|,\quad   |\Phi_c(\zeta_t)-\Phi_c(\zeta_c)|\gtrsim |\zeta_t-\zeta_c|, \quad |\Phi_t(\zeta_t)-\Phi_c(\zeta_t)|+|\Phi_t(\zeta_c)-\Phi_c(\zeta_c)| \prec tn^{-1}.$$
Thus from equation \eqref{deterself} we get
\be\label{etasim1}
t|m_{w,t}-m_{c,t}|\lesssim |\zeta_t-\zeta_c| \lesssim  |\Phi_c(\zeta_t)-\Phi_c(\zeta_c)| =|\Phi_t(\zeta_t)-\Phi_c(\zeta_t)| \prec tn^{-1}.
\ee
This concludes \eqref{averin2} and \eqref{averout2} for the $\eta\gtrsim 1$ case.

 Now we consider the case $E\le \lambda_{c,t}+ C_1t^2$ for some constant $C_1>0$. As in the proof of Proposition \ref{region11}, we first assume that the following estimate holds:
\be\label{priorassm1}|m_{w,t}-m_{c,t}|\le \frac{t+\im m_{c,t}}{(\log n)^2}.\ee
Under this estimate, we have $|x-\zeta_c|\gg |\zeta_c-\zeta_t|$ for $x\in \supp \rho_{c,t}$ by \eqref{eq_dtheta1control}. Then we rewrite \eqref{deterself} as
\be\label{trivialPhieq}\left[\Phi_t(\zeta_t)- \Phi_c(\zeta_t) \right]+\left[\Phi_c(\zeta_t) - \Phi_c(\zeta_c)\right]=0 .\ee
Now performing the Taylor expansion to $\Phi_c(\zeta_t) - \Phi_c(\zeta_c)$ as in \eqref{Gii5}, and applying Assumption \ref{assm regularW} to $\Phi_t(\zeta_t)- \Phi_c(\zeta_t)$, we obtain from \eqref{trivialPhieq} that 
\be\label{selfdeter1}\left|\Phi_c'(\zeta_c) (\zeta_t-\zeta_c) +\frac12 \Phi_c''(\zeta_c) (\zeta_{t}-\zeta_c)^2\right|  \lesssim   \frac{t^2 +t \sqrt{\kappa+\eta}}{(t^2 + \im \zeta_c )^{3}} \cdot t^3 |m_{w,t}-m_{c,t}|^3 + \OO_\prec\left( \frac{t}{n\im \zeta_c}\right),\ee
where we used \eqref{secondderiv}. For $\kappa+\eta\ge \tau_1 t^2$, we have $|\Phi_c'(\zeta_c)|\sim 1$. Then using Lemma \ref{lem derivPhi} and \eqref{selfdeter1}, we obtain that
\be\label{selfprior1}
t|m_{w,t}-m_{c,t}|\lesssim |\zeta_t-\zeta_c| \prec \frac{t}{n\im \zeta_t}  \ \Rightarrow \ |m_{w,t}-m_{c,t}|\prec \frac{1}{n(\eta+ t\sqrt{\kappa+\eta})},
\ee
as long as \eqref{priorassm1} holds. Note that \eqref{selfprior1} implies \eqref{priorassm1} since $(n\eta)^{-1} \le n^{-\vartheta}\sqrt{\kappa+\eta} \le (\log n)^{-2}(t+\im m_{c,t})$ for  $z\in \cal D^c_\vartheta$ with $E\le \lambda_{c,t}+ C_1t^2$. Thus starting from \eqref{etasim1} for the $\eta\gtrsim 1$ case, using a standard continuity argument, we obtain that \eqref{selfprior1} holds for $\kappa+\eta\ge \tau_1 t^2$ without assuming \eqref{priorassm1}. On the other hand, if $\kappa+\eta\le \tau_1 t^2$, we have $|\Phi_c'(\zeta_c)|\sim \sqrt{\kappa+\eta}/t$ and $|\Phi_c''(\zeta_c)|\sim t^{-2}$ by Lemma \ref{lem derivPhi}. Suppose that the following estimate holds:
\be\label{priorassm2}|m_{w,t}-m_{c,t}|\le \frac{\sqrt{\kappa+\eta}}{(\log n)^2}.\ee
Then from \eqref{selfdeter1}, we obtain that
\be\label{selfprior2}
|m_{w,t}-m_{c,t}|\lesssim   \frac{|m_{w,t}-m_{c,t}|^2}{\sqrt{\kappa+\eta}} + \OO_\prec\left( \frac{t}{n\im \zeta_c\cdot \sqrt{\kappa+\eta}}\right) \ \Rightarrow \ |m_{w,t}-m_{c,t}|\prec \frac{1}{n(\kappa+\eta)}.
\ee
Note that \eqref{selfprior2} implies both \eqref{priorassm1} and \eqref{priorassm2}. Moreover, we have shown that \eqref{priorassm1} and \eqref{priorassm2} hold when $\kappa+\eta=\tau_1 t^2$. Thus using a standard continuity argument, we obtain that \eqref{selfprior2} holds for $\kappa+\eta\le \tau_1 t^2$ without assuming \eqref{priorassm1} and \eqref{priorassm2}. Combining \eqref{selfprior1} and \eqref{selfprior2}, we conclude \eqref{averin3} and \eqref{averout3} for $E\le \lambda_{c,t}+ C_1t^2$ (here we also used that $\eta+ t\sqrt{\kappa+\eta}\gtrsim \kappa+\eta$ for $\lambda_{c,t}\le E\le \lambda_{c,t}+ C_1t^2$). 

Next we consider the case $E\ge \lambda_{c,t}+C_1 t^2$. We first assume that the following estimate holds:
\be\label{priorassm33}|m_{w,t}-m_{c,t}|\le \frac{t}{(\log n)^2}.\ee
Under this estimate, we have $|x-\zeta_c|\gg |\zeta_c-\zeta_t|$ for $x\in \supp \rho_{c,t}$. As above, using \eqref{firstsecondderiv} and Assumption \ref{assm regularW} we can derive from \eqref{trivialPhieq} that
\be\label{selfdeter2}
\left|\Phi_c'(\zeta_c) (\zeta_t-\zeta_c)\right|  \lesssim   \frac{t}{(t^2 + \kappa+\eta)^{3/2}} \cdot t^2 |m_{w,t}-m_{c,t}|^2 + \OO_\prec\left( \frac{t}{n|\zeta_c-\lambda_+|}+ \frac{t}{(n\im \zeta_c)^2|\zeta_c-\lambda_+|^{1/2} }\right).\ee
Together with $|\Phi_c'(\zeta_c)|\sim 1$ and $|\zeta_c-\lambda_+|\gtrsim  t^2 + \kappa +\eta $ by \eqref{eq_dtheta2bound1}, \eqref{selfdeter2} implies that 
\be\label{selfprior33}
t|m_{w,t}-m_{c,t}|\lesssim |\zeta_t-\zeta_c| \prec \frac{t}{n(\kappa+\eta)}+\frac{t}{(n\eta)^2\sqrt{\kappa+\eta}}\Rightarrow |m_{w,t}-m_{c,t}| \prec \frac{1}{n(\kappa+\eta)}+\frac{1}{(n\eta)^2\sqrt{\kappa+\eta}} .
\ee
Note that \eqref{selfprior33} implies \eqref{priorassm33} since $\eta\ge n^{-2/3+\vartheta}$ and $t\gg n^{-1/3}$. Moreover, we have shown that \eqref{selfprior33} holds for $E=\lambda_{c,t}+C_1t^2$. Hence using a standard continuity argument, we obtain that \eqref{selfprior33} holds for $E\ge \lambda_{c,t}+C_1 t^2$ without assuming \eqref{priorassm33}.
This concludes \eqref{averout3}  for $E\ge \lambda_{c,t}+C_1 t^2$.
\end{proof}

Hence it remains to consider the case where
\be\label{smalltrange}0\le t\le n^{-1/3+\e_0},\quad 0<\e_0\le \vartheta/100.\ee
Define the spectral domains
\begin{equation}\label{eq_dthetanot1}
\widehat{\mathcal{D}}_{\vartheta}^1:=\mathcal{D}^c_{\vartheta} \cap \{E+\ii \eta: E \leq \lambda_{c,t}+n^{-2/3+\vartheta}\},\quad \widehat{\mathcal{D}}_{\vartheta}^2:=\mathcal{D}^c_{\vartheta} \cap \{E+\ii \eta: E > \lambda_{c,t}+n^{-2/3+\vartheta}\}.
\end{equation}
Then we have the following deterministic estimates. 

\begin{lemma}
Let $0\le t \le n^{-1/3+\e_0}$. For $z\in \widehat{\mathcal{D}}_{\vartheta}^1$, we have 
\be\label{Dtheta1est}
|d_i-\zeta_c| \gtrsim \eta+t \sqrt{\kappa+\eta},
\ee
and for any fixed $a\ge 2$,
\begin{equation} \label{Dtheta1est2}
\frac1p\sum_{i=1}^p\frac1{|d_i-\zeta_c|^a} \lesssim \frac{\sqrt{\kappa+\eta}}{(\eta+t\sqrt{\kappa+\eta})^{a-1}}. 
\end{equation}
On the other hand, for $z\in \widehat{\mathcal{D}}_{\vartheta}^2$, we have 
\be\label{Dtheta2est} |d_i-\zeta_c| \gtrsim \kappa+\eta, \ee
and for any fixed $a\ge 2$,
\begin{equation} \label{Dtheta2est2}
\frac1p\sum_{i=1}^p\frac1{|d_i-\zeta_c|^a} \lesssim    \frac{1}{(\kappa+\eta)^{a-3/2}}. 
\end{equation}
\end{lemma}
\begin{proof}
By \eqref{Immc_c}, we see that $m_c$ is $\eta_*$-regular with $\eta_*=0$. So the analysis in Section \ref{sec anafree} goes through for any $0< t\ll 1$. In particular, by \eqref{eq_imasymptoics} we have
$$\im m_{c,t} \gtrsim \sqrt{\kappa+\eta},\quad \text{for}\ \ z\in \widehat{\mathcal{D}}_{\vartheta}^1.$$
Then using \eqref{eq_zetaex}, we can get
$
|d_i-\zeta_c| \geq \im \zeta_c \gtrsim \eta+t \sqrt{\kappa+\eta}.
$
 Furthermore, by (\ref{eq_ccc}) we get that 
\begin{equation*}
 |\zeta_c-\lambda_+| \lesssim t^ 2+ \kappa +\eta \lesssim  \kappa+\eta 
\end{equation*}
for $z\in \widehat{\mathcal{D}}_{\vartheta}^1$ and $t\le n^{-1/3+\e_0}$. Consequently, using Lemma \ref{lem_immanalysis} we obtain that 
\begin{equation*}
\frac1p\sum_{i=1}^p\frac1{|d_i-\zeta_c|^a} \lesssim \frac{|\xi_c|^{1/2}}{(\im \xi_c)^{a-1}} \lesssim \frac{\sqrt{\kappa+\eta}}{(\eta+t\sqrt{\kappa+\eta})^{a-1}}. 
\end{equation*}

In the domain $\widehat{\mathcal{D}}_{\vartheta}^2$, by Lemma \ref{lem_xigeneral} we have that if $\kappa\ge C\eta$ for some large enough constant $C>0$, then $|\re \zeta_c -d_i|\ge  c\kappa$ for some constant $c>0$, which gives \eqref{Dtheta2est}. Furthermore, with \eqref{eq_boundfarawayregion} we can get \eqref{Dtheta2est2}.
On the other hand, if $\kappa\le C\eta$, then with \eqref{eq_imasymptoics} we get
$$ |d_i-\zeta_c| \geq \im \zeta_c \gtrsim \eta \gtrsim \kappa+\eta,$$
and with \eqref{eq_boundfarawayregion} we can conclude \eqref{Dtheta2est2}.
\end{proof}  

Similar to the proof of Theorem \ref{thm_local}, we divide the proof into two parts according to the two regions $\wh{\cal D}_{\vartheta}^{1}$ and $\wh{\cal D}_{\vartheta}^{2}$. We first prove the following proposition.


%

\begin{proposition}\label{region21}
 Theorem \ref{thm_local2} holds for $z\in \wh{\cal D}_{\vartheta}^1$. 
\end{proposition}

The strategy for the proof of Proposition \ref{region21} is similar to the one for Propositions \ref{region11}, that is, we first derive a self-consistent equation estimate, with which we can prove a weak averaged local law. Then using a fluctuation averaging estimate together with a stronger self-consistent equation estimate, we can prove the strong local law in Proposition \ref{region21}. 

As in \eqref{eqn_randomerror} and \eqref{eq_defpsitheta}, we define
$$ \theta(z):= |m-m_{c,t}| ,\quad  \Psi _\theta(z)  : =  \sqrt {\frac{{\Im \, m_{c,t}(z)  + \theta(z) }}{{n\eta }}} + \frac{1}{n\eta}.$$
Moreover, we define the asymptotic matrix limit of $R$ as
\be \label{defnPic}
\Pi^c(z):=\begin{bdmatrix} \frac{-(1+c_n t m_{c,t})}{z(1+c_n t m_{c,t})(1+t \underline m_{c,t} )-WW^\top} &  \frac{-z^{-1/2} }{z(1+c_n t m_{c,t})(1+t \underline m_{c,t} )-WW^\top}W \\ 
W^\top \frac{-z^{-1/2} }{z(1+c_n t m_{c,t})(1+t \underline m_{c,t} )-WW^\top} & \frac{-(1+t \underline m_{c,t} )}{z(1+c_n t m_{c,t})(1+t \underline m_{c,t} )-W^\top W}  ,
\end{bdmatrix},
\ee
where as in \eqref{undeline m} we denoted
\be \nonumber 
\mm_{c,t}: = c_n m_{c,t} - \frac{1-c_n}{z}.\ee
First we claim the following self-consistent equation estimates. Recall that $u_t$ was defined in \eqref{defn ut}.

\begin{lemma}\label{lemm_selfcons_weak3}
For $z\in \wh{\cal D}_{\vartheta}^1$ and on the event 
$$\Xi:=\left\{  \theta \le \frac{\im m_{c,t}}{(\log n)^2}\right\},$$
we have that for any constant $\e>0$, 
\begin{equation}
 \left| \Phi_c'(\zeta_c) (u_{t}-\zeta_c) \right|\lesssim  \frac{t\theta^2}{\sqrt{\kappa+\eta}} +n^\e t  \Psi_\theta,   \label{selfcons_lemm3}
\end{equation}
with high probability. Moreover, we have the finer estimate: for any constant $\e>0$,
\begin{equation}
\left| \Phi_c'(\zeta_c) (u_{t}-\zeta_c)\right|\lesssim  \frac{t\theta^2}{\sqrt{\kappa+\eta}}  + n^\e t \| [ Z]\| + \frac{n^\e t}{n\eta},   \label{selfcons_improved3}
\end{equation}
with high probability.
\end{lemma}
\begin{proof}
Note that by \eqref{Dtheta1est}, we have that on $\Xi$,
\be\label{mclose2}
|u_{t}-\zeta_c| \lesssim t\theta\lesssim \left\{ \frac{\im \zeta_c}{(\log n)^2}, \frac{\min_{i=1}^p |d_i-\zeta_c| }{(\log n)^2}\right\}. 
\ee
Then we can repeat the proof of Lemma \ref{lemm_selfcons_weak}. In particular, using \eqref{Dtheta1est2} we can get a similar estimate as in \eqref{Gii3}:
\be\label{Gii3(3)}
\frac{1}{p}\sum_{i\in \cal I_1}R_{[ii]}=\frac{1}{p}\sum_{i\in \cal I_1} \pi_{[i]} - [Z] + \OO_\prec\left[ \frac{t\sqrt{\kappa+\eta}}{(\eta+t\sqrt{\kappa+\eta})^{2}}\left(\theta \left( t\Psi_\theta +\frac{t^{1/2}}{n^{1/2}}\right)+ t \Psi^2_\theta + \frac 1n \right)\right].
\ee
Then estimate \eqref{Giim1t} holds with 
$$\cal E:= \| [ Z]\| + \frac{t\sqrt{\kappa+\eta}}{(\eta+t\sqrt{\kappa+\eta})^{2}}\left(\theta \left( t\Psi_\theta +\frac{t^{1/2}}{n^{1/2}}\right)+ t \Psi^2_\theta + \frac 1n \right).
$$
Repeating the proof below \eqref{Giim1t}, we can get \eqref{Gii4}. Now comparing \eqref{Gii4} with equation $z=\Phi_c(\zeta_c)$, we get that
\be\label{eq compareTayloe0} 
 | \Phi_c(u_{t}) - \Phi_c(\zeta_{c})|=|\Phi_t(u_{t})-\Phi_c(u_{t})| + \OO_\prec(t\cal E)\prec \frac{t}{n\im \zeta_c} + t\cal E,
\ee
where we used Assumption \ref{assm regularW} and \eqref{mclose2} in the second step. Then performing Taylor expansion and using \eqref{Dtheta1est2}, we obtain that for any constant $\e>0$,
\be \nonumber 
\begin{split}
& \left|\Phi_c'(\zeta_c) (u_{t}-\zeta_c) \right|  \lesssim \frac{t\sqrt{\kappa+\eta}}{(\eta +t\sqrt{\kappa+\eta})^{2}} t^2\theta^2 + \OO_\prec\left( t\cal E + \frac{t}{n\im \zeta_c} \right) \\
& \lesssim \frac{t\theta^2}{ \sqrt{\kappa+\eta} } + \OO_\prec\left( t \| [ Z]\| + \frac{n^\e t\sqrt{\kappa+\eta}}{(\eta+t\sqrt{\kappa+\eta})^{2}}\left(t^2 \Psi^2_\theta + \frac tn \right)+\frac{t}{n\eta}\right) =\frac{t\theta^2}{ \sqrt{\kappa+\eta} } + \OO_\prec\left( t \| [ Z]\|    +\frac{n^\e t}{n\eta}\right),
 \end{split}
 \ee
where we used the definition of $\Psi_\theta$ in the last step. For the term $[Z]$, we can bound it as in \eqref{ZAZ} and \eqref{ZBZ}: 
\begin{align}\nonumber
\frac1p\sum_{i\in \cal I_1}\Pi_{[ii]}^c A_{[i]}\Pi_{[ii]}^c \prec \frac1n\left( \frac1p\sum_i \frac{t}{|d_i -\zeta_c|^4}\right)^{1/2} \lesssim \frac{1}{n}\left( \frac{t\sqrt{\kappa+\eta}}{(\eta+t\sqrt{\kappa+\eta})^{3}}\right)^{1/2} \lesssim \frac{1}{n\eta},
\end{align}
and 
\begin{align}\nonumber
\frac1p\sum_{i\in \cal I_1}\Pi_{[ii]}^cB_{[i]}\Pi_{[ii]}^c \prec \frac1p\sum_i \frac{t\Psi_{\theta} }{|d_i -\zeta_c|^2}\lesssim \frac{t\sqrt{\kappa+\eta}}{\eta+t\sqrt{\kappa+\eta} }\Psi_\theta\lesssim \Psi_\theta .
\end{align}
Plugging these two estimates into \eqref{selfcons_improved3} we conclude \eqref{selfcons_lemm3}.
\end{proof}

With \eqref{selfcons_lemm3}, we can prove the following weak local law as in Lemma \ref{alem_weak}.
\begin{lemma}[Weak local law]\label{alem_weak3} 
For any $z\in \wh{\cal D}_{\vartheta}^1$, we have $\theta \prec  \Psi(z) .$ 
\end{lemma}
\begin{proof}
Note that by \eqref{Phit'}, we have $ |\Phi_c'(\zeta_c)| \sim \min \left\{1, \sqrt{\kappa+\eta}/t \right\} \sim 1$, since $\sqrt{\kappa+\eta}\gg t$ for the choice of $t$ in \eqref{smalltrange}. Thus if $\Xi$ holds, then with \eqref{selfcons_lemm3} we immediately obtain that 
$$t\theta \lesssim |u_{t}-\zeta_c| \prec t\Psi_\theta \ \Rightarrow \theta \prec  \Psi(z).$$
To complete the proof, we will use the estimate for the $\eta\gtrsim 1$ case together with a standard continuity argument as the one for Proposition 5.6 of \cite{edgedbm}. We omit the details.
\end{proof}

Then we have the following fluctuation averaging estimate for $[Z]$. 

\begin{lemma}[Fluctuation averaging] \label{abstractdecoupling3}
Suppose that $\theta\prec \Phi$, where $\Phi$ is a positive, $n$-dependent deterministic function on $\wh{\cal D}_{\vartheta}^1$ satisfying that
$$\frac{1}{n\eta} \le \Phi \le  \frac{\im m_{c,t}}{(\log n)^2}.$$   
Then for any $z \in \wh{\cal D}_{\vartheta}^1$, we have
\begin{equation}\label{flucaver_ZZ3}
\left|[Z]\right|  \prec  \frac{1}{n\eta} .
\end{equation}
\end{lemma}
\begin{proof}
The proof is similar to the one for Lemma 6.6 of \cite{edgedbm}. We omit the details. 
\end{proof}

With this estimate, we can immediately conclude the proof of Proposition \ref{region21}.

\begin{proof}[Proof of Proposition  \ref{region21}]
With the stronger self-consistent equation estimate (\ref{selfcons_improved3}) and Lemma \ref{abstractdecoupling3}, we immediately get $\theta\prec  (n\eta)^{-1}.$ This concludes Proposition  \ref{region22} by noticing that $(n(\kappa+\eta))^{-1}\lesssim (n\eta)^{-1}$ for $z\in \wh{\cal D}_\vartheta^{1}$ with $E\ge \lambda_{c,t}$.
\end{proof}


Next we prove \eqref{averout2} on the domain $\wh{\cal D}_{\vartheta}^2$.

\begin{proposition}\label{region22}
Theorem \ref{thm_local2} holds for $z\in \wh{\cal D}_{\vartheta}^2$. 
\end{proposition}

Again we first prove the following self-consistent equation estimates. 

\begin{lemma}\label{lemm_selfcons_weak4}
 For $z\in \wh{\cal D}_{\vartheta}^2$ and on the event  $\Xi:=\left\{  \theta \le n^{-1/3}\right\},$
we have that for any constant $\e>0$, 
\begin{equation}
 \left| \Phi_c'(\zeta_c) (u_{t}-\zeta_c) \right|\le \frac{n^\e t}{n^{1/2}(\kappa+\eta)^{1/4}}+\frac{n^\e t}{n\eta}  \label{selfcons_lemm4}
\end{equation}
with high probability. Moreover, we have the finer estimate: for any constant $\e>0$ 
\begin{equation}
\left| \Phi_c'(\zeta_t) (u_{t}-\zeta_c)\right|\le  \frac{t\theta^2}{\sqrt{\kappa+\eta}}  + n^\e t \| [ Z]\| + \frac{n^\e t}{(n\eta)^2\sqrt{\kappa+\eta}}+\frac{n^\e t}{n(\kappa+\eta)},    \label{selfcons_improved4}
\end{equation}
with high probability.
\end{lemma}
\begin{proof}
Note that by \eqref{Dtheta2est} and $\min\{\kappa,\eta\}\ge n^{-2/3+\vartheta}$ for $z\in \wh{\cal D}_{\vartheta}^2$, we have that \eqref{mclose2} holds on $\Xi$. Then we can repeat the proof of Lemma \ref{lemm_selfcons_weak}. In particular, using \eqref{Dtheta2est2} we can obtain a similar estimate as in \eqref{Gii3}: 
\be\label{Gii3(4)}
\frac{1}{p}\sum_{i\in \cal I_1}R_{[ii]}=\frac{1}{p}\sum_{i\in \cal I_1} \pi_{[i]} - [Z] + \OO_\prec\left[ \frac{t}{(\kappa+\eta)^{3/2}}\left(\theta \left( t\Psi_\theta +\frac{t^{1/2}}{n^{1/2}}\right)+ t \Psi^2_\theta + \frac 1n \right)\right].
\ee
Then estimate \eqref{Giim1t} holds with 
$$\cal E:= \| [ Z]\| +\frac{t}{(\kappa+\eta)^{3/2}}\left(\theta \left( t\Psi_\theta +\frac{t^{1/2}}{n^{1/2}}\right)+ t \Psi^2_\theta + \frac 1n \right).
$$
Repeating the proof below \eqref{Giim1t}, we can get \eqref{Gii4}. Now we subtract $z=\Phi_c(\zeta_c)$ from \eqref{Gii4} and get that
\be\label{eq compareTayloe02} 
 | \Phi_c(u_{t}) - \Phi_c(\zeta_{c})|=|\Phi_t(u_{t})-\Phi_c(u_{t})| + \OO_\prec(t\cal E)\prec \frac{t}{(n\eta)^2\sqrt{\kappa+\eta}}+\frac{ t}{n(\kappa+\eta)} + t\cal E,
\ee
where in the second step we used Assumption \ref{assm regularW} and  $|\lambda_+-u_t|\sim  |\lambda_+-\zeta_c| \gtrsim \kappa+\eta$ by \eqref{Dtheta2est}. With \eqref{Dtheta2est2}, it is easy to check that
$$|\Phi_t''(u)|\lesssim \frac{t}{(\kappa+\eta)^{3/2}}, $$
for all $u$ lying between $u_t$ and $\zeta_c$. 
Then performing Taylor expansion to \eqref{eq compareTayloe02}, we obtain that for small enough constant $\e>0$,
\be \nonumber 
\begin{split}
 &\left|\Phi_c'(\zeta_c) (u_{t}-\zeta_c) \right|  \lesssim \frac{t}{(\kappa+\eta)^{3/2}} t^2\theta^2 + \OO_\prec\left( t\cal E + \frac{t}{(n\eta)^2\sqrt{\kappa+\eta}}+\frac{ t}{n(\kappa+\eta)}\right) \\
& \le \frac{n^{-\vartheta/2}t\theta^2}{ \sqrt{\kappa+\eta} } + \OO_\prec\left( t \| [ Z]\| + \frac{n^\e t}{(\kappa+\eta)^{3/2}}\left(t^2 \Psi^2_\theta + \frac tn \right)+ \frac{t}{(n\eta)^2\sqrt{\kappa+\eta}}+\frac{ t}{n(\kappa+\eta)}\right)\\
&\le \frac{t\theta^2}{ \sqrt{\kappa+\eta} } + \OO_\prec\left( t \| [ Z]\| + \frac{t}{(n\eta)^2\sqrt{\kappa+\eta}}+\frac{ t}{n(\kappa+\eta)}\right),
 \end{split}
 \ee
 where we used the definition of $\Psi_\theta$ is the last step. This concludes \eqref{selfcons_improved4}. For the term $[Z]$, we can bound it as in \eqref{ZAZ} and \eqref{ZBZ}: 
\begin{align}\nonumber
\frac1p\sum_{i\in \cal I_1}\Pi_{[ii]}^c A_{[i]}\Pi_{[ii]}^c \prec \frac1n\left( \frac1p\sum_i \frac{t}{|d_i -\zeta_c|^4}\right)^{1/2} \lesssim \frac{1}{n}\left( \frac{t }{(\kappa+\eta)^{5/2}}\right)^{1/2} \le  \frac{1}{n(\kappa+\eta)},
\end{align}
and 
\begin{align}\nonumber
\frac1p\sum_{i\in \cal I_1}\Pi_{[ii]}^c B_{[i]}\Pi_{[ii]}^c\prec \frac1p\sum_i \frac{t\Psi_{\theta} }{|d_i -\zeta_c|^2}\lesssim \frac{t }{\sqrt{\kappa+\eta} }\Psi_\theta \le  \frac{1}{n^{1/2}(\kappa+\eta)^{1/4}}+\frac{1}{n\eta}.
\end{align}
Plugging these two estimates into \eqref{selfcons_improved4}, we conclude \eqref{selfcons_lemm4}.
\end{proof}

With this lemma, we can prove the following weak local law.
\begin{lemma}[Weak local law]\label{alem_weak4} 
For all $z\in \wh{\cal D}_{\vartheta}^2$, we have
\be\label{eq alem_weak4}\theta \prec  \frac{1}{n^{1/2}(\kappa+\eta)^{1/4}}+\frac{1}{n\eta}  .\ee
\end{lemma}
\begin{proof}
In this case we have $ |\Phi_c'(\zeta_c)| \sim 1$. Thus if $\Xi$ holds, with \eqref{selfcons_lemm4} we immediately obtain \eqref{eq alem_weak4}. To complete the proof, we will use the estimate at $\eta\gtrsim 1$ together with a standard continuity argument as the one for Proposition 5.6 of \cite{edgedbm}. We omit the details.
\end{proof}

Then we have the following fluctuation averaging estimate for $[Z]$. 

\begin{lemma}[Fluctuation averaging] \label{abstractdecoupling4}
Suppose that $\theta\prec \Phi$, where $\Phi$ is a positive, $n$-dependent deterministic function on $\wh{\cal D}_{\vartheta}^2$ satisfying that
$$ (n\eta)^{-1} \le \Phi \le  n^{-1/3}.$$   
Then for any $z \in \wh{\cal D}_{\vartheta}^2$, we have
\begin{equation}\label{flucaver_ZZ4}
\left|[Z]\right|  \prec \frac{1}{n (\kappa+\eta)}+\frac{\Phi}{n\eta \sqrt{\kappa+\eta}} .
\end{equation}
\end{lemma}
\begin{proof}
The proof is similar to the one for Lemma 6.11 of \cite{edgedbm}. We omit the details. 
\end{proof}

Now we can complete the proof of Proposition \ref{region22}.

\begin{proof}[Proof of Proposition  \ref{region22}]
With the self-consistent equation estimate (\ref{selfcons_improved4}) and Lemma \ref{abstractdecoupling4}, we get that 
\be\label{simpleiter2}\theta\prec   \frac{1}{(n\eta)^2\sqrt{\kappa+\eta}}+\frac{1}{n(\kappa+\eta)} + \frac{\Phi}{n\eta \sqrt{\kappa+\eta}}.\ee
Iterating this estimate, we get $\theta\prec (n\eta)^{-1}$. Plugging $\Phi=n^\e /(n\eta)$ into \eqref{simpleiter2}, we conclude Proposition \ref{region22} since $\e$ can be arbitrarily small.
\end{proof}

Finally, combining Proposition \ref{region21}
and Proposition \ref{region22}, we conclude the proof of Theorem \ref{thm_local2}.

\section{Proof of Theorem \ref{thm_local_gen} and Theorem \ref{thm_implication}}\label{sec_pfgen}

With Theorem \ref{thm_local_gen}, we can prove Theorem \ref{thm_implication} with some standard arguments. Hence we will not write down all the details, and refer the reader to the known arguments in previous papers instead. 
\begin{proof}[Proof of Theorem \ref{thm_implication}]
The estimate \eqref{rigidity} follows from the averaged local laws \eqref{averin_gen} and \eqref{averout_gen} combined with a standard argument using Helffer-Sj\"ostrand calculus. The reader can refer to e.g. \cite[Theorem 2.13]{EKYY1}, \cite[Theorem 2.2]{EYY} and \cite[Theorem 3.3]{pillai2014} for more details.  

The estimate \eqref{edgeuniv} can be proved with a resolvent comparison argument as in the proof of \cite[Theorem 2.4]{EYY}. We have collected all the necessary inputs for this argument, including the rigidity of eigenvalues \eqref{rigidity}, the averaged local laws \eqref{averin_gen} and \eqref{averout_gen}, and the anisotropic local law \eqref{aniso_law_gen}. We remark that in order for these arguments to work, we need to know that the local laws hold at $z=E+\ii n^{-2/3-\delta}$ for $E$ around $\lambda_{+,t}$, where $\delta>0$ is some small constant. Our domain $\cal D_{\vartheta}$ contains such $z$ by the assumption $t\ge n^{\e/2}\sqrt{\eta_*} \ge n^{-1/3+\e/2}$ as long as $\vartheta$ is taken sufficiently small.

For \eqref{delocal}, we shall use the following spectral decomposition of $G$: for $i,j\in \cal I_1$ and $\mu,\nu\in \cal I_2$,
\begin{align}
G_{ij} = \sum_{k = 1}^{p} \frac{\bm{\xi}_k(i) \bm{\xi}_k^\top(j)}{\lambda_k-z},\ \quad \ &G_{\mu\nu} = \sum_{k = 1}^{n} \frac{\bm{\zeta}_k(\mu) \bm{\zeta}_k^\top(\nu)}{\lambda_k-z}. \label{spectral1}
\end{align}
Choose $\wt z_k=\lambda_k+\ii n^\e \eta_l(\lambda_k)$. Using \eqref{rigidity}, it is easy to check that $\wt z_k \in \cal D_\vartheta$. Then with (\ref{aniso_law_gen}) and \eqref{spectral1}, we obtain that
\begin{equation}\nonumber 
 \frac{ \vert   \mathbf u^\top \bm{\xi}_k \vert^2}{n^{\e}\eta_l(\lambda_k) }  \le \sum_{j=1}^p \frac{n^\e\eta_l(\lambda_k) \vert  \mathbf u^\top \bm{\xi}_k \vert^2}{(\lambda_j-\lambda_k)^2+n^{2\e}\eta_l^2(\lambda_k) }  = \im \mathbf u^\top {G}(\wt z_k) {\mathbf u}= \im \mathbf u^\top\Pi(\wt z_k) \mathbf u + \OO_\prec( \Phi(\wt z_k)\cdot \|\mathbf u\|_{\Pi}(\wt z_k)) .
\end{equation}
This estimate immediately gives that 
\begin{align*}\vert  \mathbf u^\top \bm{\xi}_k \vert^2 \le n^{\e}\eta_l(\lambda_k) \left[ \im \mathbf u^\top\Pi(\wt z_k) \mathbf u + \OO_\prec( \Phi(\wt z_k)\cdot \|\mathbf u\|_{\Pi}(\wt z_k)) \right]\\
\lesssim n^{\e}\eta_l(\gamma_k) \left[ \im \mathbf u^\top\Pi(z_k) \mathbf u+ \OO_\prec( \Phi(z_k)\cdot \|\mathbf u\|_{\Pi}(z_k)) \right]
\end{align*}
with high probability, where in the second step we used the rigidity estimate \eqref{rigidity} to replace $\lambda_k$ with $\gamma_k$. Since $\epsilon$ is arbitrary, we conclude the first estimate in \eqref{delocal}. The second estimate of \eqref{delocal} can be proved in the same way.
\end{proof}

For the rest of this section, we focus on the proof of Theorem \ref{thm_local_gen}. 
Our proof mainly uses a self-consistent comparison argument developed in \cite{Anisotropic}. We take the proof of the anisotropic local law \eqref{aniso_law_gen} as an example. Since Theorem \ref{thm_local} already implies that \eqref{aniso_law} holds when $X$ is Gaussian, it suffices to prove that for $X$ satisfying the assumptions in Theorem \ref{thm_local_gen}, 
\begin{equation*}
\left| \mathbf u^\top \left[ G(X,z) -  G(X^{Gauss},z)\right] \mathbf v\right| \prec \Phi(z)\|\mathbf u\|_\Pi^{1/2} \|\mathbf v\|_\Pi^{1/2} ,
\end{equation*}
for any deterministic unit vectors $\mathbf u,\mathbf v\in{\mathbb R}^{\mathcal I}$, where the entries of $X^{Gauss}$ are i.i.d. Gaussian random variables satisfying \eqref{assm1}. For simplicity of notations, in the proof we shall use the following notion of generalized entries. For $\mathbf v,\mathbf w \in \mathbb R^{\mathcal I}$ and $\fa\in \mathcal I$, we denote
\begin{equation}
G_{\mathbf{vw}}:= \mathbf v^\top G\mathbf w , \quad G_{\mathbf{v}a}:= \mathbf v^\top G\mathbf e_a , \quad G_{a\mathbf{w}}:= \mathbf e_a^\top G\mathbf w ,
\end{equation}
where $\mathbf e_a$ is the standard unit vector along $a$-th axis. 

\subsection{The anisotropic local law}

This subsection is devoted to the proof of the anisotropic local law \eqref{aniso_law_gen}. Our proof is an extension of the arguments in \cite[Section 7]{Anisotropic} and \cite[Section 6]{yang20190}. We will not give all the details, but mainly focus on the main differences from the previous arguments. 
The proof consists of a bootstrap argument from larger scales to smaller scales in multiplicative increments of $n^{-\delta}$, where $\delta$ is a positive constant satisfying
\begin{equation}
 \delta \in\left(0,\frac{\vartheta}{2C_a}\right). \label{assm_comp_delta}
\end{equation}
Here $\vartheta>0$ is the constant in the definition of $\cal D_\vartheta$ and $C_a> 0$ is an absolute constant that will be chosen large enough in the proof (for example, $C_a=100$ will work). For any $z=E+\ii \eta\in \cal D_\vartheta$, we define
\begin{equation}\label{eq_comp_eta}
\eta_l:=\eta n^{\delta l} \ \text{ for } \ l=0,...,L-1,\ \ \ \eta_L:=1.
\end{equation}
where
$L\equiv L(\eta):=\max\left\{l\in\mathbb N|\ \eta n^{\delta(l-1)}<1\right\}.$
Note that $L\le \delta^{-1}$ since $\eta\gg n^{-1}$ for $z=E+\ii \eta\in \cal D_\vartheta$. By (\ref{eq_gbound}), the function $z\mapsto G(z)- \Pi(z)$ is Lipschitz continuous in $\cal D_\vartheta$ with Lipschitz constant bounded by $n^2$. Thus to prove (\ref{aniso_law_gen}) for all $z\in \cal D_\vartheta$, it suffices to show that (\ref{aniso_law_gen}) holds for all $z$ in some discrete but sufficiently dense subset ${\mathbf S} \subset \cal D_\vartheta$. We will use the following discretized domain:
$\mathbf S$ is an $n^{-10}$-net of $\cal D_\vartheta$ such that $ |\mathbf S |\le n^{20}$ and
\[E+\ii\eta\in\mathbf S \ \Rightarrow \ E+\ii\eta_l\in\mathbf S\text{ for }l=1,...,L(\eta).\]

The bootstrapping is formulated in terms of two scale-dependent properties ($\bA_k$) and ($\bC_k$) defined on the subsets
\[\mathbf S_k:=\left\{z\in\mathbf S\mid\text{Im} \, z\ge n^{-\delta k}\right\}.\]
{\bf Property} ${(\bA_k)}$: For all $z\in\mathbf S_k$, any deterministic unit vectors $\mathbf x \in \mathbb R^{\mathcal I}$ and any $X$ satisfying the assumptions in Theorem \ref{thm_local_gen}, we have
\begin{equation}\label{eq_comp_Am}
 \im  G_{\mathbf x\mathbf x}(z) \prec \im \Pi_{\bx_1\bx_1}(z)+\im \Pi_{\bx_2\bx_2}(z) +n^{C_a\delta}\Phi(z) \|\mathbf x\|_\Pi (z),
\end{equation}
where $\bx_1\in \R^{\cal I_1}$ and $\bx_2\in \R^{\cal I_2}$ are defined such that 
\be\label{x12}\bx=\begin{pmatrix} \bx_1 \\ \bx_2\end{pmatrix}.\ee
{\bf Property} ${(\bC_k)}$: For all $z\in\mathbf S_k$, any deterministic unit vectors $\mathbf x, \mathbf y\in \mathbb R^{\mathcal I}$, and any $X$ satisfying the assumptions in Theorem \ref{thm_local_gen}, 
we have
\begin{equation}\label{eq_comp_Cm}
 \left|G_{\mathbf x\mathbf y}(z)-\Pi_{\mathbf x\mathbf y}(z)\right|\prec n^{C_a\delta}\Phi(z)\|\mathbf x\|_\Pi^{1/2}(z)\|\mathbf y\|_\Pi^{1/2}(z).
\end{equation}
It is easy to see that ${(\mathbf A_0)}$ holds by \eqref{eq_gbound} and $ \im \Pi_{\bx_1\bx_1}(z)+\im \Pi_{\bx_2\bx_2}(z)\sim 1$. Moreover, it is not hard to check that
\begin{equation}\label{lemm_boot2}
\text{for any $k$, property ${(\mathbf C_k)}$ implies property $(\mathbf A_k)$.}
\end{equation}
The key step is the following induction result.
\begin{lemma}\label{lemm_boot}
For any $1\le k\le \delta^{-1}$, property $(\mathbf A_{k-1})$ implies property $(\mathbf C_k)$.
\end{lemma}

Combining \eqref{lemm_boot2} and Lemma \ref{lemm_boot}, we conclude that (\ref{eq_comp_Cm}) holds for all $z\in\mathbf S$. Since $\delta$ can be chosen arbitrarily small under the condition (\ref{assm_comp_delta}), we conclude that (\ref{aniso_law_gen}) holds for all $z\in\mathbf S$, which further gives \eqref{aniso_law_gen} for all $z\in \cal D_\vartheta$. What remains now is the proof of Lemma \ref{lemm_boot}. For any deterministic unit vectors $\bu,\bv\in \R^{\cal I}$, we denote
\begin{equation}\label{eq_comp_F(X)}
 F_{\mathbf u\mathbf v}(X,z):=\left|G_{\mathbf{uv}}(X,z)-\Pi_{\mathbf {uv}}(z)\right|.
\end{equation}
By Markov's inequality, it suffices to prove the following lemma.
\begin{lemma}\label{lemm_comp_0}
 Suppose that the assumptions of Theorem \ref{thm_local_gen} and property $(\mathbf A_{k-1})$ hold. Then for any fixed $a\in 2\N$, we have that
 \begin{equation}
  \mathbb EF_{\mathbf u\mathbf v}^a(X,z)\lesssim \left[ n^{C_a\delta}\Phi(z)\|\mathbf u\|_\Pi^{1/2}(z)\|\mathbf v\|_\Pi^{1/2}(z) \right]^a
 \end{equation}
 for all $z\in{\mathbf S}_k$ and any deterministic unit vectors $\bu,\bv\in \R^{\cal I}$.
\end{lemma}
The rest of this subsection is devoted to proving Lemma \ref{lemm_comp_0}. 
First, in order to make use of the assumption $(\mathbf A_{k-1})$, which has spectral parameters in ${\mathbf S}_{k-1}$, to get some estimates for $G(z)$ with spectral parameters in ${\mathbf S}_{k}$, we shall use the following rough bounds on $ G_{\mathbf{xy}}$.

\begin{lemma}\label{lemm_comp_1}
 For any $z=E+\ii\eta\in\mathbf S$ and unit vectors $\mathbf x,\mathbf y\in \R^{\mathcal I}$,  we have 
\begin{align*}
\left|G_{\mathbf x\mathbf y}(z)-\Pi_{\mathbf x\mathbf y}(z)\right|\prec & n^{2\delta}\left(\sum_{l=1}^{L(\eta)} \sum_{\al=1}^2  \im   G_{\mathbf x_\al\mathbf x_\al}(E+\ii\eta_l) \right)^{1/2}\left(\sum_{l=1}^{L(\eta)}\sum_{\al=1}^2  \im  G_{\mathbf y_\al \mathbf y_\al}(E+\ii\eta_l) \right)^{1/2},
\end{align*}
where ${\mathbf x}_1,{\mathbf y}_1\in\mathbb R^{\mathcal I_1}$ and ${\mathbf x}_2,{\mathbf y}_2\in\mathbb R^{\mathcal I_2}$ are defined as in \eqref{x12}, and $\eta_l$ is defined in (\ref{eq_comp_eta}).
\end{lemma}
\begin{proof} The proof uses the spectral decomposition of $G$, and is the same as the one for \cite[Lemma 7.12]{Anisotropic}.\end{proof}

\begin{lemma}\label{lemm_comp_2}
Suppose $(\mathbf A_{k-1})$ holds. Then for any deterministic unit vector $\mathbf x,\mathbf y \in \mathbb R^{\mathcal I}$ and $z \in \mathbf S_k$, we have 
 \begin{equation}\label{eq_comp_apbound}
  |G_{\mathbf x\mathbf y}(z)-\Pi_{\mathbf x\mathbf y}(z)|\prec n^{2\delta} \|\mathbf x\|_\Pi^{1/2}(z)\|\mathbf y\|_\Pi^{1/2}(z),
 \end{equation}
 and
\begin{equation}\label{eq_comp_apbound2}
\im  G_{\mathbf x\mathbf x}(z) \prec n^{2\delta}\left[\im \Pi_{\bx_1\bx_1}(z)+\im \Pi_{\bx_2\bx_2}(z) +n^{C_a\delta}\Phi(z) \|\mathbf x\|_\Pi (z)\right].
\end{equation}
\end{lemma}
\begin{proof} 
This lemma is an easy consequence of Lemma \ref{lemm_comp_1}, together with the following property for Stieltjes transforms. Let
$$ s(z)=\int \frac{\dd\mu(z)}{x-z},\quad z\in \C_+,$$
be the Stieltjes transform of a measure $\mu$. Then for any fixed $E$,  $\eta^{-1}\im s(E+\ii \eta)$ is decreasing in $\eta$, while $\eta\im s(E+\ii \eta)$ is increasing in $\eta$. The reader can refer to  \cite[Lemma 7.13]{Anisotropic} for more details.
\end{proof}

The self-consistent comparison is performed using the following interpolating matrices between $X$, which satisfies the assumptions of Theorem \ref{thm_local_gen}, and the Gaussian matrix $X^{Gauss}$.
\begin{definition}[Interpolating matrices]
Introduce the notations $X^0:=X^{Gauss}$ and $X^1:=X$. Let $\rho_{i\mu}^0$ and $\rho_{i\mu}^1$ be the laws of $X_{i\mu}^0$ and $X_{i\mu}^1$, respectively. For $\theta\in [0,1]$, we define the interpolated law
$$\rho_{i\mu}^\theta := (1-\theta)\rho_{i\mu}^0+\theta\rho_{i\mu}^1.$$
We shall work on the probability space consisting of triples $(X^0,X^\theta, X^1)$ of independent $\mathcal I_1\times \mathcal I_2$ random matrices, where the matrix $X^\theta=(X_{i\mu}^\theta)$ has law
\begin{equation}\label{law_interpol}
\prod_{i\in \mathcal I_1}\prod_{\mu\in \mathcal I_2} \rho_{i\mu}^\theta(dX_{i\mu}^\theta).
\end{equation}
For $\lambda \in \mathbb R$, $i\in \mathcal I_1$ and $\mu\in \mathcal I_2$, we define the matrix $X_{(i\mu)}^{\theta,\lambda}$ through
\[\left(X_{(i\mu)}^{\theta,\lambda}\right)_{j\nu}:=\begin{cases}X_{i\mu}^{\theta}, &\text{ if }(j,\nu)\ne (i,\mu)\\ \lambda, &\text{ if }(j,\nu)=(i,\mu)\end{cases}.\]
We also introduce the matrices \[G^{\theta}(z):=G\left(X^{\theta},z\right),\ \ \ G^{\theta, \lambda}_{(i\mu)}(z):=G\left(X_{(i\mu)}^{\theta,\lambda},z\right).\]
\end{definition}


By Theorem \ref{thm_local}, we have that 
Lemma \ref{lemm_comp_0} holds if $X=X^0$.
Moreover, using (\ref{law_interpol}) and fundamental calculus, we can get the following basic interpolation formula.
\begin{lemma}\label{lemm_comp_3}
 For any differentiable function $F:\mathbb R^{\mathcal I_1 \times\mathcal I_2}\rightarrow \mathbb C$, we have that
\begin{equation}\label{basic_interp}
\frac{\dd}{\dd\theta}\mathbb E F(X^\theta)=\sum_{i\in\mathcal I_1}\sum_{\mu\in\mathcal I_2}\left[\mathbb E F\left(X^{\theta,X_{i\mu}^1}_{(i\mu)}\right)-\mathbb E F\left(X^{\theta,X_{i\mu}^0}_{(i\mu)}\right)\right]
\end{equation}
 provided all the expectations exist.
\end{lemma}

We shall apply Lemma \ref{lemm_comp_3} with $F(X)=F_{\mathbf u\mathbf v}^a(X,z)$ for $F_{\mathbf u\mathbf v}(X,z)$ defined in (\ref{eq_comp_F(X)}). The main work is devoted to proving the following self-consistent estimate for the right-hand side of (\ref{basic_interp}).

\begin{lemma}\label{lemm_comp_4}
 Suppose $(\mathbf A_{k-1})$ holds. Then for any fixed $a\in 2\N$, we have that 
 \begin{equation}
  \sum_{i\in\mathcal I_1}\sum_{\mu\in\mathcal I_2}\left[\mathbb EF_{\mathbf u\mathbf v}^a\left(X^{\theta,X_{i\mu}^1}_{(i\mu)},z\right)-\mathbb EF_{\mathbf u\mathbf v}^a\left(X^{\theta,X_{i\mu}^0}_{(i\mu)},z\right)\right]=
  \OO\left(\left[n^{C_a\delta}\Phi(z)\|\mathbf u\|_\Pi^{1/2}\|\mathbf v\|_\Pi^{1/2}\right]^a+\mathbb EF_{\mathbf u \mathbf v}^a(X^\theta,z)\right),
 \end{equation}
 for all $\theta\in[0,1]$, $z\in\mathbf S_k$ and any deterministic unit vectors $\mathbf u,\mathbf v\in \R^{\cal I}$.
\end{lemma}
Combining Lemma \ref{lemm_comp_3} and Lemma \ref{lemm_comp_4} with a Gr\"onwall's argument, we get Lemma \ref{lemm_comp_0}, which concludes the proof of \eqref{aniso_law_gen}. 
It remains to prove Lemma \ref{lemm_comp_4}. For this purpose, we compare $X^{\theta,X_{i\mu}^0}_{(i\mu)}$ and $X^{\theta,X_{i\mu}^1}_{(i\mu)}$ via a common $X^{\theta,0}_{(i\mu)}$, i.e. we will prove that
\begin{equation}\label{lemm_comp_5}
\sum_{i\in\mathcal I_1}\sum_{\mu\in\mathcal I_2}\left[\mathbb EF_{\mathbf u\mathbf v}^a\left(X^{\theta,X_{i\mu}^u}_{(i\mu)},z\right)-\mathbb EF_{\mathbf u\mathbf v}^a\left(X^{\theta,0}_{(i\mu)},z\right)\right]= \OO\left(\left[n^{C_a\delta}\Phi(z)\|\mathbf u\|_\Pi^{1/2}\|\mathbf v\|_\Pi^{1/2}\right]^a+\mathbb EF_{\mathbf {uv}}^a(X^\theta,z)\right),
 \end{equation}
for all $u\in \{0,1\}$, $\theta\in[0,1]$, $w\in \mathbf{S}_k$, and any deterministic unit vectors $\mathbf u,\mathbf v\in \R^{\cal I}$.

Underlying the proof of (\ref{lemm_comp_5}) is an expansion approach which we will describe below. 
During the proof, we always assume that $(\mathbf A_{k-1})$ holds. Also the rest of the proof is performed at a fixed $z\in \mathbf S_k$. We introduce the $\mathcal I \times \mathcal I$ matrix $\Delta_{(i\mu)}:=\mathbf e_i \mathbf e_\mu^{\top}+\mathbf e_\mu \mathbf e_i^{\top}$, whose $(i,\mu)$ and $(\mu,i)$-th entries are equal to 1, and all other entries are zero. Then we have the resolvent expansion
\begin{equation}\label{eq_comp_expansion}
\G_{(i \mu)}^{\theta,\lambda'} = G_{(i\mu)}^{\theta,\lambda}+\sum_{k=1}^{K}  (\lambda-\lambda')^k G_{(i\mu)}^{\theta,\lambda}\left( \Delta_{(i\mu)} G_{(i\mu)}^{\theta,\lambda}\right)^k+({\lambda-\lambda'})^{K+1} G_{(i\mu)}^{\theta,\lambda'}\left(\Delta_{(i\mu)} G_{(i\mu)}^{\theta,\lambda}\right)^{K+1}.
\end{equation}
With this expansion, we can obtain the following bound on the entries of $G_{(i\mu)}^{\theta,\lambda}$.
\begin{lemma}\label{lemm_comp_6}
 Suppose that $y$ is a random variable satisfying $|y|\prec t^{1/2}n^{-1/2}$. Then  for any deterministic unit vector $\mathbf x,\mathbf y \in \mathbb R^{\mathcal I}$, we have 
 \begin{equation}\label{comp_eq_apriori}
 \left|  \left(G_{(i\mu)}^{\theta,y}\right)_{\mathbf x\mathbf y}-\Pi_{\mathbf x\mathbf y}\right| \prec n^{2\delta} \|\mathbf x\|_\Pi^{1/2} \|\mathbf y\|_\Pi^{1/2},\quad i\in\sI_1, \ \mu\in\sI_2.
 \end{equation}
\end{lemma}
\begin{proof} The proof is based on the expansion \eqref{eq_comp_expansion} with $\lambda=X^\theta_{i\mu}$ and $\lambda'=y$, and Lemma \ref{lemm_comp_2} for $G^\theta$. 
In the proof, we need to use that for any deterministic unit vector $\bx\in \R^{\cal I}$, 
\be\label{xsmall} \|\mathbf x\|_\Pi(z) \lesssim \varpi^{-1}(z),\quad  \|\mathbf x\|_\Pi(z) \cdot |y|\prec \Phi(z)\lesssim n^{-\vartheta/2},\ee
by \eqref{rough Pi} and \eqref{boundpsi}. For more details of the proof, we refer the reader to \cite[Lemma 7.14]{Anisotropic}. \end{proof}

In the following proof, for simplicity of notations, we denote $f_{(i\mu)}(\lambda):=F_{\mathbf u\mathbf v}^a(X_{(i\mu)}^{\theta, \lambda})$. We shall use $f_{(i\mu)}^{(r)}$ to denote the $r$-th derivative of $f_{(i\mu)}$ with respect to $\lambda$. With Lemma \ref{lemm_comp_6} and (\ref{eq_comp_expansion}), it is easy to prove the following result.
\begin{lemma}\label{lemm_comp_6_add}
Suppose that $y$ is a random variable satisfying $|y|\prec t^{1/2}n^{-1/2}$. Then for any fixed $r\in\bbN$,
  \begin{equation}
  \left|f_{(i\mu)}^{(r)}(y)\right|\prec n^{2\delta(r+a)} \|\mathbf u\|_\Pi^{a/2} \|\mathbf v\|_\Pi^{a/2}\|\mathbf e_i\|_\Pi^{r/2} \|\mathbf e_\mu\|_\Pi^{r/2}.
 \end{equation}
\end{lemma}
Let $y$ be a random variable satisfying $|y|\prec t^{1/2}n^{-1/2}$. Using the Taylor expansion of $f_{(i\mu)}$ and \eqref{xsmall}, we get that
\begin{equation}\label{eq_comp_taylor}
f_{(i\mu)}(y)=\sum_{r=0}^{4a+4}\frac{y^r}{r!}f^{(r)}_{(i\mu)}(0)+\OO_\prec\left( \Phi^{a+4}\|\mathbf x\|_\Pi^{a/2} \|\mathbf y\|_\Pi^{a/2}\right),
\end{equation}
provided $C_a$ is chosen large enough in (\ref{assm_comp_delta}). Therefore we have for $u\in\{0,1\}$,
\begin{align*}
&\mathbb EF_{\mathbf u\mathbf v}^a\left(X^{\theta,X_{i\mu}^u}_{(i\mu)}\right)-\mathbb EF_{\mathbf u \mathbf v}^a\left(X^{\theta,0}_{(i\mu)}\right)=\bbE\left[f_{(i\mu)}\left(X_{i\mu}^u\right)-f_{(i\mu)}(0)\right]\\
&=\bbE f_{(i\mu)}(0)+\frac{t}{2n}\bbE f_{(i\mu)}^{(2)}(0)+\sum_{r=4}^{4a+4}\frac{t^{r/2}}{r!}\bbE f^{(r)}_{(i\mu)}(0)\bbE\left(X_{i\mu}^u\right)^r+\OO_\prec\left( \Phi^{a+4}\|\mathbf x\|_\Pi^{a/2} \|\mathbf y\|_\Pi^{a/2}\right),
\end{align*}
where we used that $X_{i\mu}^u$ has vanishing third moment. We remark that this is the only place where we need the condition \eqref{condition_3rd}. By \eqref{conditionA2}, we have that for any fixed $r\in \N$,
\be\label{moment-4}
\left|\bbE\left(X_{i\mu}^u\right)^r\right| \lesssim n^{-r/2}.
\ee
Thus to show (\ref{lemm_comp_5}), we only need to prove that for $r=4,5,...,4a+4$,
\begin{equation}\label{eq_comp_est}
\frac{t^{r/2}}{n^{r/2}}\sum_{i\in\mathcal I_1}\sum_{\mu\in\mathcal I_2}\left|\bbE f^{(r)}_{(i\mu)}(0)\right|=\OO\left(\left[n^{C_a\delta}\Phi(z)\|\mathbf u\|_\Pi^{1/2}\|\mathbf v\|_\Pi^{1/2}\right]^a+\mathbb EF_{\mathbf u \mathbf v}^a(X^\theta,z)\right).\end{equation}
In order to get an estimate in terms of the matrix $X^\theta$ only as on the right-hand side of (\ref{eq_comp_est}), we want to replace $X^{\theta,0}_{(i\mu)}$ in $f_{(i\mu)}(0)=F_{\mathbf u \mathbf v}^a(X_{(i\mu)}^{\theta, 0})$ with $X^\theta = X_{(i\mu)}^{\theta, X_{i\mu}^\theta}$ on the left-hand side. 
\begin{lemma}
Suppose that
\begin{equation}\label{eq_comp_selfest}
\frac{t^{r/2}}{n^{r/2}}\sum_{i\in\mathcal I_1}\sum_{\mu\in\mathcal I_2}\left|\bbE f^{(r)}_{(i\mu)}(X_{i\mu}^\theta)\right|=\OO\left(\left[n^{C_a\delta}\Phi(z)\|\mathbf u\|_\Pi^{1/2}\|\mathbf v\|_\Pi^{1/2}\right]^a+\mathbb EF_{\mathbf u \mathbf v}^a(X^\theta,z)\right)
\end{equation}
holds for $r=4,...,4a+4$. Then (\ref{eq_comp_est}) holds for $r=4,...,4a+4$.
\end{lemma}
\begin{proof}
The proof is the same as the one for \cite[Lemma 7.16]{Anisotropic}. 
\end{proof}


It remains to prove (\ref{eq_comp_selfest}). For simplicity, in the following proof we abbreviate $X^\theta \equiv X$. 
In order to exploit the detailed structure of the derivatives on the left-hand side of (\ref{eq_comp_selfest}), we introduce the following algebraic objects.

\begin{definition}[Words]\label{def_comp_words}
Given $i\in \mathcal I_1$ and $\mu\in \mathcal I_2$. Let $\sW$ be the set of words of even length in two letters $\{\mathbf i, \bm{\mu}\}$. We denote the length of a word $w\in\sW$ by $2\ell(w)$ with $\ell(w)\in \mathbb N$. We use bold symbols to denote the letters of words. For instance, $w=\mathbf t_1\mathbf s_2\mathbf t_2\mathbf s_3\cdots\mathbf t_r\mathbf s_{r+1}$ denotes a word of length $2r$.
Define $\sW_r:=\{w\in \mathcal W: \ell(w)=r\}$ to be the set of words of length $2r$, and such that
each word $w\in \sW_r$ satisfies that $\mathbf t_l\mathbf s_{l+1}\in\{\mathbf i\bm{\mu},\bm{\mu}\mathbf i\}$ for all $1\le l\le r$. Next we assign to each letter a value $[\cdot]$ through $[\mathbf i]:=i$ and $[\bm {\mu}]:=\mu,$ which are regarded as summation indices. 
Finally, to each word $w$ we assign a random variable $A_{\mathbf {uv}, i, \mu}(w)$ as follows. If $\ell(w)=0$ we define
 $$A_{\mathbf {uv}, i, \mu}(w):=G_{\mathbf u\mathbf v}-\Pi_{\mathbf u\mathbf v}.$$
 If $\ell(w)\ge 1$, say $w=\mathbf t_1\mathbf s_2\mathbf t_2\mathbf s_3\cdots\mathbf t_r\mathbf s_{r+1}$, we define
 \begin{equation}\label{eq_comp_A(W)}
 A_{\mathbf {uv}, i, \mu}(w):=G_{\bu[\mathbf t_1]} G_{[\mathbf s_2][\mathbf t_2]}\cdots G_{[\mathbf s_r][\mathbf t_r]} G_{[\mathbf s_{r+1}]\bv}.
 \end{equation}
\end{definition}

Notice the words are constructed such that, by (\ref{eq_comp_expansion}),  
\[\left(\frac{\partial}{\partial X_{i\mu}}\right)^r \left( G_{\mathbf u\mathbf v}-\Pi_{\mathbf u\mathbf v}\right)=(-1)^r r!\sum_{w\in \mathcal W_r} A_{\mathbf {uv}, i, \mu}(w), \]
with which we get that 
\begin{align*}
f^{(r)}_{(i\mu)}=(-1)^r & \sum_{\ell_1+\cdots+\ell_a=r}\prod_{s=1}^{a/2}\left(\ell_s! \ell_{s+a/2}!\right) \left(\sum_{w_s\in\sW_{\ell_s}}\sum_{w_{s+a/2}\in\sW_{\ell_{s+a/2}}}A_{\mathbf {uv}, i, \mu}(w_s)\overline{A_{\mathbf {uv}, i, \mu}(w_{s+a/2})}\right).
\end{align*}
Then to prove (\ref{eq_comp_selfest}), it suffices to show that
\begin{equation}
\frac{t^{r/2}}{n^{r/2}}\sum_{i\in\mathcal I_1}\sum_{\mu\in\mathcal I_2}\left|\bbE\prod_{s=1}^{a/2}A_{\mathbf {uv}, i, \mu}(w_s)\overline{A_{\mathbf {uv}, i, \mu}(w_{s+a/2})}\right|=\OO\left(\left[n^{C_a\delta}\Phi(z)\|\mathbf u\|_\Pi^{1/2}\|\mathbf v\|_\Pi^{1/2}\right]^a+\mathbb EF_{\mathbf u \mathbf v}^a(X,z)\right)\label{eq_comp_goal1}
\end{equation}
for  $4\le r\le 4a+4$ and all words $w_1,...,w_a\in \sW$ satisfying $\ell(w_1)+\cdots+\ell(w_a)=r$.
To avoid the unimportant notational complications associated with the complex conjugates, we will actually prove that
\begin{equation}\label{eq_comp_goal2}
\frac{t^{r/2}}{n^{r/2}}\sum_{i\in\mathcal I_1}\sum_{\mu\in\mathcal I_2}\left|\bbE\prod_{s=1}^{a}A_{\mathbf {uv}, i, \mu}(w_s)\right|=\OO\left(\left[n^{C_a\delta}\Phi(z)\|\mathbf u\|_\Pi^{1/2}\|\mathbf v\|_\Pi^{1/2}\right]^a+\mathbb EF_{\mathbf u \mathbf v}^a(X,z)\right).
\end{equation}
The proof of $(\ref{eq_comp_goal1})$ is essentially the same but with slightly heavier notations. Treating empty words separately, we find that it suffices to prove
\begin{equation}
\label{eq_comp_goal3}
\frac{t^{r/2}}{n^{r/2}}\sum_{i\in\mathcal I_1}\sum_{\mu\in\mathcal I_2}\bbE\left|A^{a-b}_{\mathbf {uv}, i, \mu}(w_0)\prod_{s=1}^{b}A_{\mathbf {uv}, i, \mu}(w_s)\right|=\OO\left(\left[n^{C_a\delta}\Phi(z)\|\mathbf u\|_\Pi^{1/2}\|\mathbf v\|_\Pi^{1/2}\right]^a+\mathbb EF_{\mathbf u \mathbf v}^a(X,z)\right)
\end{equation}
for  $4\le r\le 4a+4$, $1\le b \le a$, and words such that $\ell(w_0)=0$, $\sum_{s=1}^b \ell(w_s)=r$ and $\ell(w_s)\ge 1$ for $s\ge 1$.

Finally, it remains to prove estimate (\ref{eq_comp_goal3}). Using Lemma \ref{lemm_comp_2} and \eqref{xsmall}, we obtain the following bounds for any word $w$:
  \begin{equation}
  |A_{\mathbf {uv}, i, \mu}(w)|\prec n^{2\delta(\ell(w)+1)}  \|\mathbf e_i\|_\Pi^{\ell(w)/2}\|\mathbf e_\mu\|_\Pi^{\ell(w)/2}  \|\mathbf u\|_\Pi^{1/2}\|\mathbf v\|_\Pi^{1/2};\label{eq_comp_A1}
  \end{equation}
  for $\ell(w)=1$, we have 
  \begin{equation}
  |A_{\mathbf {uv}, i, \mu}(w)|\prec  |G_{\mathbf u i}||G_{\mathbf v \mu}|+|G_{\mathbf u \mu}||G_{\mathbf v i}|;\label{eq_comp_A3}
  \end{equation}
  for $\ell(w)\ge 2$, we have
  \begin{equation}
  \begin{split}
  |A_{\mathbf {uv}, i, \mu}(w)| &\prec n^{2\delta(\ell(w)-1)} \left[|G_{\mathbf u i}||G_{\mathbf v i}|\|\mathbf e_i\|_\Pi^{(\ell(w)-2)/2}\|\mathbf e_\mu\|_\Pi^{\ell(w)/2} +|G_{\mathbf u \mu}||G_{\mathbf v \mu}|\|\mathbf e_i\|_\Pi^{\ell(w)/2}\|\mathbf e_\mu\|_\Pi^{(\ell(w)-2)/2}\right] \\
   &+n^{2\delta(\ell(w)-1)} \left(|G_{\mathbf u i}||G_{\mathbf v \mu}|+|G_{\mathbf u \mu}||G_{\mathbf v i}|\right)\|\mathbf e_i\|_\Pi^{(\ell(w)-1)/2}\|\mathbf e_\mu\|_\Pi^{(\ell(w)-1)/2} ;\label{eq_comp_A2}
   \end{split}
  \end{equation}
   Define $\bu_1,\bv_1 \in\mathbb R^{\mathcal I_1}$ and $\bu_2,\bv_2 \in\mathbb R^{\mathcal I_2}$ as in \eqref{upi}. Then with Lemma \ref{Ward_id}, we can get that
\begin{align}
 &\frac{t}{n}\sum_{i\in\sI_1}|G_{\mathbf u i}|^2+ \frac{t}{n}\sum_{\mu\in\sI_2}|G_{\mathbf u \mu}|^2  \lesssim t\frac{\im  G_{\mathbf u_1\mathbf u_1}  + \im  G_{\mathbf u_2\mathbf u_2} +\eta \left|G_{\mathbf u_1\mathbf u_1}\right|+\eta \left|G_{\mathbf u_2\mathbf u_2}\right|}{n\eta} \nonumber\\
 & \prec  tn^{2\delta}\frac{ \im  \Pi_{\mathbf u_1\mathbf u_1}  + \im  \Pi_{\mathbf u_2\mathbf u_2}  + n^{C_a\delta} \Phi(z) \|\mathbf u\|_\Pi + \eta \|\mathbf u\|_\Pi }{n\eta} \nonumber\\
 &\prec n^{(C_a+2)\delta}\|\mathbf u\|_\Pi  \left( t^2 \Psi^2(z) \|\mathbf u\|_\Pi + \frac{t}{n}\|\mathbf u\|_\Pi  + \frac{t}{n\eta}\Phi\right) \lesssim n^{(C_a+2)\delta} \|\mathbf u\|_\Pi \cdot \varpi \Phi^2 ,\label{eq_comp_r2}
\end{align}
where in the second step we used the two bounds in Lemma \ref{lemm_comp_2}, in the third step we used
$$\im  \Pi_{\mathbf u_\al\mathbf u_\al} \lesssim \left(\eta+ t\im m_{w,t}\right)\|\mathbf u\|_\Pi^2 ,\quad \al=1,2, $$
by the definition of $\Pi$, and in the last step we used the definition of $\Phi$ in \eqref{defnPhi} and $\|\mathbf u\|_\Pi\lesssim \varpi^{-1}$ by \eqref{rough Pi}.
We have a similar estimate by replacing $\bu$ with $\bv$. Using \eqref{eq_dtheta1extension} and \eqref{eq_dtheta2bound2}, we can get that
\be\label{eiemu_12}
\frac tn\sum_{i\in \cal I_1}\|\mathbf e_i\|_\Pi^2+ \frac{t}n\sum_{\mu\in \cal I_2}\|\mathbf e_\mu\|_\Pi^2\lesssim \frac{t}{p}\sum_{i=1}^p\frac{1}{|d_i-\zeta_t|^2} \lesssim 1.
\ee

Now we consider the following cases for the left-hand side of (\ref{eq_comp_goal3}).

\vspace{5pt}

\noindent{\bf Case 1}: There exist at least two words $w_s$ with $\ell(w_s)=1$. Then we have $b\le r\le 2b -2$. Furthermore, using \eqref{eq_comp_A1}-\eqref{eq_comp_r2}, we can bound 
\begin{align}
&\frac{t^{r/2}}{n^{r/2}}\sum_{i\in\mathcal I_1}\sum_{\mu\in\mathcal I_2} \left|A^{a-b}_{\mathbf {uv}, i, \mu}(w_0)\prod_{s=1}^{b}A_{\mathbf {uv}, i, \mu}(w_s)\right| \nonumber\\
&\prec  n^{2\delta(r-4+b)}F_{\mathbf u\mathbf v}^{a-b}(X) \|\mathbf u\|_\Pi^{b/2-1}\|\mathbf v\|_\Pi^{b/2-1} \varpi^{-(r-2)}\frac{t^{r/2-2}}{n^{r/2-2}}\frac{t^2}{n^2}\sum_{i\in\mathcal I_1}\sum_{\mu\in\mathcal I_2} \left(|G_{\mathbf u i}|^2|G_{\mathbf v \mu}|^2+|G_{\mathbf u \mu}|^2|G_{\mathbf v i}|^2\right) \nonumber\\
&\prec n^{2\delta(r+b)}F_{\mathbf u\mathbf v}^{a-b}(X)  \|\mathbf u\|_\Pi^{b/2}\|\mathbf v\|_\Pi^{b/2} \varpi^{-(r-2)}\left( \varpi\Phi\right)^{r-4}\left( n^{C_a\delta} \varpi \Phi^2 \right)^2 \nonumber\\
& \le F_{\mathbf u\mathbf v}^{a-b}(X)   \|\mathbf u\|_\Pi^{b/2}\|\mathbf v\|_\Pi^{b/2} \left(n^{C_a\delta/2+4\delta}\Phi\right)^{r}\label{eq_comp_r1_add}
\end{align}
where we also used $t^{1/2}n^{-1/2}\le \varpi \Phi$ in the second step. If we take $C_a= 100$, then under (\ref{assm_comp_delta}) we have $n^{C_a\delta/2+4\delta}\Phi \ll 1$, so we can bound \eqref{eq_comp_r1_add} as
$$\eqref{eq_comp_r1_add}\prec  F_{\mathbf u\mathbf v}^{a-b}(X)  \left(n^{3C_a\delta/4} \|\mathbf u\|_\Pi^{1/2}\|\mathbf v\|_\Pi^{1/2} \Phi\right)^b.$$ 


\noindent{\bf Case 2}: There is at most one word $w_s$ with $\ell(w_s)=1$. Then we have $r\ge \max\{b,2b-1\}$. Furthermore, using \eqref{eq_comp_A1}-\eqref{eiemu_12}, we can bound 
\begin{align}
&\frac{t^{r/2}}{n^{r/2}}\sum_{i\in\mathcal I_1}\sum_{\mu\in\mathcal I_2} \left|A^{a-b}_{\mathbf {uv}, i, \mu}(w_0)\prod_{s=1}^{b}A_{\mathbf {uv}, i, \mu}(w_s)\right| \nonumber\\
&\prec  n^{2\delta(r-2+b)}F_{\mathbf u\mathbf v}^{a-b}(X) \|\mathbf u\|_\Pi^{b/2-1/2}\|\mathbf v\|_\Pi^{b/2-1/2} \varpi^{-(r-4)} \frac{t^{r/2-2}}{n^{r/2-2}} \nonumber\\
&\times \frac{t^2}{n^2}\sum_{i\in\mathcal I_1}\sum_{\mu\in\mathcal I_2} \left[|G_{\mathbf u i}||G_{\mathbf v i}|\|\mathbf e_\mu\|_\Pi^{2}  +|G_{\mathbf u \mu}||G_{\mathbf v \mu}| \|\mathbf e_i\|_\Pi^{2} + \left(|G_{\mathbf u i}||G_{\mathbf v \mu}|+|G_{\mathbf u \mu}||G_{\mathbf v i}|\right)\|\mathbf e_i\|_\Pi \|\mathbf e_\mu\|_\Pi \right] \nonumber\\
&\prec n^{2\delta(r+b)}F_{\mathbf u\mathbf v}^{a-b}(X)  \|\mathbf u\|_\Pi^{b/2}\|\mathbf v\|_\Pi^{b/2} \varpi^{-(r-4)}\left( \varpi\Phi\right)^{r-4} \left(n^{C_a\delta}  \varpi \Phi^2\right) \nonumber\\
& \le F_{\mathbf u\mathbf v}^{a-b}(X)   \|\mathbf u\|_\Pi^{b/2}\|\mathbf v\|_\Pi^{b/2} \left(n^{C_a\delta/2+8\delta}\Phi\right)^{r-2} . \label{eq_comp_r2_add}
\end{align}
If we take $C_a= 100$, then under (\ref{assm_comp_delta}) we have $n^{C_a\delta/2+8\delta}\Phi \ll 1$. Moreover, if $r\ge 4$ and $r\ge 2b-1$, then $r\ge b+2$. Thus we can bound \eqref{eq_comp_r2_add} as
$$\eqref{eq_comp_r2_add}\prec  F_{\mathbf u\mathbf v}^{a-b}(X)  \left(n^{3C_a\delta/4} \|\mathbf u\|_\Pi^{1/2}\|\mathbf v\|_\Pi^{1/2} \Phi\right)^b.$$ 

Combining the above two cases, we conclude that 
\begin{equation}\nonumber
\frac{t^{r/2}}{n^{r/2}}\sum_{i\in\mathcal I_1}\sum_{\mu\in\mathcal I_2}\bbE\left|A^{a-b}_{\mathbf {uv}, i, \mu}(w_0)\prod_{s=1}^{b}A_{\mathbf {uv}, i, \mu}(w_s)\right| \le \bbE F_{\bv}^{a-b}(X)\left[n^{C_a\delta} \|\mathbf u\|_\Pi^{1/2}\|\mathbf v\|_\Pi^{1/2}\Phi\right]^b.
\end{equation}
Now (\ref{eq_comp_goal3}) follows from H\"older's inequality. This concludes the proof of (\ref{eq_comp_selfest}), which further implies (\ref{lemm_comp_5}), which gives Lemma \ref{lemm_comp_4}, which then concludes the proof of Lemma \ref{lemm_comp_0}, which finally implies Lemma \ref{lemm_boot}. Hence we have concluded the proof of the anisotropic local law \eqref{aniso_law_gen}. 


\subsection{The averaged local law}

In this section, we prove the averaged local laws \eqref{averin_gen} and \eqref{averout_gen}. 
The proof is similar to that for \eqref{aniso_law_gen} in previous subsection, and we only explain the main differences. A great simplification is that the bootstrapping argument is not necessary, since we already have a good bound on the resolvent entries by \eqref{aniso_law_gen}.
In analogy to (\ref{eq_comp_F(X)}), we define
\begin{align*}
\wt F(X,z) : &= |m(z)-m_{w,t}(z)| =\left|\frac{1}{p}\sum_{i\in\sI_1} \left(G_{ii}(X,z)- \Pi_{ii}(z)\right)\right|.
\end{align*}
Moreover, by Theorem \ref{thm_local}, \eqref{averin_gen} and \eqref{averout_gen} hold for Gaussian $X$. 
For simplicity of notations, we denote 
$$ \phi (z):=\frac{1}{n\eta}\mathbf 1_{E\le \lambda_{+,t}}+ \left( \frac{1}{n(\kappa +\eta)} + \frac{1}{(n\eta)^2\sqrt{\kappa +\eta}}  + \frac{\Phi}{n\eta}\right)\mathbf 1_{E>\lambda_{+,t}},\quad \text{for} \ \ z=E+\ii \eta.$$
Then following the argument in previous subsection, analogous to (\ref{eq_comp_selfest}), we only need to prove that for any fixed $a\in 2\N$ and $r=4,...,4a+4$, 
\begin{equation}\label{eq_comp_selfestAvg}
\frac{t^{r/2}}{n^{r/2}}\sum_{i\in\mathcal I_1}\sum_{\mu\in\mathcal I_2}\left|\bbE \left(\frac{\partial}{\partial X_{i\mu}}\right)^r\wt F^a(X)\right|=\OO\left(\left(n^{\delta}\phi(z)\right)^a+\mathbb E\wt F^a(X)\right),
\end{equation}
for any small constant $\delta>0$. Then similar to (\ref{eq_comp_goal2}), it suffices to prove that
\begin{equation}\label{eq_comp_goalAvg}
\begin{split}
\frac{t^{r/2}}{n^{r/2}}\sum_{i\in\mathcal I_1}\sum_{\mu\in\mathcal I_2}\left|\bbE \left(\frac{1}{n}\sum_{j\in\sI_1}A_{ \mathbf e_j\mathbf e_j, i, \mu}(w_0)\right)^{a-b}\prod_{s=1}^{b}\left(\frac{1}{n}\sum_{j\in\sI_1}A_{ \mathbf e_j\mathbf e_j, i, \mu}(w_s)\right)\right|=\OO\left(\left(n^{\delta}\phi(z)\right)^a+\mathbb E\wt F^a(X)\right)
\end{split}
\end{equation}
for $4\le r\le 4a+4$, $1\le b \le a$, and words such that $\ell(w_0)=0$, $\sum_{s=1}^b \ell(w_s)=r$ and $\ell(w_s)\ge 1$ for $s\ge 1$.
\nc
Using \eqref{aniso_law_gen}, similar to \eqref{eq_comp_r2}, we can get that for any deterministic unit vector $\bu\in \R^{\cal I}$,
\begin{equation}\label{eq_comp_r22}
\begin{split}
\frac{1}{n}\sum_{i\in\sI_1}|G_{\mathbf u i}|^2+ \frac{1}{n}\sum_{\mu\in\sI_2}|G_{\mathbf u \mu}|^2 & \prec \|\mathbf u\|_\Pi  \left( \frac{t}{\varpi} \Psi^2(z)  + \frac{1}{n\varpi}   + \frac{\Phi}{n\eta}\right) \lesssim  \phi(z)\|\mathbf u\|_\Pi.
 \end{split}
\end{equation}
where in the second step we used \eqref{eq_defpsi}, \eqref{Immc} and the definition of $\varpi$ in \eqref{rough Pi}. Now using \eqref{aniso_law_gen} and \eqref{eq_comp_r22}, we can bound that for any $j\in \cal I_1$ and word $w$ with $\ell(w)\ge 1$, 
\begin{equation}\label{average_bound}
\left|\frac{1}{n}\sum_{j\in\sI_1}A_{ \mathbf e_j \mathbf e_j, i, \mu}(w)\right| \prec \|\mathbf e_i\|_\Pi^{\ell(w)/2}\|\mathbf e_i\|_\Pi^{\ell(w)/2} \phi(z). 
\end{equation}
With (\ref{average_bound}), for any $r\ge 4$, we can bound that
\begin{align*}
\frac{t^{r/2}}{n^{r/2}}\sum_{i\in\mathcal I_1}\sum_{\mu\in\mathcal I_2}\left|\prod_{s=1}^{b}\left(\frac{1}{n}\sum_{j\in\sI_1}A_{ \mathbf e_j\mathbf e_j, i, \mu}(w_s)\right)\right|  & \prec \frac{t^{r/2}}{n^{r/2}}\sum_{i\in\cal I_1}\sum_{\mu\in\cal I_2}\|\mathbf e_i\|_\Pi^{r/2}\|\mathbf e_i\|_\Pi^{r/2}\phi^b \prec \frac{t^{r/2-2}}{n^{r/2-2}}\varpi^{-(r-4)} \phi^b \le \phi^b ,
\end{align*}
where we used \eqref{eiemu_12} in the second step. Thus the left-hand side of (\ref{eq_comp_goalAvg}) can be bounded as
\begin{equation}\nonumber
\begin{split}
\frac{t^{r/2}}{n^{r/2}}\sum_{i\in\mathcal I_1}\sum_{\mu\in\mathcal I_2}\left|\bbE \left(\frac{1}{n}\sum_{j\in\sI_1}A_{ \mathbf e_j\mathbf e_j, i, \mu}(w_0)\right)^{a-b}\prod_{s=1}^{b}\left(\frac{1}{n}\sum_{j\in\sI_1}A_{ \mathbf e_j\mathbf e_j, i, \mu}(w_s)\right)\right| \prec \mathbb E\wt F^{a-b}(X) \phi^b.
\end{split}
\end{equation}
Applying Holder's inequality, we get \eqref{eq_comp_selfestAvg}, which completes the proof of \eqref{averin_gen} and \eqref{averout_gen}.


\appendix

\section{Matching properties for rectangular free convolutions}\label{sec comparison}

In this section, with the estimates proved in Section \ref{sec anafree}, we compare the edge behaviors of two free rectangular convolutions when their initial measures at $t=0$ satisfy certain matching properties. The estimates proved in this section will be used in \cite{DY20201} to study the evolution of the rectangular DBM. We also expect these estimates to be of independent interest from the point of view of free probability theory. We remark that similar matching properties for additive free convolution have been established in \cite[Section 7.3]{edgedbm}.

Let $t_0= n^{-1/3+\omega_0}$ for some constant $0<\omega_0<1/3$, and let $\psi>0$ be a fixed constant. We consider two probability measures $\rho_1$ and $\rho_2$ having densities on the interval $[0,2\psi],$ such that for some constant $c_\psi>0$, 
\begin{equation}\label{eq_closerho}
\rho_1(\psi-x)=\rho_2(\psi-x)\left(1+\OO\left( \frac{|x|}{t_0^2} \right) \right), \quad 0 \leq x \leq c_\psi t_0^2,
\end{equation}
and 
\be\label{sqrt12}
\rho_1(x)=\rho_2(x)=0 \ \ \text{on}\ \ [\psi,2\psi],\quad \rho_2(x) \sim \sqrt{\psi-x} \ \ \text{on}\ \ [\psi- c_\psi,\psi].
\ee 
Let $\rho_{1,t}$ and $\rho_{2,t}$ be the rectangular free convolutions of the MP law with $\rho_1$ and $\rho_2$, respectively. We denote the Stieltjes transform of $\rho_{i,t}$ by $m_{i,t}$. By \eqref{eq_keyequation11}, $m_{i,t}$ satisfies the equation
\begin{equation}\label{eq m12}
\frac{1}{c_n t}\left(1-\frac{1}{b_{i,t}}\right)=\int\frac{\rho_i(x)}{x-\zeta_{i}(z,t)}\dd x ,\quad i=1,2,
\end{equation}
where
\be\label{defnbit} b_{i,t}(z):=1+c_nt m_{i,t}(z),\quad  \zeta_{i}(z,t):=b_{i,t}^2 z-b_{i,t}t(1-c_n) . \ee
Moreover, as in Section \ref{sec anafree}, we introduce the notations $\xi_{i}(z,t):=\zeta_{i}(z,t) -\psi$ and $ \zeta_{i}(z,t)=\al_{i}(z,t)+\ii \beta_{i}(z,t)$. 


Throughout this section, we assume that $t\le t_0$ and $|E-\psi|+\eta \leq \tau t_0^2$ for some small constant $\tau>0$. Then with the square root behavior of $\rho_{i}$, $i=1,2,$ one can check that 
\begin{equation}\label{eq_partialzmcontrolcontrol000}
|m_i(z)|\lesssim 1,\quad \left| \partial_z m_i(z) \right| \leq  (|E-\psi|+\eta)^{-1/2}, \quad z=E+\ii \eta,
\end{equation}
and that \eqref{regular1} and \eqref{regular2} hold for $m_{i}(z)$ with $\eta_*=0$. Hence by replacing $\lambda_+$ with $\psi$, Lemma \ref{lem_immanalysis} can be applied to any $z$ with $\psi-3c_\psi/4 \le E \le \psi+3c_\psi/4$ and $0\le \eta\le 10$. Moreover, all the analysis in Section \ref{sec anafree} goes through for any $0< t\le t_0$. In particular, $\rho_{i,t}$ has a right edge at, say, $\lambda_{i,t}$, and we have 
\be\label{eq_edgeappd}\xi_{+,i}(t):=\xi_{i}(\lambda_{i,t},t) \sim t^2\ee 
by \eqref{eq_edgebound}. Moreover, for
$E \leq \lambda_{i,t},$ we have 
\begin{equation}\label{eq_contourproperty}
|\alpha_{i}-\xi_{i,t}| \sim |E-\lambda_{i,t}|, \quad \beta_{i} \sim t|E-\lambda_{i,t}|^{1/2},
\end{equation}    
by \eqref{alEsqrt} and \eqref{betasimal}. 
Using Lemma \ref{lem_partialm} and \eqref{eq_partialzmcontrolcontrol000}, we can obtain the following estimate.

\begin{lemma}
For $z$ satisfying $|E-\psi|+\eta \leq \tau t_0^2$, we have
\begin{equation}\label{eq_partialzmcontrolcontrol}
\left| \partial_z m_{i,t}(z) \right| \lesssim  (\kappa_i+\eta)^{-1/2} ,
\end{equation}
where we denoted $\kappa_i:= |E-\lambda_{i,t}|$. 
\end{lemma}
\begin{proof}
Taking the derivative of equation \eqref{eq m12} with respect to $z$, we get that 
\be\label{addpartialm}\left| \partial_z m_{i,t}(z) \right| \lesssim |\partial_{\zeta}m_{i}(\zeta_i)|\left|\partial_z \zeta_i\right| \lesssim \frac{|\Phi_i'(\zeta_i)|^{-1}}{\sqrt{|\al_i|+\beta_i}},\ee
 where we used \eqref{eq_partialzmcontrolcontrol000} and equation $\Phi_i(\zeta_i)=z$ by \eqref{simplePhizeta} with
$$ \Phi_i(\zeta):=\zeta(1-c_nt m_{i}(\zeta))^2+(1-c_n)t(1-c_nt m_{i}(\zeta)),\quad i=1,2.$$
If $\kappa_i + \eta \le \tau_1 t^2$ for some small contant $\tau_1>0$, we have $|\al_i - \xi_{+,i}|+\beta_i \sim t\sqrt{\kappa_i + \eta}$ by \eqref{eq_approximate}. Together with \eqref{Phit'}, we get from \eqref{addpartialm} that
$$\left| \partial_z m_{i,t}(z) \right| \lesssim |\partial_{\zeta}m_{i}(\zeta_i)|\left|\partial_z \zeta_i\right| \lesssim \frac{t}{\sqrt{|\al_i|+\beta_i} \sqrt{\kappa_i+\eta}} \lesssim \frac{1}{\sqrt{\kappa_i+\eta}},$$ 
where in the second step we used $|\al_i|\gtrsim t^2$ by \eqref{eq_edgeappd}. On the other hand, if $\kappa_i + \eta > \tau_1 t^2$, we have  $|\al_i|+\beta_i \sim |\al_i -  \xi_{+,i}|+\beta_i \sim \kappa_i + \eta$ by Lemma \ref{lem_xigeneral}. Again using \eqref{Phit'}, we get from \eqref{addpartialm} that
$$\left| \partial_z m_{i,t}(z) \right| \lesssim \frac{1}{\sqrt{\kappa_i+\eta}}.$$
This concludes the proof.
 \end{proof}
 Due to the matching condition \eqref{eq_closerho}, we can show that $\xi_{+,1}$ and $\xi_{+,2}$ are close to each other. 

\begin{lemma}\label{lem edgemoving}
We have that 
 \begin{equation}\label{eq_edgecomparebound}
| \xi_{+,1}- \xi_{+,2}| \lesssim \frac{ t}{t_0}t^2,
\end{equation}
and
 \begin{equation}\label{eq_edgecomparebound2}
|\lambda_{1,t}-\psi|+|\lambda_{2,t}-\psi| \lesssim t.
\end{equation}
 \end{lemma}
 \begin{proof}
 By Lemma \ref{lem_eigenvalueslarger}, $\zeta_{+,1}=\psi+\xi_{+,1}$ and $\zeta_{+,2}=\psi+\xi_{+,2}$ satisfy the equations $\Phi_1'(\zeta_{+,1})=\Phi_2'(\zeta_{+,2})=0,$ from which we can derive the following equation of $\zeta_{+,1}$ and $\zeta_{+,2}$:  
\begin{align}
0 &=c_nt(m_{2}(\zeta_{+,2})-m_{1}(\zeta_{+,1})) \left[2-c_nt(m_{2}(\zeta_{+,2})+m_{1}(\zeta_{+,1}))\right]+ c_n(1-c_n)t^2 (m'_{2}(\zeta_{+,2})-m'_{1}(\zeta_{+,1})) \tag{T1} \label{eq_t1} \\ 
& + 2c_nt \zeta_{+,2} (1-c_nt m_2(\zeta_{+,2})) m'_{2}(\zeta_{+,2})-2c_nt \zeta_{+,1}  (1-c_nt m_1(\zeta_{+,1}))  m'_{1}(\zeta_{+,1})   \tag{T2} \label{eq_t2}.
\end{align}
For the term (\ref{eq_t2}), we decompose it as
\begin{align*}
(\ref{eq_t2})&= 2c_nt (\zeta_{+,2}-\zeta_{+,1}) (1-c_nt m_2(\zeta_{+,2})) m'_{2}(\zeta_{+,2})+2c_n^2t^2 \zeta_{+,1} (m_{1}(\zeta_{+,1})-m_{2}(\zeta_{+,2}))m'_{2}(\zeta_{+,2}) \\
&+2c_nt \zeta_{+,1}(1-c_nt m_1(\zeta_{+,1}))(m'_{2}(\zeta_{+,2})-m'_{1}(\zeta_{+,1})) . 
\end{align*}
Hence from \eqref{eq_t1} and \eqref{eq_t2}, we obtain that
\begin{align}
&c_n t\left[ 2 \zeta_{+,1}(1-c_nt m_1(\zeta_{+,1})) + (1-c_n)t \right](m'_{2}(\zeta_{+,2})-m'_{1}(\zeta_{+,1}))= -2c_nt m'_{2}(\zeta_{+,2})(1-c_nt m_2(\zeta_{+,2}))  (\zeta_{+,2}-\zeta_{+,1}) \nonumber\\
&\qquad \qquad\qquad \qquad -c_nt(m_{2}(\zeta_{+,2})-m_{1}(\zeta_{+,1}))\left[2-c_n t(m_{2}(\zeta_{+,2})+m_{1}(\zeta_{+,1})) - 2c_nt \zeta_{+,1} m'_{2}(\zeta_{+,2})\right].\label{eq_decompose1}
\end{align}
By \eqref{eq_imasymptoics}, we have $b_{i,t}=1+\OO(t)$, which gives $m_{i}(\zeta_{+,i})=\OO(1)$ by \eqref{eq_keyequation11}. With this estimate and \eqref{eq_m'norm}, we can obtain from \eqref{eq_decompose1} the following bound
\begin{equation}\label{eq_ccccbegin}
\left|m'_{2}(\zeta_{+,2})-m'_{1}(\zeta_{+,1})\right| \lesssim t^{-1}|\zeta_{+,2}-\zeta_{+,1}| +  |m_{2}(\zeta_{+,2})-m_{1}(\zeta_{+,1})|\lesssim t^{-1}|\zeta_{+,2}-\zeta_{+,1}| + 1 .
\end{equation}
On the other hand, we have
\begin{align}\label{eq_decompose2}
m'_{2}(\zeta_{+,2})-m_{1}'(\zeta_{+,1})=(\xi_{+,1}-\xi_{+,2}) \int_0^\psi \frac{((2\psi + \xi_{+,1}+\xi_{+,2})-2x) \rho_1(x)\dd x}{(x-\xi_{+,1}-\psi)^2(x-\xi_{+,2}-\psi)^2}+  \int \frac{\rho_2(x)-\rho_1(x)}{(x-\xi_{+,2}-\psi)^2}\dd x .
\end{align} 
With (\ref{eq_closerho}), we can bound the second term on the RHS as
\begin{align*}
\left|\int \frac{\rho_2(x)-\rho_1(x)}{(x-\xi_{+,2}-\psi)^2}\dd x\right| \lesssim 1 + \frac1{t_0^2}\int_0^{c_\psi t_0^2} \frac{x^{3/2}\dd x}{(x + t^2)^2} +\int_{c_\psi t_0^2}^{c_\psi} \frac{\sqrt{x}\dd x}{(x + t^2)^2} \lesssim \frac{1}{t_0}.
\end{align*}
On the other hand, for the first term on the RHS of \eqref{eq_decompose1}, we have
\begin{align*}
\int_0^\psi \frac{((2\psi + \xi_{+,1}+\xi_{+,2})-2x) \rho_1(x)\dd x}{(x-\xi_{+,1}-\psi)^2(x-\xi_{+,2}-\psi)^2} \gtrsim t^2 \int_0^{c_\psi} \frac{\sqrt{x}\dd x}{(x+t^2)^4} \gtrsim t^{-3}.
\end{align*}
Plugging the above two estimates into \eqref{eq_ccccbegin}, we get \eqref{eq_edgecomparebound}.
The estimate \eqref{eq_edgecomparebound2} can be obtained using $\lambda_{1,t}-\lambda_{2,t}=\Phi_1(\zeta_{+,1})-\Phi_2(\zeta_{+,2})$ and \eqref{eq_edgecomparebound}.
\end{proof}

Next we show that $ \xi_{1}(\lambda_{1,t}-x)$ and $\xi_{2}(\lambda_{2,t}-x)$ match each other up to a negligible error.

\begin{lemma}\label{lem_keydensitycontrol} 
Let $0 < t \leq t_0 n^{-\epsilon_0}$ for a constant $\epsilon_0>0$. Suppose that $0 \leq x \leq \tau n^{-2\epsilon} t_0^2$ for some small enough constants $\tau,\e >0$. 
Then we have that for any fixed $D>0$,
\begin{equation}\label{beta-beta}
\left|\beta_{1}(\lambda_{1,t}-x)-\beta_{2}(\lambda_{2,t}-x) \right| \lesssim \left(\frac{n^{\epsilon} t }{t_0} + n^{-D}\right)t\sqrt{x},
\end{equation}
and
\begin{equation} \label{al-al}
\left| [\xi_{+,1}-\al_1 (\lambda_{1,t}-x)]-  [\xi_{+,2}-\al_{2}(\lambda_{2,t}-x) ] \right| \lesssim \left( \frac{n^{\epsilon}t}{t_0}+n^{-D}\right) x.
\end{equation}
\end{lemma}
\begin{proof}
We divide the proof into two cases according to the value of $x$. 
 
 \vspace{5pt}

\noindent{\bf {Case 1}:} Consider the case where $0 \leq x \leq n^{-\delta_0 }t^2$ for a sufficiently small constant $0<\delta_0<\epsilon_0/100$. We shall use the Taylor expansion (\ref{eq_taylorbasic}). We claim that
\begin{equation}\label{claimPhi1-2}
\left|\Phi_1^{(k)}(\xi_{+,1})-\Phi_2^{(k)}(\xi_{+,2}) \right| \lesssim \frac{t}{t_0 }\frac1{t^{2(k-1)}}.
\end{equation} 
In fact with (\ref{eq_derivativePhi}) and (\ref{eq_edgecomparebound}), we can get that 
\begin{equation}\label{claimPhi1-21}
\left|\Phi_1^{(k)}(\xi_{+,1})-\Phi_1^{(k)}(\xi_{+,2}) \right| \lesssim \frac{t}{t_0 }\frac1{t^{2(k-1)}}.
\end{equation} 
On the other hand, using \eqref{eq_closerho} we can get that for $k\ge 1$,
\begin{align*}
\left|m_1^{(k)}(\xi_{+,2})-m_2^{(k)}(\xi_{+,2})\right|\lesssim 1 + \frac1{t_0^2}\int_0^{c_\psi t_0^2} \frac{x^{3/2}\dd x}{(x + t^2)^{k+1}} +\int_{c_\psi t_0^2}^{c_\psi} \frac{\sqrt{x}\dd x}{(x + t^2)^{k+1}} \lesssim \frac{1}{t_0^2 t^{2k-3}} + \frac{1}{t_0^{2k-1}},
\end{align*}
which further implies that
\be\label{claimPhi1-22}
 \left|\Phi_1^{(k)}(\xi_{+,2})-\Phi_2^{(k)}(\xi_{+,2})\right|\lesssim t\left|m_1^{(k+1)}(\xi_{+,2})-m_2^{(k+1)}(\xi_{+,2})\right| \le \frac{1}{t_0 t^{2k-3}} .
 \ee
Combining \eqref{claimPhi1-21} and \eqref{claimPhi1-22}, we conclude \eqref{claimPhi1-2}.

By (\ref{eq_taylorbasic}) and \eqref{eq_derivativePhi}, we have that for any fixed $\ell \in \N$, 
\begin{equation}\label{E-lambda=}
E-\lambda_{i,t}=\sum_{k=2}^\ell \frac{\Phi_i^{(k)}(\xi_{+,i})}{k!}(\xi_i(E)-\xi_{+,i})^k+\OO(|\xi_i(E)-\xi_{+,i}|^{\ell+1} t^{-2\ell}).
\end{equation}
Using repeated back-substitution as in the proof of Lemma \ref{lem_edgebehavior} and \eqref{claimPhi1-2}, we obtain that for $0\le x \le n^{-\delta_0}t^2$,
\begin{equation}\label{eq_alphabound}
\xi_{+,1}-\alpha_1(\lambda_{1,t}-x)=(\xi_{+,2}-\alpha_2(\lambda_{2,t}-x))\left[1+\OO(t/t_0+n^{-D})\right],
\end{equation} 
and
\begin{equation}\label{eq_imsmallpart}
\beta_1(\lambda_{1,t}-x)=\beta_2(\lambda_{2,t}-x)\left[1+\OO(t/t_0+n^{-D})\right],
\end{equation}
for any fixed $D>0$, as long as we choose $\ell$ large enough depending on $\delta_0$ and $D$. With \eqref{eq_contourproperty}, we see that \eqref{eq_alphabound} implies \eqref{al-al}, while (\ref{eq_imsmallpart}) implies (\ref{beta-beta}). 

\vspace{5pt}

\noindent{\bf {Case 2}:} Then we consider the case where $n^{-\delta_0}t^2 \leq x \leq \tau n^{-2\epsilon} t_0^2.$ 
In this case, we can check with (\ref{eq_reducedform}) that $\dd \alpha_{1,2} /\dd E>0$ for small enough $\delta_0$. Hence we can parameterize $\beta_1=\beta_1(\alpha)$ and $\beta_2=\beta_2(\al)$ for $\al \in \{\al_1(\lambda_{1,t}-x): n^{-\delta_0}t^2 \leq x \leq \tau n^{-2\epsilon} t_0^2\}\cup \{\al_2(\lambda_{2,t}-x): n^{-\delta_0}t^2 \leq x \leq \tau n^{-2\epsilon} t_0^2\}$. Since both $(\alpha, \beta_1(\al))$ and $(\alpha, \beta_2(\al))$ satisfy equation (\ref{eq_deterministicalphabeta}), we obtain that 
\begin{align}
 \int \frac{2x\rho_1(x) \dd x}{(x-\alpha-\psi)^2+\beta_1^2}-\int \frac{2x \rho_2(x)\dd x}{(x-\alpha-\psi)^2+\beta_2^2} =c_nt \left( \frac{1-c_n}{c_n}R_1+R_2+R_3 \right), \label{eq_compareone}
\end{align}
where $R_i, i=1,2,3,$ are defined as 
\begin{align*}
& R_1=\frac{ \im m_{1}(\zeta_1)}{ \beta_1} -\frac{\im m_{2}(\zeta_2)}{\beta_2}, \quad  R_2=(\im m_{2}(\zeta_2))^2-(\im m_{1}(\zeta_1))^2, \\
& R_3=\re m_{1}(\zeta_1) \int \frac{(x+\alpha+\psi)\rho_1(x)\dd x}{(x-\alpha-\psi)^2+\beta_1^2}- \re m_{2}(\zeta_2)\int \frac{(x+\alpha+\psi)\rho_2(x)\dd x}{(x-\alpha-\psi)^2+\beta_2^2} .
\end{align*} 
Here we denoted $\zeta_i=\alpha+\psi+\ii \beta_i(\al)$, $i=1,2.$

First, for the left-hand side of (\ref{eq_compareone}), we have
\begin{align}
 &R_0:=\int \frac{2x\rho_1(x) \dd x}{(x-\alpha-\psi)^2+\beta_1^2}-\int \frac{2x \rho_2(x)\dd x}{(x-\alpha-\psi)^2+\beta_2^2} \nonumber \\
 & =(\beta_2^2-\beta_1^2)\int \frac{2x\rho_1(x)\dd x}{\left[ (x-\alpha-\psi)^2+\beta_1^2 \right] \left[ (x-\alpha-\psi)^2+\beta_2^2 \right]} +\int \frac{2x(\rho_1(x)-\rho_2(x))\dd  x}{(x-\alpha-\psi)^2+\beta_2^2}. \label{eq_keykeydecomposition}
\end{align}
Using (\ref{eq_closerho}) and \eqref{eq_contourproperty}, we can bound
\begin{equation}\label{eq_b1use}
\left|\int \frac{2x(\rho_1(x)-\rho_2(x))\dd  x}{(x-\alpha-\psi)^2+\beta_2^2}  \right| \lesssim \frac{1}{t_0}. 
\end{equation}
Then we bound the first term in \eqref{eq_keykeydecomposition}. By (\ref{eq_contourproperty}), we have that $|\alpha-\xi_{+,i}| \gtrsim n^{-\delta_0} t^2$, $i=1,2$. Moreover, since $\delta_0<\epsilon_0/100,$ we get from (\ref{eq_edgecomparebound}) that $|\alpha-\xi_{+,i}|  \gg |\xi_{+,1}-\xi_{+,2}|.$ Together with (\ref{eq_contourproperty}), we obtain that
\begin{equation*}
|\alpha-\xi_{+,1}| = |\alpha-\xi_{+,2}| (1+\oo(1)), \quad  \beta_1 \sim \beta_2.
\end{equation*} 
Then using Lemma \ref{lem_immanalysis} and \eqref{eq_contourproperty}, we get the lower bound
\be
\begin{split}\label{eq_b2use}
 & \int \frac{2x\rho_1(x)\dd x}{\left[ (x-\alpha-\psi)^2+\beta_1^2 \right] \left[ (x-\alpha-\psi)^2+\beta_2^2 \right]} \\
 &\gtrsim   \frac{\mathbf{1}(\alpha < 0)\sqrt{|\alpha|+t|\alpha-\xi_{+,1}|^{1/2}}}{t^3 |\alpha-\xi_{+,1}|^{3/2}}+\frac{\mathbf{1}(\al\ge 0)}{(|\alpha|+t|\alpha-\xi_{+,1}|^{1/2})^{5/2}} =:\cal E_l ,
\end{split}
\ee
and the upper bound
\be
\begin{split}\label{eq_b2upper}
 & \int \frac{2x\rho_1(x)\dd x}{\left[ (x-\alpha-\psi)^2+\beta_1^2 \right] \left[ (x-\alpha-\psi)^2+\beta_2^2 \right]}  \lesssim   \frac{\sqrt{|\alpha|+t |\alpha-\xi_{+,1}|^{1/2}}}{t^3 |\alpha-\xi_{+,1}|^{3/2}} =:\cal E_{u} .
\end{split}
\ee
Note that we have that
\be\label{alalal}
\mathbf 1(\al\le 0) |\al - \xi_{+,1}| \sim \mathbf 1(\al\le 0) \left(|\al| + t |\alpha-\xi_{+,1}|^{1/2} \right),
\ee
and 
\be\label{alalal222}
\mathbf 1(\al\ge 0) |\al- \xi_{+,1}|\gtrsim n^{-\delta_0}\mathbf 1(\al\ge 0)( |\al | + t^2),
\ee
using the estimates $|\al|\lesssim t^2$ and $|\alpha-\xi_{+,1}| \gtrsim t^2 n^{-\delta_0}$ when $\al\ge 0$. 
In particular, from \eqref{alalal222} we get that
\be\label{Elu} \cal E_{u} \lesssim n^{3\delta_0/2}\cal E_l.\ee



Then we control the right-hand side of (\ref{eq_compareone}). First $R_1$ can be bounded in the same way as \eqref{eq_keykeydecomposition}:
\begin{equation}\label{eq_addin0}  
t|R_1|=t\left| \int \frac{\dd \mu_{1,0}(x)}{(x-\alpha-\psi)^2+\beta_1^2}-\int \frac{ \dd \mu_{2,0}(x)}{(x-\alpha-\psi)^2+\beta_2^2}\right|\lesssim  \frac{t}{t_0}+ t \cal E_{u} |\beta_2^2-\beta_1^2|.
\end{equation}
For $R_2$, we can simply use \eqref{eq_partialzmcontrolcontrol000} to bound it as
\be\label{eq_addin} 
t|R_2| \lesssim t. 
\ee
Finally, for term $R_3$ we have
\begin{align}
t|R_3| &\lesssim  t |\re m_1(\zeta_1)- \re m_2(\zeta_2)| \int \frac{ \rho_1(x)\dd x}{(x-\alpha-\psi)^2+\beta_1^2} +  t|\re m_2(\zeta_2)|(|R_0| + |{R}_1|) \nonumber\\
&\lesssim  \frac{ \sqrt{|\alpha|+t |\alpha-\xi_{+,1}|^{1/2}}}{ |\alpha-\xi_{+,1}|^{1/2}}  +  t (|R_0| + |{R}_1|)  \lesssim n^{\delta_0/2} +  t (|R_0| + |{R}_1|) ,\label{eq_addin1} 
\end{align}
where we used \eqref{eq_boundnear} in the second step, and \eqref{alalal} and \eqref{alalal222} in the last step. 
Combining \eqref{eq_addin0}-\eqref{eq_addin1}, we get that 
\be\label{eq_addin3}
\left|ct(R_1+R_2+R_3)\right| \le n^{-2\delta_0}\cdot \left(  \cal E_{\mu} |\beta_2^2-\beta_1^2| + t_0^{-1}\right)
\ee
as long as $\delta_0$ is small enough. Plugging \eqref{eq_keykeydecomposition}-\eqref{eq_b2upper}, \eqref{Elu} and \eqref{eq_addin3} into \eqref{eq_compareone}, we obtain that
\begin{align*} 
& |\beta_2^2-\beta_1^2| \cal E_l \lesssim n^{-2\delta_0} \cal E_{\mu} |\beta_2^2-\beta_1^2| + t_0^{-1} \Rightarrow |\beta_2-\beta_1| \lesssim  \frac{1}{t_0 |\beta_2+\beta_1|\cal E_l}.
\end{align*}
Then using \eqref{eq_contourproperty}, \eqref{alalal} and \eqref{alalal222}, we can simplify the bound as
\begin{align}
 |\beta_2-\beta_1|&\lesssim \mathbf{1}_{ \{\alpha < 0 \}} \frac{t^2 |\alpha-\xi_{+,1}|}{t_0\sqrt{|\alpha|+t|\alpha-\xi_{+,1}|^{1/2}}}+\mathbf{1}_{\al\ge 0}\frac{(|\alpha|+t|\alpha-\xi_{+,1}|^{1/2})^{5/2}}{t_0 t |\alpha-\xi_{+,1}|^{1/2}} \nonumber\\
 & \lesssim \mathbf{1}_{ \{\alpha < 0 \}} \frac{t}{t_0} t|\alpha-\xi_{+,1}|^{1/2}+n^{2\delta_0}\mathbf{1}_{\al\ge 0} \frac{t}{t_0} t|\alpha-\xi_{+,1}|^{1/2}\le \frac{n^{2\delta_0}t}{t_0}\left( t|\alpha-\xi_{+,1}|^{1/2}\right).\label{eq_betabound}
\end{align}
This bound shows that $\beta_1$ and $\beta_2$ are close to each other if taking the same variable $\al$. 



Now fix an $n^{-\delta_0}t^2 \leq x \leq \tau n^{-2\epsilon} t_0^2.$ We choose $E_*\equiv E_*(x)$ such that $\alpha_2(E_*)=\alpha_1(\lambda_{1,t}-x)=:\al_*$, and write 
\begin{equation}\label{decomposebeta}
\beta_1(\lambda_{1,t}-x)-\beta_2(\lambda_{2,t}-x)=\left( \beta_1(\lambda_{1,t}-x)-\beta_2(E_{*}) \right)+\left( \beta_2(E_*)-\beta_2(\lambda_{2,t}-x) \right).
\end{equation} 
We have shown that $\beta_1(\lambda_{1,t}-x)-\beta_2(E_{*})$ is small in \eqref{eq_betabound}. For the second term on the right-hand side, we need to control $E_*-(\lambda_{2,t}-x).$ 
Let $E_i\equiv E_i(\al)$ be the inverse function of $\al_i\equiv \al_i(E)$, $i=1,2$. Then from \eqref{E-lambda=} we get that for $y \le \tau_1 n^{-\delta_0 }t^2$ for some small enough constant $\tau_1>0$,
\begin{equation*}
\lambda_{1,t}-E_1(\zeta_{+,1}-y)=(\lambda_{2,t}-E_{2}(\zeta_{+,2}-y))(1+\OO(t/t_0+n^{-D})). 
\end{equation*}
Next we show that $\dd E_1/\dd \al$ and $\dd E_2/\dd \al$ are close to each other for $\al$ that is at least $ t^2 n^{-\delta_0}$ away from $\xi_{1,t}$. More precisely, we shall prove the bound
\begin{equation}\label{eq_derivativebound2}
\left| \frac{\dd E_1}{\dd \alpha} -\frac{\dd E_2}{\dd \alpha}\right| \lesssim \frac{n^{4 \delta_0}t}{t_0},\quad \text{for} \quad \xi_{+,1}-\al \gtrsim  n^{-\delta_0} t^2.
\end{equation}
If (\ref{eq_derivativebound2}) holds, then we get that for any $y \lesssim n^{-2\epsilon} t_0^2$,
\begin{equation}\label{eq_inversebound}
(\lambda_{1,t}-E_{1}(\zeta_{+,1}-y))=(\lambda_{2,t}-E_{2}(\zeta_{+,2}-y))(1+\OO(n^{4 \delta_0}t/t_0+n^{-D})).
\end{equation}
Taking $\zeta_{+,1}-y=\al_*$, we get from \eqref{eq_inversebound} that
\begin{align}\label{eq_inversebound2}
E_* - (\lambda_{2,t} -x)=   E_2(\al_*)-E_{2}(\al_* + \zeta_{+,2}-\zeta_{+,1})+\OO(n^{4 \delta_0}xt/t_0+xn^{-D}).
\end{align}
Using \eqref{eq_reducedform} and the arguments below it, we can check that for small enough $\delta_0$, 
\be\label{rePhi2'} 
\begin{split}
 \re\Phi_2'(\zeta)& \sim \int \frac{4c_n t \beta^2 x \rho_2(x)}{\left[ (x-\alpha-\lambda_+)^2+\beta^2 \right]^2} \dd x \\
&\sim  \frac{\mathbf{1}(\alpha < 0)\sqrt{|\alpha|+t|\alpha-\xi_{+,2}|^{1/2}}}{ |\alpha-\xi_{+,2}|^{1/2}}+\frac{\mathbf{1}(\al\ge 0) t^3 |\alpha-\xi_{+,2}| }{(|\alpha|+t|\alpha-\xi_{+,2}|^{1/2})^{5/2}} \gtrsim n^{-\delta_0},
\end{split}
\ee
where in the last step we used similar estimates as in \eqref{alalal} and \eqref{alalal222} with $\xi_{+,1}$ replaced by $\xi_{+,2}$. Combining \eqref{rePhi2'} with \eqref{Phit'}, we get
$$\frac{\dd E_2}{\dd \al}=\frac{|\Phi_2'(\zeta)|^2}{\re\Phi_2'(\zeta)} \lesssim n^{\delta_0}.$$ 
Hence applying mean value theorem to \eqref{eq_inversebound2}, we get that for any fixed $D>0$,
\begin{align*}
|E_* - (\lambda_{2,t} -x)| \lesssim  n^{\delta_0} |\zeta_{+,2}-\zeta_{+,1}|+ n^{5 \delta_0}x\frac{t}{t_0}+xn^{-D} \lesssim x\left(n^{5 \delta_0}\frac{t}{t_0}+n^{-D}\right),
\end{align*}
where we used \eqref{eq_edgecomparebound} and $x\ge n^{-\delta_0}t^2$ in the second step. Combining this estimate with  (\ref{eq_partialzmcontrolcontrol}) and $|\partial_E \beta_2|\lesssim t |\partial_E m_{2,t}|$, we get that
\be\label{eq_bbb}\left| \beta_2(E_*)-\beta_2(\lambda_{2,t}-x)\right| \lesssim t\sqrt{x}\left(  n^{5 \delta_0}\frac{t}{t_0}+n^{-D}\right).\ee
Finally, using (\ref{eq_bbb}) and (\ref{eq_betabound}), we conclude from \eqref{decomposebeta} that (\ref{beta-beta}) holds for Case 2.

Next we show that \eqref{eq_derivativebound2} implies  \eqref{al-al}. We introduce the functions
\begin{equation*}
f_i(x)=\xi_{+,i}-\alpha_i(\lambda_{i,t}-x),\quad i=1,2.
\end{equation*}  
By \eqref{eq_alphabound}, we have already seen that
\begin{align}\label{eq_boundf1f2smaller} 
f_1(x) & =f_2(x) (1+ \OO(t/t_0+n^{-D})),\quad x\le n^{-\delta_0} t^2.
\end{align}
It remains to consider the case $x \geq n^{-\delta_0} t^2. $ Both $f_1$ and $f_2$ are bijections on some intervals, and their inverses satisfy  
\begin{equation}\label{eq_inverinverbound}
f_1^{-1}(y)=f_2^{-1}(y)\left(1+\OO(n^{5 \delta_0}t/t_0)\right)
\end{equation}
by (\ref{eq_derivativebound2}). Thus we can write 
\begin{equation}\label{f1-f2}
f_1(x)-f_2(x)=\frac{f_2(f_2^{-1}(f_1(x)))-f_2(f_1^{-1}(f_1(x)))}{f_2^{-1}(f_1(x))-f_1^{-1}(f_1(x))} \left[f_2^{-1}(f_1(x))-f_1^{-1}(f_1(x))\right].
\end{equation}
By \eqref{rePhi2'}, we have $\re\Phi_2'(\zeta) \lesssim 1$, and by \eqref{Phit'}, we have $|\Phi_2'(\zeta)| \gtrsim n^{-\delta_0/2}$ for $x \ge n^{-\delta_0} t^2$. Thus we get
$$\frac{\dd \al}{\dd E_2}=\frac{\re\Phi_2'(\zeta)}{|\Phi_2'(\zeta)|^2} \lesssim n^{\delta_0},$$ 
which gives that
$$\left|\frac{f_2(f_2^{-1}(f_1(x)))-f_2(f_1^{-1}(f_1(x)))}{f_2^{-1}(f_1(x))-f_1^{-1}(f_1(x))} \right| \lesssim n^{\delta_0}.$$
Plugging it into \eqref{f1-f2} and using \eqref{eq_inverinverbound}, we obtain that
\begin{equation}\label{eq_boundf1f2larger}
f_1(x)=f_2(x)(1+\OO(n^{6\delta_0} t/t_0)),
\end{equation}
which concludes \eqref{al-al} for Case 2.

It still remains to prove the estimate \eqref{eq_derivativebound2}. With equation \eqref{Cauchyderiv}, we obtain that 
\be\label{Cauchyderiv222} 
\frac{\dd E_1}{\dd \al}- \frac{\dd E_2}{\dd \al}= \frac{|\Phi_1'(\zeta_1)|^2}{\re \Phi_1'(\zeta)} - \frac{|\Phi_2'(\zeta_2)|^2}{\re \Phi_2'(\zeta)}.
\ee
To control \eqref{Cauchyderiv222}, we need to bound $|\Phi_1'(\zeta_1)- \Phi_2'(\zeta_2)|$:  
\begin{align}
\Phi_1'(\zeta_1)- \Phi_2'(\zeta_2) &= c_n t (m_2(\zeta_2)-m_1(\zeta_1))(2-c_nt m_1(\zeta_1)-c_nt m_2(\zeta_2)) + c_n(1-c_n)t^2 (m_2'(\zeta_2)-m_1'(\zeta_1)) \nonumber\\
&+ 2c_nt \left[\zeta_2 m_2'(\zeta_2) (1-c_nt m_2(\zeta_2))- \zeta_1 m_1'(\zeta_1) (1-c_nt m_1(\zeta_1))\right]=:\cal K_1 + \cal K_2 + \cal K_3.\label{dE-dE}
\end{align}
First, with \eqref{eq_partialzmcontrolcontrol000} we can bound
\be\label{calK1}
|\cal K_1|\lesssim t.
\ee
For the terms $\cal K_2$ and $\cal K_3$, we need to bound $|m_2'(\zeta_2)-m_1'(\zeta_1)|$:
\begin{align*}
m_2'(\zeta_2)-m_1'(\zeta_1)= \int \frac{(\rho_2(x)-\rho_1(x))\dd x}{(x-\zeta_2)^2}+ \int \left[\frac{1}{(x-\zeta_2)^2}-\frac{1}{(x-\zeta_1)^2}\right]\rho_1(x)\dd x=:P_1+P_2.
\end{align*} 
As in \eqref{eq_b1use}, we can bound $|P_1| \lesssim t_0^{-1}$.
For the term $P_2$, we write it as
\begin{align*}
P_2= \ii (\beta_2-\beta_1)\int  \frac{(x-\zeta_1)+(x-\zeta_2)}{(x-\zeta_1)^2(x-\zeta_2)^2} \rho_1(x)\dd x.
\end{align*}
When $\al\ge 0$, we use (\ref{eq_betabound}) and \eqref{eq_boundfarawayregion} to bound
\begin{align*}|P_2|&\lesssim \frac{n^{2\delta_0}t}{t_0}t|\al-\zeta_{+,1}|^{1/2} \int \left( \frac{1}{|x-\zeta_1| |x-\zeta_2|^2}  + \frac{1}{|x-\zeta_1|^2 |x-\zeta_2|} \right)\rho_1(x)\dd x \\
&\lesssim \frac{n^{2\delta_0}}{t_0}\frac{t^2|\al-\zeta_{+,1}|^{1/2}}{(|\al| + t |\al-\zeta_{+,1}|^{1/2})^{3/2}} \lesssim \frac{n^{2\delta_0}}{t_0},
\end{align*}
where in the last step we used $|\al|\lesssim t^2$ and $|\al| + t |\al-\zeta_{+,1}|^{1/2}\gtrsim t^2$ for $\al\ge 0$. If $\alpha<0,$ with (\ref{eq_boundnear}) we get that 
 \begin{align*}
|P_2| \lesssim \frac{n^{2\delta_0}t}{t_0}t|\al-\zeta_{+,1}|^{1/2} \frac{\sqrt{|\alpha|+t|\alpha-\xi_{+,1}|^{1/2}}}{t^2|\alpha-\xi_{+,1}|} \lesssim \frac{n^{2\delta_0}}{t_0},
\end{align*}
where in the last step we used $|\alpha-\xi_{+,1}| \gtrsim |\al| + t^2$ for $\al\le 0$. With similar arguments, we can show that 
$$t|m_2'(\zeta_2)|+t|m_1'(\zeta_1)|\lesssim 1.$$
Together with the estimates on $P_1$ and $P_2$, we conclude 
\be\nonumber |m_2'(\zeta_2)-m_1'(\zeta_1)| \lesssim \frac{n^{2\delta_0}}{t_0},\ee
which implies
\be\label{calK23}
|\cal K_2|+|\cal K_3| \lesssim t|m_2(\zeta_2)-m_1(\zeta_1)|+t|m_2'(\zeta_2)-m_1'(\zeta_1)| \lesssim\frac{n^{2\delta_0}t}{t_0}.
\ee
Together with \eqref{calK1}, we obtain that
$$|\Phi_1'(\zeta_1)- \Phi_2'(\zeta_2)| \lesssim\frac{n^{2\delta_0}t}{t_0}.$$
Inserting it into \eqref{Cauchyderiv222} and using \eqref{rePhi2'}, we can get \eqref{eq_derivativebound2}.
 \end{proof}

 With Lemma \ref{lem_keydensitycontrol}, it is easy to prove the matching properties for the real and imaginary parts of $m_{1,t}$ and $m_{2,t}$. 


\begin{lemma}\label{lem comparison1}
Under the assumptions of Lemma \ref{lem_keydensitycontrol}, we have that for any fixed $D>0$,
\begin{equation}\label{eq_densitycomparison}
\rho_{1,t}(\lambda_{1,t}-x)=\rho_{2,t}(\lambda_{2,t}-x)\left(1+\OO\left(\frac{n^{\epsilon}t}{t_0}+n^{-D}\right)\right) ,
\end{equation}
and
\begin{equation} \label{eq_realpartinside}
 \left| \Re [m_{1,t}(\lambda_{1,t}-x)-m_{1,t}(\lambda_{1,t})]-\Re [m_{2,t}(\lambda_{2,t}-x)-m_{2,t}(\lambda_{2,t})] \right| \lesssim \left( \frac{n^{\epsilon}}{t_0}+\frac{n^{-D}}{t}\right)x.
\end{equation}
\end{lemma}
\begin{proof}
With \eqref{defnbit}, for $E_i:=\lambda_{i,t}-x$, $i=1,2,$ we get that
\begin{align}\label{beta=Im}
\beta_i(E_i)=\left[2E_i \re b_{i,t}(E_i)  - (1-c_n)t\right]\cdot \pi c_nt \rho_{i,t}(E_i) .
\end{align}
Using \eqref{eq_imasymptoics} and \eqref{eq_edgecomparebound2}, we get
$$ \frac{2E_1 \re b_{1,t}(E_1)  - (1-c_n)t}{2E_2 \re b_{2,t}(E_2)  - (1-c_n)t} = 1+\OO(t).$$
Together with \eqref{beta-beta} and the estimate $\beta_i\sim t\sqrt{x}$, we conclude \eqref{eq_densitycomparison} from \eqref{beta=Im}. Moreover, by \eqref{betasimal}, \eqref{alEsqrt} and \eqref{beta=Im}, we have that 
\be\label{density1+2}
\rho_i(\lambda_{i,t}-x) \sim \sqrt{x},\quad i=1,2, \quad \text{for} \ \ 0\le x\le \tau,
\ee
as long as the constant $\tau>0$ is sufficiently small. Now for (\ref{eq_realpartinside}), with \eqref{defnbit} we get that
\begin{align}\label{eq_realexpansion}
\alpha_{i}(E)+\psi  =E\left[ (\re b_{i,t}(E))^2 - \pi^2 c_n^2 t^2 (\rho_{i,t}(E))^2\right] - (1-c_n)t \re b_{i,t}(E). 
\end{align}
Using \eqref{eq_realexpansion} and $\rho_{i,t}(\lambda_{i,t})=0$, we obtain that 
\begin{align}
\xi_{+,i}-\al_i (E_i) &=x\left[ (\re b_{i,t}(E_i))^2 - \pi^2 c_n^2 t^2 (\rho_{i,t}(E_i))^2\right] - c_n(1-c_n)t^2 \re \left[m_{i,t}(\lambda_{i,t}) -m_{i,t}(E_i) \right] \nonumber\\
&+ \lambda_{i,t} \left[   \re  \left(b_{i,t}(\lambda_{i,t}) + b_{i,t}(E_i) \right)\cdot c_n t\re \left(m_{i,t}(\lambda_{i,t}) -m_{i,t}(E_i) \right) + \pi^2 c_n^2 t^2 (\rho_{i,t}(E_i))^2\right] \nonumber\\
 &=x  +  \left[ \lambda_{i,t} \re  \left(b_{i,t}(\lambda_{i,t}) + b_{i,t}(E_i) \right) - (1-c_n)t  \right]c_nt\re \left(m_{i,t}(\lambda_{i,t}) -m_{i,t}(E_i) \right) + \OO(xt), \label{xi-alsimp}
\end{align}
where we used \eqref{eq_imasymptoics} and \eqref{density1+2} in the second step. Now with \eqref{alEsqrt}, we can obtain from \eqref{xi-alsimp} that
\be\label{redensity1+2}
\left| \re \left(m_{i,t}(\lambda_{i,t}) -m_{i,t}(E_i) \right) \right|\lesssim \frac{x}{t}.
\ee
Inserting it back into \eqref{xi-alsimp} and using \eqref{eq_edgecomparebound2}, we get
\begin{align}
\xi_{+,i}-\al_i (E_i) = x +   2E_1 c_n t\re \left(m_{i,t}(\lambda_{i,t}) -m_{i,t}(E_i) \right) + \OO(xt), \quad i=1,2.\label{xi-alsimp2}
\end{align}
Combining \eqref{xi-alsimp2} with \eqref{al-al}, we conclude \eqref{eq_realpartinside}. 
\end{proof}

\begin{lemma} \label{lem comparison2}
Suppose the assumptions of Lemma \ref{lem_keydensitycontrol} hold. If $ 0\le x \le \tau n^{-2\epsilon} t_0 t $, then we have that for any constant $D>0$,
\begin{equation}\label{eq_realpartoutside}
\left| \Re \ [m_{1,t}(\lambda_{1,t}+x)-m_{1,t}(\lambda_{1,t})]-\Re \ [m_{2,t}(\lambda_{2,t}+x)-m_{2,t}(\lambda_{2,t})] \right| \lesssim \left(n^{\epsilon} \frac{t^{1/2} }{t_0^{1/2}} + n^{-D}\frac{t_0^{1/2}}{t^{1/2} }\right)x^{1/2}.
\end{equation}
\end{lemma}
 \begin{proof}
We fix a scale $\eta \le \tau n^{-2\epsilon} t_0^2$, and estimate that
\begin{align*}
& \left| \Re \ [m_{1,t}(\lambda_{1,t}+x)-m_{1,t}(\lambda_{1,t})]-\Re \ [m_{2,t}(\lambda_{2,t}+x)-m_{2,t}(\lambda_{2,t})]  \right| \\
& \leq \left| \int_{E \geq \eta} \left( \frac{1}{E}-\frac{1}{E+x} \right) \rho_{1,t}(\lambda_{1,t}-E) \dd E \right|+\left| \int_{E \geq \eta}\left( \frac{1}{E}-\frac{1}{E+x} \right) \rho_{2,t}(\lambda_{1,t}-E) \dd E \right| \\
& + \left| \int_{0 \leq E \leq \eta} \left[\rho_{2,t}(\lambda_{2,t}-E)-\rho_{1,t}(\lambda_{1,t}-E)\right]\left(\frac{1}{E}-\frac{1}{E+x}\right)\dd E \right|=: A_1+A_2+A_3.
\end{align*}  
By the square root behavior of $\rho_{i,t}$ around $\lambda_{i,t}$, we can get the bound 
\begin{equation*}
|A_1|+|A_2| \lesssim \frac{x}{\eta^{1/2}}. 
\end{equation*}
One the other hand, using (\ref{eq_densitycomparison}) we can bound that
\begin{equation*}
|A_3| \lesssim \left(n^{\epsilon} \frac{t}{t_0} + n^{-D}\right) \eta^{1/2}.
\end{equation*} 
Finally, we can conclude the proof by choosing $\eta=x(t_0/t).$ 
\end{proof}  

\bibliographystyle{abbrv}
\bibliographystyle{plain}
\bibliography{references,aniso_bib}

\end{document}